\documentclass[a4paper,11pt]{amsart}

\usepackage{amsmath,amssymb,amsthm}
\usepackage{tikz,mathrsfs,fancyheadings}
\usepackage[colorlinks]{hyperref}

\newtheorem{theorem}{Theorem}[section]

\newtheorem{corollary}[theorem]{Corollary}
\newtheorem{lemma}[theorem]{Lemma}
\newtheorem{proposition}[theorem]{Proposition}

\theoremstyle{definition}
\newtheorem{definition}[theorem]{Definition}
\newtheorem{remark}[theorem]{Remark}
\newtheorem{example}[theorem]{Example}
\newtheorem{construction}[theorem]{Construction}
\newtheorem{observation}[theorem]{Observation}
\newtheorem{notation}[theorem]{Notation}

\DeclareMathOperator{\CM}{CM}
\DeclareMathOperator{\stabCM}{\underline{CM}}
\DeclareMathOperator{\Supp}{Supp}
\DeclareMathOperator{\thick}{thick}

\DeclareMathOperator{\pretr}{pretr}
\DeclareMathOperator{\rep}{rep}
\DeclareMathOperator{\eff}{eff}

\DeclareMathAlphabet{\mathpzc}{OT1}{pzc}{m}{it}
\newcommand{\HOM}{\mathcal{H}\mathpzc{om}}

\newcommand{\Ho}{{\rm H}}

\numberwithin{equation}{section}

\newcommand{\iso}{\cong}

\newcommand{\then}{\ensuremath{\Longrightarrow}}
\renewcommand{\iff}{\ensuremath{\Longleftrightarrow}}

\DeclareMathAlphabet{\mathpzc}{OT1}{pzc}{m}{it}

\DeclareMathOperator{\Mod}{Mod}
\DeclareMathOperator{\modules}{mod} \renewcommand{\mod}{\modules}
\DeclareMathOperator{\Proj}{Proj}
\DeclareMathOperator{\proj}{proj}

\DeclareMathOperator{\inj}{inj}

\DeclareMathOperator{\Add}{Add}
\DeclareMathOperator{\add}{add}

\DeclareMathOperator{\stabmod}{\underline{mod}}

\DeclareMathOperator{\costmod}{\overline{mod}}
\DeclareMathOperator{\perf}{perf}

\DeclareMathOperator{\End}{End}
\DeclareMathOperator{\Hom}{Hom}
\DeclareMathOperator{\Ext}{Ext}

\DeclareMathOperator{\Aut}{Aut}

\DeclareMathOperator{\Ob}{\mathpzc{Ob}}
\DeclareMathOperator{\Rad}{Rad}

\DeclareMathOperator{\Img}{Im} \renewcommand{\Im}{\Img}

\DeclareMathOperator{\gld}{gl.\!dim}

\DeclareMathOperator{\id}{id}

\DeclareMathOperator{\op}{op}

\newenvironment{smallpmatrix}
	       {\left( \! \begin{smallmatrix}}
	       {\end{smallmatrix} \! \right)}

\def\mathclap{\mathpalette\mathclapinternal}
\def\mathclapinternal#1#2{\clap{$\mathsurround=0pt#1{#2}$}}
\def\clap#1{\hbox to 0pt{\hss#1\hss}}

\newcommand{\leftsub}[2]{{\vphantom{#2}}_{#1}{#2}}

\usetikzlibrary{arrows,decorations.pathmorphing,decorations.pathreplacing,positioning}

\tikzset{>=stealth',
         vertex/.style={circle,draw=black,inner sep=1.5pt,outer sep=2pt},
         tvertex/.style={inner sep=1pt,font=\scriptsize},
         gap/.style={fill=white,inner sep=1pt}}

\newcommand{\arrow}[2][20]
 {
  \hspace{-5pt}
  \begin{tikzpicture}
   \node (A) at (0,0) {};
   \node (B) at (#1pt,0) {};
   \draw [#2] (A) -- (B);
  \end{tikzpicture}
  \hspace{-5pt}
 }

\newcommand{\arrowl}[3][20]
 {
  \hspace{-5pt}
  \begin{tikzpicture}
   \node (A) at (0,0) {};
   \node (B) at (#1pt,0) {};
   \draw [#2] (A) -- node [above] {$#3$} (B);
  \end{tikzpicture}
  \hspace{-5pt}
 }

\renewcommand{\to}[1][20]{\arrow[#1]{->}}
\newcommand{\tol}[2][20]{\arrowl[#1]{->}{#2}}

\newcommand{\epi}[1][20]{\arrow[#1]{->>}}

\newcommand{\mono}[1][20]{\arrow[#1]{>->}}

\newcommand{\sub}[1][20]{\arrow[#1]{right hook->}}

\renewcommand{\mapsto}[1][20]{\arrow[#1]{|->}}

\title{Stable categories of higher preprojective algebras}
\author{Osamu Iyama \and Steffen Oppermann}

\address{Osamu Iyama: Graduate School of Mathematics, Nagoya University, Chikusa-ku, Nagoya,
464-8602 Japan}
\email{iyama@math.nagoya-u.ac.jp}

\address{Steffen Oppermann: Institutt for matematiske fag, NTNU, 7491 Trondheim, Norway}
\email{steffen.oppermann@math.ntnu.no}

\newcommand{\psL}{}
\DeclareMathOperator{\TRUNCATE}{t}
\newcommand{\trunc}{\TRUNCATE}

\begin{document}

\begin{abstract}
We introduce $(n+1)$-preprojective algebras of algebras of global dimension $n$. We show that if an algebra is $n$-representation-finite then its $(n+1)$-preprojective algebra is self-injective. In this situation, we show that the stable module category of the $(n+1)$-preprojective algebra is $(n+1)$-Calabi-Yau, and, more precisely, it is the $(n+1)$-Amiot cluster category of the stable $n$-Auslander algebra of the original algebra. In particular this stable category contains an $(n+1)$-cluster tilting object. We show that even if the $(n+1)$-preprojective algebra is not self-injective, under certain assumptions (which are always satisfied for $n \in \{1,2\}$) the results above still hold for the stable category of Cohen-Macaulay modules.
\end{abstract}

\thanks{The first author was supported by JSPS Grant-in-Aid for Scientific Research 21740010}
\thanks{The second author was supported by NFR Storforsk grant no.\ 167130.}

\maketitle
\tableofcontents

\section{Introduction}

The preprojective algebras $\widetilde{\Lambda}$ of path algebras $\Lambda$ of quivers have been introduced by Gelfand and Ponomarev to understand representation theory of path algebras of quivers. They have played important roles in representation theory and other areas of mathematics, e.g.\ resolutions of Kleinian singularities, crystal bases of quantum groups, quiver varieties, and cluster theory. Preprojective algebras satisfy many important homological properties, in particular they give rise to 2-Calabi-Yau triangulated categories: For the Dynkin case the stable module category $\stabmod \widetilde{\Lambda}$ is $2$-Calabi-Yau \cite{MR1388043, MR1781930}, and for the non-Dynkin case the bounded derived category $\mathscr{D}^{\rm b}({\rm fd} \; \widetilde{\Lambda})$ of finite dimensional modules is 2-Calabi-Yau \cite{MR1781930}. Moreover Geiss-Leclerc-Schr{\"o}er \cite{GLS1, GLS2} constructed a $2$-cluster tilting object in the stable category $\stabmod \widetilde{\Lambda}$ for the Dynkin case. These properties play an important role in the categorification of Fomin-Zelevinsky cluster algebras \cite{FZ}. Also there is another important class of $2$-Calabi-Yau triangulated categories in cluster theory, namely cluster categories \cite{BMRRT} and their generalization given by Amiot \cite{C_PhD, CC}, and sometimes there exists a triangle equivalence between the stable category $\stabmod \widetilde{\Lambda}$ and a cluster category. Actually for each Dynkin case Amiot \cite{CC} constructed a triangle equivalence
\[ \stabmod \widetilde{\Lambda} \approx \mathscr{C}^2_{\Gamma}\]
with the $2$-Amiot cluster category $\mathscr{C}^2_{\Gamma}$ of the stable Auslander algebra $\Gamma$ of $\Lambda$. This gives another proof of the existence of a $2$-cluster tilting object in
$\stabmod \widetilde{\Lambda}$, since Amiot cluster categories contain cluster tilting objects \cite{CC}.

The aim of this paper is to introduce a higher analogue of preprojective algebras and generalize the properties of classical preprojective algebras discussed above to the higher case. In particular we compare them with higher (Amiot-)cluster categories \cite{BaurMarsh08,Thomas06,Zhu08,Guo}. For an algebra $\Lambda$ of global dimension at most $n$, its \emph{$(n+1)$-preprojective algebra} is defined as the tensor algebra
\[ \widetilde{\Lambda} = T_{\Lambda} \Ext_{\Lambda}^n(D\Lambda, \Lambda) := \bigoplus_{i \geqslant 0} \Ext_{\Lambda}^n(D\Lambda, \Lambda)^{\otimes_{\Lambda}^i}. \]
This is a generalization of Crawley-Boevey's description $\widetilde{\Lambda}=T_\Lambda\Ext^1_\Lambda(D\Lambda,\Lambda)$ of classical preprojective algebras \cite{MR1730657} and Assem-Br{\"u}stle-Schiffler's description of cluster tilted algebras \cite{ABS_trivext}. Our $(n+1)$-preprojective algebra $\widetilde{\Lambda}$ is the $0$-th homology of Keller's derived $(n+1)$-preprojective DG algebra ${\bf\Pi}$ \cite{Kel_DefCY} which plays an important role in the study of $n$-Amiot cluster categories \cite{CC,Guo}, and in particular $\widetilde{\Lambda}$ is isomorphic to the $n$-Calabi-Yau tilted algebra $\End_{\mathscr{C}^n_{\Lambda}}(\Lambda)$. While it is known that $\mathbf{\Pi}$ is always a bimodule $(n+1)$-Calabi-Yau DG algebra \cite{Kel_DefCY}, our $(n+1)$-preprojective algebra $\widetilde{\Lambda}$ is not so nice in general from a homological viewpoint. Therefore we have to restrict ourselves to certain nice classes of algebras $\Lambda$, which are called ``$n$-representation-finite algebras'' or, more generally, ``algebras satisfying the
vosnex property''.

The notion of $n$-representation-finite algebras has been introduced by the authors in \cite{IO} in the context of higher Auslander-Reiten theory. We define a finite dimensional algebra to be weakly $n$-representation-finite if there is an $n$-cluster tilting object in the module category, and we call it $n$-representation-finite if moreover it has global dimension at most $n$ (see Section~\ref{section.background}). Various papers have constructed classes of $n$-representation finite algebras \cite{ErdmannHolm09,HerI,HerschendIyama,HuangZhangA,HuangZhangB,HuangZhang09,Iy_Auslander_corr,I2,Iy_n-Auslander,IO}.
%In \cite{IO} we studied how $n$-representation-finite algebras can be tilted in a way preserving this property (we called this process $n$-APR tilting), and used this tilting mutation to obtain a family of $n$-representation-finite algebras we call ``type $A$''. In \cite{HerI}, Herschend and the first author investigate connections between $n$-representation-finiteness and fractional Calabi-Yau properties. Using this, they construct $n$-representation-finite algebras which come up as tensor products.
%In the first part of this paper we take a slightly different, less constructive approach: We assume we are given a basic $n$-representation-finite algebra $\Lambda$.
In this case the $n$-Amiot cluster category $\mathscr{C}^n_{\Lambda}$  is $\Hom$-finite, and the $(n+1)$-preprojective algebra
$\widetilde{\Lambda}$ is a finite dimensional algebra which, regarded as a $\Lambda$-module, gives the unique basic $n$-cluster tilting object in $\mod \Lambda$ \cite{Iy_n-Auslander} (see Theorem~\ref{theorem.tilde_is_ctmod}). One of our main results in this paper is that the $(n+1)$-preprojective algebra $\widetilde{\Lambda}$ of an $n$-representation finite algebra $\Lambda$ is self-injective (see Corollary~\ref{corollary.isselfinj}). As an application we give a complete classification of iterated tilted $2$-representation-finite algebras (see Theorem~\ref{theorem.tilted2repfin}) by using Ringel's classification of self-injective cluster tilted algebras \cite{Ringel_selfinj_cta}. Moreover we show that any quasi-tilted algebra of canonical type $(2,2,2,2)$ is $2$-representation-finite (see Proposition~\ref{proposition.canonical}). Thus, in contrast to the (classical) $1$-representation-finite situation, there is a $1$-parameter family of $2$-representation-finite algebras.

%When we have shown that the $(n+1)$-preprojective algebra $\widetilde{\Lambda}$ of an $n$-representation-finite algebra $\Lambda$ is self-injective, it is natural to ask what can be said about the stable module category $\stabmod \widetilde{\Lambda}$, which is a triangulated category.
For an $n$-representation-finite algebra $\Lambda$, we study the stable module category $\stabmod \widetilde{\Lambda}$ (and a derived category version of it, called $\stabmod \mathscr{U}\psL$ -- see Theorem~\ref{theorem.ct_in_Db}) in Section~\ref{sect.selfinj}, and obtain the following result. %of this stable module category in terms of the stable $n$-Auslander algebra
%\[ \Gamma = \underline{\End}_{\Lambda}(\widetilde{\Lambda}) \]
%of $\Lambda$ (see Definition~\ref{definition.Auslalg}).
\begin{theorem} \label{thm.intro1}
Let $\Lambda$ be an $n$-representation-finite algebra.
\begin{enumerate}
\item (Corollary~\ref{corollary.stabmod_CY}) The stable module category $\stabmod \widetilde{\Lambda}$ is an $(n+1)$-Calabi-Yau triangulated category.
\item (Theorem~\ref{theorem.H_equiv_selfinj}) We have a triangle equivalence
\[ \stabmod \widetilde{\Lambda} \approx \mathscr{C}_{\Gamma}^{n+1} \]
between the stable module category of $\widetilde{\Lambda}$ and the $(n+1)$-Amiot cluster category of $\Gamma$, where $\Gamma$ is the stable $n$-Auslander algebra $\underline{\End}_{\Lambda}(\widetilde{\Lambda})$  (see Definition~\ref{definition.Auslalg}).
\item (Corollary~\ref{cor.stab_is_Amiot}) $\stabmod \widetilde{\Lambda}$ contains an $(n+1)$-cluster tilting object.
\end{enumerate}
\end{theorem}
For the case $n=1$ we recover the properties of classical preprojective algebras explained above.

In the final Section of this paper we generalize the results above by weakening the assumption that $\Lambda$ is $n$-representation-finite. We introduce a homological condition called ``vosnex property'' (see \ref{notation.vosnex}) which generalizes $n$-representation-finiteness. By definition this condition is always satisfied in the cases $n=1$ or $2$. We generalize Theorem~\ref{thm.intro1} and the main result of Keller-Reiten \cite{KelRei}  for the case $n=2$ by showing the following.
\begin{theorem} \label{thm.intro_2}
Let $\Lambda$ be an algebra satisfying the vosnex property. 
\begin{enumerate}
\item (Lemma~\ref{lem.selinj_dim_1}, Corollary~\ref{cor.stabCM_is_CY}) $\widetilde{\Lambda}$ is Iwanaga-Gorenstein of dimension at most $1$, and the stable category of Cohen-Macaulay modules $\stabCM(\widetilde{\Lambda})$ over $\widetilde{\Lambda}$ is an $(n+1)$-Calabi-Yau triangulated category.
\item (Theorem~\ref{theorem.H_equiv_CM}) We have a triangle equivalence
\[ \stabCM(\widetilde{\Lambda}) \approx \mathscr{C}_{\Gamma}^{n+1}, \]
with the $(n+1)$-Amiot cluster category of $\Gamma$, where $\Gamma$ is the stable endomorphism algebra $\underline{\End}_{\Lambda}(\widetilde{\Lambda})$.
\item (Corollary~\ref{cor.CMunderline_cto}) $\stabCM(\widetilde{\Lambda})$ contains an $(n+1)$-cluster tilting object.
\end{enumerate}
\end{theorem}

Even for the case $n=2$, we have the following properties of $2$-Calabi-Yau tilted algebras which were not known before.

%
%\begin{corollary}[$n=1$] \label{cor.amiot=preproj}
%Let $\Lambda$ be a representation-finite hereditary algebra with stable Auslander algebra $\Gamma$ and preprojective algebra $\widetilde{\Lambda}$. Then we have a triangle equivalence
%\[ \stabmod \widetilde{\Lambda} \approx \mathscr{C}_{\Gamma}^2.\]
%In particular $\widetilde{\Lambda}$ is weakly 2-representation-finite.
%\end{corollary}

%Corollary~\ref{cor.amiot=preproj} was observed by Amiot \cite{CC} (see also \cite{GLS1}). Our new approach allows to generalize Corollary~\ref{cor.amiot=preproj} to Theorem~\ref{thm.1.1}.

\begin{corollary}[$n=2$]
Let $\Lambda$ be a finite dimensional algebra with $\gld \Lambda \leqslant 2$ such that $\mathscr{C}_{\Lambda}^2$ is $\Hom$-finite. Then we have a triangle equivalence
\[ \stabCM(\widetilde{\Lambda}) \approx \mathscr{C}_{\Gamma}^3, \]
where $\Gamma = \underline{\End}_{\Lambda}(\widetilde{\Lambda})$. In particular $\stabCM(\widetilde{\Lambda})$ contains a $3$-cluster tilting object.
\end{corollary}

It should be noted that all results of Section~\ref{sect.selfinj} are contained in the corresponding results of Section~\ref{sect.non-selfinj}. However we give separate proofs for most of them, since the proofs in the setup of Section~\ref{sect.selfinj} are considerably shorter and more explicit.

In Appendix~\ref{appendix} we give a detailed explanation of how we use the universal property of $n$-Amiot cluster categories.

\subsection*{Acknowledgements}

We would like to thank Claire Amiot and Bernhard Keller for helpfully explaining to us how to apply their results.

\section{Background and notation} \label{section.background}

We refer to \cite{ARS,ASS,R} for general background on representation theory of algebras, to \cite{Aus_CohF,Au,AR_dualR_I} for the theory of functor categories and to \cite{handbook_tilting,Ha} for tilting theory and derived categories.

Throughout this paper we fix a base field $k$, which is not necessarily algebraically closed. All algebras and categories are assumed to be $k$-algebras and $k$-categories, respectively. Moreover, all algebras are assumed to be basic and connected.

We say a category $\mathscr{T}$ is $\Hom$-finite if for any $X, Y \in \Ob \mathscr{T}$ the $k$-vector space $\Hom_{\mathscr{T}}(X,Y)$ is finite dimensional.

\begin{notation}
We use the notation $\Hom_{\mathscr{T}}(X,Y)$ for the morphisms in $\mathscr{T}$ from $X$ to $Y$, if $X, Y \in \Ob \mathscr{T}$, and, by abuse of notation, also for morphisms in $\mod \mathscr{T}$ from $X$ to $Y$, if $X$ and $Y$ are $\mathscr{T}$-modules. Moreover, in some diagrams (when there is little space) we omit the $\Hom$ and just write $\leftsub{\mathscr{T}}{(X,Y)}$. For $\mathscr{T}$-modules $X$ and $Y$ we denote by $\underline{\Hom}_{\mathscr{T}}(X,Y)$ the space of morphisms in the stable module category.
\end{notation}

\begin{definition}[\cite{BonKap}]
Let $\mathscr{T}$ be a $\Hom$-finite triangulated category. A \emph{Serre functor} $\leftsub{\mathscr{T}}{\mathbb{S}}$ is an auto-equivalence of $\mathscr{T}$, such that there is an isomorphism
\[ \Hom_{\mathscr{T}}(X,Y) \iso D \Hom_{\mathscr{T}}(Y, \leftsub{\mathscr{T}}{\mathbb{S}} X) \]
natural in $X$ and $Y$. If a Serre functor exists it is unique up to unique natural isomorphism.
\end{definition}

\begin{notation}
Let $\mathscr{T}$ be a triangulated category with a Serre functor $\leftsub{\mathscr{T}}{\mathbb{S}}$. As in \cite{Iy_n-Auslander, IO} we set
\[ \leftsub{\mathscr{T}}{\mathbb{S}}_n := \leftsub{\mathscr{T}}{\mathbb{S}}[-n] \colon \mathscr{T} \to[30] \mathscr{T}, \]
where $[1]$ denotes the suspension in $\mathscr{T}$. Whenever there is no danger of confusion we omit the left index $\mathscr{T}$.
\end{notation}

\begin{definition}[\cite{Kontsevich}]
Let $\mathscr{T}$ be a $\Hom$-finite triangulated category. We say $\mathscr{T}$ is \emph{$n$-Calabi-Yau} if $[n]$ is a Serre functor on $\mathscr{T}$. Equivalently we could ask $\mathscr{T}$ to have a Serre functor $\mathbb{S}$, and $\mathbb{S}_n \iso \id$.
\end{definition}

The two most important types of triangulated categories for this paper are the following:
\begin{itemize}
\item The bounded derived category of $\mod \Lambda$, for a finite dimensional algebra $\Lambda$, is denoted by $\mathscr{D}_{\Lambda}$, or by $\mathscr{D}$ if there is no danger of confusion.
\item The $n$-Amiot cluster category $\mathscr{C}_{\Lambda}^n$ of a finite dimensional algebra $\Lambda$ with $\gld \Lambda \leqslant n$ (see Subsection~\ref{subsect.amiot} below, and \cite{CC, C_PhD}).
\end{itemize}

\begin{notation}[\cite{I2}] \label{def.tau_n}
Let $\Lambda$ be a finite dimensional algebra. We denote by
\begin{align*}
\tau_n & = \tau \Omega^{n-1} \colon \stabmod \Lambda \to[30] \costmod \Lambda \text{, and} \\
\tau_n^- & = \tau^- \Omega^{-(n-1)} \colon \costmod \Lambda \to[30] \stabmod \Lambda
\end{align*}
the \emph{$n$-Auslander-Reiten translation} and \emph{inverse $n$-Auslander-Reiten translation}, respectively. Note that if $\Lambda$ is an algebra with $\gld \Lambda \leqslant n$, then $\tau_n^{\pm} = \Ho^0(\leftsub{\mathscr{D}_{\Lambda}}{\mathbb{S}}_n^{\pm 1} -)$.

We call an algebra $\Lambda$ with $\gld \Lambda \leqslant n$ \emph{$\tau_n$-finite} if $\tau_n^{-i} \Lambda = 0$ for sufficiently large $i$, or, equivalently, if $\tau_n^i D\Lambda = 0$ for sufficiently large $i$. The equivalence of these two conditions can be seen as follows:
\[ \tau_n^{-i} \Lambda=0 \iff \Hom_{\mathscr{D}_{\Lambda}}(\Lambda, \mathbb{S}_n^{-i} \Lambda) = 0 \iff \Hom_{\mathscr{D}_{\Lambda}}(\mathbb{S}_n^i D \Lambda, D\Lambda) = 0 \iff \tau_n^i(D\Lambda) = 0. \]
\end{notation}

\subsection{\texorpdfstring{$n$}{n}-Amiot cluster categories} \label{subsect.amiot}

For the convenience of the reader we recall the construction and the most important properties of $n$-Amiot cluster categories.

It should be noted that while Amiot \cite{CC, C_PhD} formulates the results we recall in this subsection only in case $n=2$, they immediately generalize to arbitrary $n$.

\begin{construction}[Amiot -- see \cite{CC, C_PhD}] \label{const.n-Amiot}
Let $\Lambda$ be an algebra of finite global dimension, such that $\Lambda / \Rad \Lambda$ is separable over $k$. We set $\Gamma$ to be the DG algebra $\Lambda \oplus D\Lambda[-n-1]$. Projection $\Gamma \to \Lambda$ yields a restriction functor $\mathscr{D}_{\Lambda} \to \mathscr{D}_{\Gamma}$. Now the $n$-Amiot cluster category is defined to be the quotient category
\[ \mathscr{C}^n_{\Lambda} = \thick_{\mathscr{D}_{\Gamma}}(\Lambda) / \perf \Gamma, \]
where $\thick_{\mathscr{D}_{\Gamma}}(\Lambda)$ denotes the smallest thick subcategory of $\mathscr{D}_{\Gamma}$ containing (the image of) $\Lambda$, and $\perf \Gamma = \thick_{\mathscr{D}_{\Gamma}}(\Gamma)$ denotes the perfect complexes over $\Gamma$.
\end{construction}

\begin{remark}
Whenever we use the $n$-Amiot cluster category of an algebra $\Lambda$, we assume $\Lambda / \Rad \Lambda$ to be separable over the base field $k$ (see \cite{CR,Waer}) -- that is $\Lambda / \Rad \Lambda$ is a product of matrix rings over skew fields, such that the centers of these skew fields are separable field extensions of $k$. See the assumptions at the beginning of Section~2 in \cite{CC} on why we need this.
\end{remark}

\begin{lemma}[Amiot -- see \cite{CC, C_PhD}]
The restriction functor induces a functor $\pi\colon \mathscr{D}_{\Lambda} \to \mathscr{C}^n_{\Lambda}$. This functor commutes with the Serre functors of the two categories.
\end{lemma}

\begin{lemma}[Amiot -- see \cite{CC, C_PhD}]
Let $\Lambda$ be a finite dimensional algebra with $\gld \Lambda \leqslant n$. The $n$-Amiot cluster category $\mathscr{C}_{\Lambda}^n$ is $\Hom$-finite if and only if $\Lambda$ is $\tau_n$-finite.
\end{lemma}

The $n$-Amiot cluster category is the ``algebraic $n$-Calabi-Yau'' version of the derived category, as indicated by the following theorem.

\begin{theorem}[[Amiot -- see \cite{CC, C_PhD}] \label{theorem.clairemain1}
Let $\Lambda$ be an algebra with $\gld \Lambda \leqslant n$, which is $\tau_n$-finite. Then the $n$-Amiot cluster category is $n$-Calabi-Yau, and satisfies a universal property. (See Appendix~\ref{appendix}, and in particular Theorem~\ref{thm.universal_stable} for details on the universal property.)
\end{theorem}

\subsection{\texorpdfstring{$(n+1)$}{(n+1)}-preprojective algebras}

\begin{definition}
Let $\Lambda$ be an algebra of global dimension at most $n$. Then the \emph{$(n+1)$-preprojective algebra} of $\Lambda$ is
\[ \widetilde{\Lambda} := T_{\Lambda} \Ext_{\Lambda}^n(D \Lambda, \Lambda), \]
that is the tensor algebra of the $\Lambda$-$\Lambda$-bimodule $\Ext_{\Lambda}^n(D \Lambda, \Lambda)$ over $\Lambda$.
\end{definition}

\begin{remark}
In \cite{Kel_DefCY} Keller introduced the notion of derived $(n+1)$-preprojective algebras. The $(n+1)$-preprojective algebras are the $0$-th homology of his derived $(n+1)$-preprojective algebras.
\end{remark}

One basic property of $(n+1)$-preprojective algebras is the following.

\begin{lemma}
Let $\Lambda$ be an algebra of global dimension at most $n$. Then
\[ \widetilde{\Lambda} \iso \bigoplus_{i \geqslant 0} \tau_n^- \Lambda \]
as $\Lambda$-modules.

In particular $\Lambda$ is $\tau_n$-finite if and only if $\widetilde{\Lambda}$ is finite dimensional. In that case we also have
\[ D \widetilde{\Lambda} \iso \bigoplus_{i \geqslant 0} \tau_n D \Lambda. \]
\end{lemma}

\begin{proof}
Since $\gld \Lambda \leqslant n$ we have
\[ \tau_n^- = \tau^- \Omega^{-(n-1)} = \Ext_{\Lambda}^1( D\Lambda, \Omega^{-(n-1)} -) = \Ext_{\Lambda}^n( D\Lambda, -) = \Ext_{\Lambda}^n( D\Lambda, \Lambda) \otimes_{\Lambda} - .\]
This implies the first isomorphism. The remaining claims follow.
\end{proof}

\subsection{\texorpdfstring{$n$}{n}-cluster tilting subcategories}

\begin{definition}[\cite{IyYo}] \label{def.n-cto}
Let $\mathscr{T}$ be a triangulated category. A full subcategory $\mathscr{U} \subseteq \mathscr{T}$ is called \emph{$n$-cluster tilting subcategory}, if it is functorially finite and
\begin{align*}
\mathscr{U} & = \{X \in \mathscr{T} \mid \Hom_{\mathscr{T}}(\mathscr{U}, X[i]) = 0 \, \forall i \in \{1, \ldots, n-1\} \} \\
& = \{X \in \mathscr{T} \mid \Hom_{\mathscr{T}}(X[-i], \mathscr{U}) = 0 \, \forall i \in \{1, \ldots, n-1\} \}.
\end{align*}
An object $X$ in $\mathscr{T}$ is called \emph{$n$-cluster tilting object}, if $\add X$ is an $n$-cluster tilting subcategory.

Similarly one defines $n$-cluster tilting subcategories and objects in abelian categories.
\end{definition}

We have the following basic property.

\begin{proposition}[{\cite[Corollary~3.3]{IyYo}}] \label{prop.ct-resolution}
Let $\mathscr{T}$ be a triangulated category, and let $\mathscr{U}$ be an $n$-cluster tilting subcategory of $\mathscr{T}$. Then for any $X_0 \in \mathscr{T}$, there exist triangles
\[ \begin{tikzpicture}[xscale=1.3,yscale=1.5]
 \node (X0) at (0,0) {$U_{n-1}$};
 \node (X1) at (2,0) {$X_{n-2}$};
 \node (X2) at (4,0) {$X_{n-3}$};
 \node (X8) at (7,0) {$X_1$};
 \node (X9) at (9,0) {$X_0$};
 \node (U0) at (1,1) {$U_{n-2}$};
 \node (U1) at (3,1) {$U_{n-3}$};
 \node (U7) at (6,1) {$U_1$};
 \node (U8) at (8,1) {$U_0$};
 \node at (5.5,0) {$\cdots$};
 \node at (4.5,1) {$\cdots$};
 \draw [->] (U0) -- (U1);
 \draw [->] (U7) -- (U8);
 \draw [->] (X0) -- (U0);
 \draw [->] (X1) -- (U1);
 \draw [->] (X8) -- (U8);
 \draw [->] (U0) -- (X1);
 \draw [->] (U1) -- (X2);
 \draw [->] (U7) -- (X8);
 \draw [->] (U8) -- (X9);
 \draw [->] (X1) -- node [pos=.15] {\tikz{\draw[-] (0,.15) -- (0,-.15);}} (X0);
 \draw [->] (X2) -- node [pos=.15] {\tikz{\draw[-] (0,.15) -- (0,-.15);}} (X1);
 \draw [->] (X9) -- node [pos=.15] {\tikz{\draw[-] (0,.15) -- (0,-.15);}} (X8);
\end{tikzpicture} \]
with $U_i \in \mathscr{U}$.
\end{proposition}

We will make use of the following examples of $n$-cluster tilting subcategories.

\begin{theorem}[{\cite[Theorem~1.22]{Iy_n-Auslander}}] \label{theorem.ct_in_Db}
Let $\Lambda$ be an algebra of global dimension at most $n$, which is $\tau_n$-finite. Then
\[ \mathscr{U}\psL = \add \{ \mathbb{S}_n^i \Lambda \mid i \in \mathbb{Z} \} \subseteq \mathscr{D}_{\Lambda} \]
is an $n$-cluster tilting subcategory of the derived category.
\end{theorem}

\begin{theorem}[Amiot -- see \cite{CC, C_PhD}] \label{theorem.clairemain2}
Let $\Lambda$ be an algebra of global dimension at most $n$, which is $\tau_n$-finite. Then $\pi \Lambda$ is an $n$-cluster tilting object in the $n$-Amiot cluster category $\mathscr{C}_{\Lambda}^n$. Moreover the endomorphism algebra $\End_{\mathscr{C}_{\Lambda}^n}(\pi \Lambda)$ of $\pi \Lambda$ is isomorphic to the $(n+1)$-preprojective algebra $\widetilde{\Lambda}$ of $\Lambda$.
\end{theorem}

\begin{remark}
Note that $\mathscr{U}\psL$ is the preimage of $\add \pi \Lambda$ under the functor from the derived category to the $n$-Amiot cluster category as indicated in the following diagram.
\[ \begin{tikzpicture}[xscale=3,yscale=-1.5]
 \node (A) at (0,0) {$\mathscr{U}\psL$};
 \node (B) at (1,0) {$\add \pi(\Lambda)$};
 \node (C) at (0,1) {$\mathscr{D}_{\Lambda}$};
 \node (D) at (1,1) {$\mathscr{C}_{\Lambda}^n$};
 \draw [->>] (A) -- (B);
 \draw [->] (C) -- node [above] {$\pi$} (D);
 \draw [right hook->] (A) -- (C);
 \draw [right hook->] (B) -- (D);
\end{tikzpicture} \]
\end{remark}

\subsection{\texorpdfstring{$n$}{n}-representation-finiteness}

\begin{definition}
\begin{enumerate}
\item A finite dimensional algebra $\Lambda$ is called \emph{weakly $n$-representation-finite} if $\mod \Lambda$ contains an $n$-cluster tilting object.
\item It is called  \emph{$n$-representation-finite} if moreover $\gld \Lambda \leqslant n$.
\end{enumerate}
\end{definition}

Note that an algebra is weakly $1$-representation-finite if and only if it is representation-finite, and $1$-representation-finite if and only if it is hereditary and representation-finite.

For examples of $n$-representation-finite algebras see \cite{HerI,IO} and Section~\ref{subsect.tilted}.

We have the following basic property of weakly $n$-representation-finite algebras.

\begin{proposition} \label{prop.tauderived}
Let $\Lambda$ be a weakly $n$-representation-finite algebra with an $n$-cluster tilting object $M$ in $\mod \Lambda$.
\begin{enumerate}
\item $\tau_n$ induces a bijection from isomorphism classes of indecomposable non-projective modules in $\add M$ to isomorphism classes of indecomposable non-injective modules in $\add M$.
\item If moreover $\gld \Lambda \leqslant n$ (that is, $\Lambda$ is $n$-representation-finite), then for any $X \in \add M$ without non-zero projective summands we have $\tau_n X \iso \mathbb{S}_n X$.
\end{enumerate}
\end{proposition}

\begin{proof}
\begin{enumerate}
\item \cite[Lemma~2.3]{I2}.
\item Since $\Lambda \in \add M$ we have $\Ext^i_{\Lambda}(X, \Lambda) = 0$ for $0 < i < n$. This implies $\Hom_{\Lambda}(X, \Lambda) = 0$ (see \cite[Lemma~2.3]{Iy_n-Auslander}). Thus $\tau_n X \iso \mathbb{S}_n X$. \qedhere
\end{enumerate}
\end{proof}

Using this, we have the following result.

\begin{theorem}[\cite{Iy_n-Auslander}] \label{theorem.tilde_is_ctmod}
Let $\Lambda$ be an $n$-representation-finite algebra. Then
\begin{enumerate}
\item $\widetilde{\Lambda}$ is the unique basic $n$-cluster tilting object in $\mod \Lambda$, and
\item for $\mathscr{U}$ as in Theorem~\ref{theorem.ct_in_Db} we have
\[ \mathscr{U}\psL = \add \{ \widetilde{\Lambda}[in] \mid i \in \mathbb{Z} \}. \]
\end{enumerate}
\end{theorem}

\begin{definition} \label{definition.Auslalg}
The \emph{$n$-Auslander algebra} and \emph{stable $n$-Auslander algebra} of an $n$-representation-finite algebra $\Lambda$ are defined as $\End_{\Lambda}(\widetilde{\Lambda})$ and $\underline{\End}_{\Lambda}(\widetilde{\Lambda})$, respectively.
\end{definition}

The following is the basic property of (stable) Auslander algebras.

\begin{theorem}[\cite{Iy_n-Auslander,Iy_Auslander_corr}] \label{theorem.Auslalg_gld}
Let $\Lambda$ be $n$-representation-finite. Then
\[ \gld \End_{\Lambda}(\widetilde{\Lambda}) \leqslant n+1, \quad \text{and} \quad \gld \underline{\End}_{\Lambda}(\widetilde{\Lambda}) \leqslant n+1. \]
\end{theorem}

\section{\texorpdfstring{$n$}{n}-representation-finiteness and self-injectivity}

\begin{theorem} \label{theorem.nrepfinderived}
Let $\Lambda$ be an algebra with $\gld \Lambda \leqslant n$. Then the following are equivalent:
\begin{enumerate}
\item $\Lambda$ is $n$-representation-finite,
\item $D\Lambda \in \mathscr{U}\psL$, and
\item $\mathscr{U}\psL = \mathbb{S} \mathscr{U}\psL$.
\end{enumerate}
\end{theorem}

For the proof we will need the following easy lemma (e.g.\ \cite[Proposition~5.4]{Iy_n-Auslander}).

\begin{lemma} \label{lemma.nuorders}
Let $\Lambda$ be an algebra with $\gld \Lambda \leqslant n$. Denote by $(\mathscr{D}^{\leqslant 0}, \mathscr{D}^{\geqslant 0})$ the standard t-structure of $\mathscr{D}$. Then we have $\mathbb{S}_n \mathscr{D}^{\geqslant 0} \subseteq \mathscr{D}^{\geqslant 0}$ and $\mathbb{S}_n^{-1} \mathscr{D}^{\leqslant 0} \subseteq \mathscr{D}^{\leqslant 0}$.
\end{lemma}

\begin{proof}[Proof of Theorem~\ref{theorem.nrepfinderived}]
(1) \then\ (3): We only have to show $\Lambda \in \mathbb{S} \mathscr{U}\psL$ and $D\Lambda \in \mathscr{U}\psL$. For any indecomposable projective $\Lambda$-module $P$ there exists an indecomposable injective $\Lambda$-module $I$ such that $\tau_n^\ell I \iso P$ for some $\ell \geqslant 0$ (Proposition~\ref{prop.tauderived}(1)). By Proposition~\ref{prop.tauderived}(2) we have $\mathbb{S}_n^\ell I \iso P$. Thus $P \in \add \mathbb{S}_n^{\ell} \mathbb{S} \Lambda \subseteq \mathbb{S} \mathscr{U}\psL$. Hence we have $\Lambda \in \mathbb{S} \mathscr{U}\psL$. Similarly we have $D \Lambda \in \mathscr{U}\psL$.

(3) \then\ (2): This is clear.

(2) \then\ (1): By assumption any indecomposable injective $\Lambda$-module $I$ is in $\mathscr{U}\psL$, and hence of the form $I = \mathbb{S}_n^{- \ell_I} P_I$ for some indecomposable projective module $P_I$. Hence, by Lemma~\ref{lemma.nuorders}, we have
\begin{align}
\notag \mathbb{S}_n^i I & \in \mathscr{D}^{\leqslant -n} \text{ for any } i < 0, \\
\mathbb{S}_n^i I & = \mathbb{S}_n^{i - \ell_I} P_I \in \mathscr{D}^{\geqslant 0} \cap \mathscr{D}^{\leqslant 0} = \mod \Lambda \text{ for } 0 \leqslant i \leqslant \ell_I \text{, and} \label{eq.location_S_n} \\
\notag \mathbb{S}_n^i I & = \mathbb{S}_n^{i - \ell_I} P_I \in \mathscr{D}^{\geqslant n} \text{ for any } i > \ell_I.
\end{align}
In particular we see that $\Lambda$ is $\tau_n$-finite. Set
\[ M = \bigoplus_{\substack{I \text{ indecomp.}\\\text{injective}}} \bigoplus_{i=0}^{\ell_I} \mathbb{S}_n^i I. \]
Then $\add M = \mathscr{U}\psL \cap \mod \Lambda$. By Theorem~\ref{theorem.ct_in_Db} we have $\Ext^i_{\Lambda}(M,M) = 0$ for any $0 < i < n$.

Assume that a $\Lambda$-module $X$ satisfies $\Ext^i_{\Lambda}(M,X) = 0$ for all $0 < i < n$. Then \eqref{eq.location_S_n}, together with $\gld \Lambda \leqslant n$, implies $\Hom_{\mathscr{D}}(\mathscr{U}\psL, X[i]) = 0$ for any $0 < i < n$. Hence, by Theorem~\ref{theorem.ct_in_Db}, we have $X \in \mathscr{U}\psL \cap \mod \Lambda = \add M$. Similarly any $\Lambda$-module $X$ satisfying $\Ext^i_{\Lambda}(X,M) = 0$ for all $0 < i < n$ belongs to $\add M$. Therefore $M$ is an $n$-cluster tilting $\Lambda$-module.
\end{proof}

\subsection{Self-injective \texorpdfstring{$(n+1)$}{(n+1)}-preprojective algebras}

\begin{notation}
As customary for algebras, we say that a category $\mathscr{U}$ is self-injective if projective and injective objects in $\mod \mathscr{U}$ coincide. (Here and throughout this paper $\mod \mathscr{U}$ denotes the category of finitely presented functors $\mathscr{U}^{\op} \to \mod k$.)
\end{notation}

We obtain the following immediate consequence of Theorem~\ref{theorem.nrepfinderived}.

\begin{corollary} \label{corollary.isselfinj}
Let $\Lambda$ be $n$-representation-finite. Then $\widetilde{\Lambda}$ ($ \iso \End_{\mathscr{C}^n_{\Lambda}}(\pi \Lambda)$) and $\mathscr{U}\psL$ are self-injective.
\end{corollary}

\begin{proof}
By Theorem~\ref{theorem.nrepfinderived}, $\mathscr{U}\psL$ is closed under $\mathbb{S}$. Hence $\pi(\add \Lambda) = \pi(\mathscr{U}\psL)$ is closed under $\mathbb{S}$ in $\mathscr{C}_{\Lambda}^n$. Using Serre duality we have
\[ D \widetilde{\Lambda} = D \End_{\mathscr{C}_{\Lambda}^n}(\pi \Lambda) \iso \Hom_{\mathscr{C}_{\Lambda}^n}(\pi \Lambda, \mathbb{S} \pi \Lambda) \in \add \End_{\mathscr{C}_{\Lambda}^n}(\pi(\Lambda)) = \add \widetilde{\Lambda}. \]
Thus $\widetilde{\Lambda}$ is self-injective. Similarly we have that $\mathscr{U}\psL$ is self-injective.
\end{proof}

We now turn the implication of this corollary into an equivalence by strengthening the second condition.

\begin{notation} \label{notation.vosnex}
We say an $n$-cluster tilting subcategory $\mathscr{U}$ in a triangulated category $\mathscr{T}$ has the ``vanishing of small negative extensions''-property (= \emph{vosnex}), if $\Hom_{\mathscr{T}}(\mathscr{U}[i], \mathscr{U}) = 0$ for all $i \in \{1, \ldots, n-2\}$.

By abuse of notation we say that an algebra $\Lambda$ has the vosnex property, if it is $\tau_n$-finite, and the category $\mathscr{U}\psL$ (as in Theorem~\ref{theorem.ct_in_Db}) has the vosnex property.
\end{notation}

\begin{proposition} \label{prop.selfinjeq}
Let $\mathscr{U}$ be an $n$-cluster tilting subcategory in a triangulated category $\mathscr{T}$. Then the following conditions are equivalent.
\begin{enumerate}
\item $\mathscr{U} = \mathbb{S} \mathscr{U}$.
\item $\mathscr{U} = \mathscr{U}[n]$.
\item $\Hom_{\mathscr{T}}(\mathscr{U}[i], \mathscr{U}) = 0 \, \forall i \in \{1, \ldots, n-1\}$.
\item $\mathscr{U}$ is self-injective, and has the vosnex property (see \ref{notation.vosnex}).
\end{enumerate}
\end{proposition}

Before we start the proof, let us note that what it means in the situation of derived and $n$-Amiot cluster categories.

\begin{corollary}
Let $\Lambda$ be an algebra of global dimension at most $n$. Then the following are equivalent.
\begin{enumerate}
\item $\Lambda$ is $n$-representation-finite.
\item $\mathscr{U}\psL = \leftsub{\mathscr{D}}{\mathbb{S}} \mathscr{U}\psL$.
\item $\add \pi \Lambda = \leftsub{\mathscr{C}_{\Lambda}^n}{\mathbb{S}} (\add \pi \Lambda)$.
\item $\mathscr{U}\psL$ is self-injective and has the vosnex property.
\item $\widetilde{\Lambda}$ is a self-injective algebra, and $\Lambda$ has the vosnex property.
\end{enumerate}
\end{corollary}

Finally note that for $n=2$ the vosnex property is always satisfied, so that we obtain the following.

\begin{corollary} \label{corollary.selfinj2repfin}
Let $\Lambda$ be an algebra of global dimension at most $2$. Then the following are equivalent.
\begin{enumerate}
\item $\Lambda$ is $2$-representation-finite.
\item $\mathscr{U}\psL = \leftsub{\mathscr{D}}{\mathbb{S}} \mathscr{U}\psL$.
\item $\add \pi \Lambda = \leftsub{\mathscr{C}_{\Lambda}^2}{\mathbb{S}} (\add \pi \Lambda)$.
\item $\mathscr{U}\psL$ is self-injective.
\item $\widetilde{\Lambda}$ is a self-injective algebra.
\end{enumerate}
\end{corollary}

For the proof of Proposition~\ref{prop.selfinjeq} we need the following preliminaries.

\begin{construction}
Let $\mathscr{T}$ be a triangulated category, and $\mathscr{A}$ and $\mathscr{B}$ two subcategories. Then $\mathscr{A} * \mathscr{B}$ denotes the full subcategory of $\mathscr{T}$, such that
\[ \Ob \mathscr{A} * \mathscr{B} = \{X \in \Ob \mathscr{T} \mid \exists \text{ a triangle } A \to[30] X \to[30] B \to[30] A[1] \text{ with } A \in \mathscr{A} \text{ and } B \in \mathscr{B} \}. \]
\end{construction}

\begin{lemma}[{\cite[Corollary~6.4]{IyYo}}] \label{lemma.factorismod}
Let $\mathscr{U}$ be an $n$-cluster tilting subcategory of a triangulated category $\mathscr{T}$. Set 
\[ \mathscr{Z} = \{ X \in \mathscr{T} \mid \Hom_{\mathscr{T}}(\mathscr{U},X[i])=0 \text{ for any } 0 < i \leqslant n-2\}. \]
Then we have $\mathscr{Z} = \mathscr{U} * \mathscr{U}[1]$, and there is an equivalence
\[ \Hom_{\mathscr{T}}(\mathscr{U},-) \colon \mathscr{Z} / (\mathscr{U}[1]) \to[30] \mod \mathscr{U}. \]
\end{lemma}

\begin{lemma} \label{lemma.equiv_vosnex}
Let $\mathscr{U}$ be an $n$-cluster tilting subcategory of a triangulated category $\mathscr{T}$. Then the following are equivalent:
\begin{enumerate}
\item $\mathscr{U}$ has the vosnex property,
\item $\mathbb{S} \mathscr{U} \subseteq \mathscr{U} * \mathscr{U}[1]$, and
\item $\mathbb{S}^{-1} \mathscr{U} \subseteq \mathscr{U}[-1] * \mathscr{U}$.
\end{enumerate}
\end{lemma}

\begin{proof}
We only show the equivalence (1) \iff\ (2), the equivalence (1) \iff\ (3) is dual. Let $\mathscr{Z} = \mathscr{U} * \mathscr{U}[1]$ as in Lemma~\ref{lemma.factorismod}. Then
\begin{align*}
& \mathbb{S} \mathscr{U} \subseteq \mathscr{Z} \\
\iff & \Hom_{\mathscr{T}}(\mathscr{U}, \mathbb{S} \mathscr{U}[i]) = 0 \quad \forall i \in \{1, \ldots, n-2\} && \text{(by Lemma~\ref{lemma.factorismod})} \\
\iff & \Hom_{\mathscr{T}}(\mathscr{U}[i], \mathscr{U}) = 0 \quad \forall i \in \{1, \ldots, n-2\} && \qedhere
\end{align*}
\end{proof}

Using this, we prove the desired proposition.

\begin{proof}[Proof of Proposition~\ref{prop.selfinjeq}]
Since $\mathbb{S} \mathscr{U}=\mathscr{U}[n]$, conditions (1) and (2) are clearly equivalent.

Next we show (2) \iff\ (3) in the following sequence of equivalent conditions.
\begin{align*}
\text{(3)} & \iff \Hom_{\mathscr{D}}(\mathscr{U}[n], \mathscr{U}[n-i]) = 0 \quad \forall i \in \{1, \ldots, n-1\} \\
& \iff \mathscr{U}[n] \subseteq \mathscr{U} \qquad \qquad \text{since $\mathscr{U}$ is $n$-cluster tilting} \\
& \iff \text{(2)} \hspace{2.8cm} \text{since $\mathscr{U}$ and $\mathscr{U}[n]$ are both $n$-cluster tilting}
\end{align*}

For the equivalence (1) \iff\ (4) first note that we have seen above that (1) \then\ (3), and clearly (3) \then\ $\mathscr{U}$ vosnex. So (1) and (4) both imply $\mathscr{U}$ to have the vosnex property. Hence, by Lemma~\ref{lemma.equiv_vosnex} we know that $\mathbb{S} \mathscr{U} \in \mathscr{U} * \mathscr{U}[1]$. Now
\begin{align*}
& \text{all injective $\mathscr{U}$-modules are projective} \\
& \iff \forall U \in \mathscr{U} \exists U' \in \mathscr{U} \colon \underbrace{D \Hom_{\mathscr{T}}(U, -)}_{\Hom_{\mathscr{T}}(-, \mathbb{S} U)} \iso \Hom_{\mathscr{T}}(-, U') \text{ in } \mod \mathscr{U} \\
& \iff \forall U \in \mathscr{U} \exists U' \in \mathscr{U} \colon \mathbb{S} U = U' \qquad \qquad \text{(by Lemma~\ref{lemma.factorismod})} \\
& \iff \mathbb{S} \mathscr{U} \subseteq \mathscr{U}
\end{align*}
Similarly all projective $\mathscr{U}$-modules are injective if and only if $\mathscr{U} \subseteq \mathbb{S} \mathscr{U}$.

Summing up we obtain the equivalence (1) \iff\ (4).
\end{proof}

\subsection{The \texorpdfstring{$2$}{2}-representation-finite iterated tilted algebras} \label{subsect.tilted}

As an application of Corollary~\ref{corollary.selfinj2repfin}, we classify the iterated tilted algebras which are $2$-representation-finite.

\begin{theorem} \label{theorem.tilted2repfin}
Let $\Lambda$ be an iterated tilted algebra. Then $\Lambda$ is $2$-representation-finite if and only if $\Lambda \iso \widetilde{k} Q / I$, where $\widetilde{k}$ is some finite dimensional skew field over $k$, and $Q$ and $I$ are one of the following:
\begin{enumerate}
\item $Q = \underbrace{\circ \tol[30]{\scriptstyle \alpha_1} \circ \tol[30]{\scriptstyle \alpha_2} \cdots \tol[30]{\scriptstyle \alpha_{n-2}} \circ \tol[30]{\scriptstyle \alpha_{n-1}} \circ}_{n \text{ vertices}}$ and $I = (\alpha_1 \cdots \alpha_{n-1})$.
\item $Q$ is a full subquiver of the following quiver, obtained by choosing for any $i \in \{1, \ldots, n-1\}$ exactly one of the arrows $\gamma_i$ or $\delta_i$ (and all other arrows).
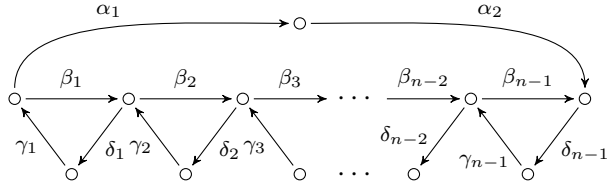
\[ \begin{tikzpicture}[xscale=1.5]
 \node (t) at (2.5,1) [vertex] {};
 \node (1) at (0,0) [vertex] {};
 \node (2) at (1,0) [vertex] {};
 \node (3) at (2,0) [vertex] {};
 \node (4) at (3,0) {$\cdots$};
 \node (5) at (4,0) [vertex] {};
 \node (6) at (5,0) [vertex] {};
 \node (b1) at (.5,-1) [vertex] {};
 \node (b2) at (1.5,-1) [vertex] {};
 \node (b3) at (2.5,-1) [vertex] {};
 \node (b4) at (3,-1) {$\cdots$};
 \node (b5) at (3.5,-1) [vertex] {};
 \node (b6) at (4.5,-1) [vertex] {};
 \draw [->,out=90,in=180] (1) to node [above] {$\scriptstyle \alpha_1$} (t);
 \draw [->,out=0,in=90] (t) to node [above] {$\scriptstyle \alpha_2$} (6);
 \draw [->] (1) -- node [above] {$\scriptstyle \beta_1$} (2);
 \draw [->] (2) -- node [above] {$\scriptstyle \beta_2$} (3);
 \draw [->] (3) -- node [above] {$\scriptstyle \beta_3$} (4);
 \draw [->] (4) -- node [above] {$\scriptstyle \beta_{n-2}$} (5);
 \draw [->] (5) -- node [above] {$\scriptstyle \beta_{n-1}$} (6);
 \draw [->] (b1) -- node [below left=-3pt] {$\scriptstyle \gamma_1$} (1);
 \draw [->] (2) -- node [below right=-3pt] {$\scriptstyle \delta_1$} (b1);
 \draw [->] (b2) -- node [below left=-3pt] {$\scriptstyle \gamma_2$} (2);
 \draw [->] (3) -- node [below right=-3pt] {$\scriptstyle \delta_2$} (b2);
 \draw [->] (b3) -- node [below left=-3pt] {$\scriptstyle \gamma_3$} (3);
 \draw [->] (5) -- node [left=1pt] {$\scriptstyle \delta_{n-2}$} (b5);
 \draw [->] (b6) -- node [below=3pt] {$\scriptstyle \gamma_{n-1} \quad$} (5);
 \draw [->] (6) -- node [below right=-3pt] {$\scriptstyle \delta_{n-1}$} (b6);
\end{tikzpicture} \]
and
\[ I = \left( \begin{array}{ll} \alpha_1 \alpha_2 - \beta_1 \cdots \beta_{n-1} & \\ \gamma_i \beta_i & \text{ whenever } \gamma_i \in Q_1 \\ \beta_i \delta_i & \text{ whenever } \delta_i \in Q_1 \end{array} \right). \]
\end{enumerate}
\end{theorem}

\begin{proof}
Let $H$ be a hereditary algebra derived equivalent to $\Lambda$. Then $\mathscr{C}_{\Lambda}^2 \approx \mathscr{C}_H^2$, and under this equivalence $\pi \Lambda$ corresponds to a cluster tilting object $T \in \mathscr{C}_H^2$. In particular $\widetilde{\Lambda} \iso \End_{\mathscr{C}_{\Lambda}^2}(\pi \Lambda) \iso \End_{\mathscr{C}_H^2}(T)$ is a cluster tilted algebra. By Corollary~\ref{corollary.selfinj2repfin} we know that $\widetilde{\Lambda}$ is self-injective. Ringel's classification \cite{Ringel_selfinj_cta} of self-injective cluster tilted algebras determine possible $H$ and $T$. This immediately gives us the classification.
\end{proof}

Inspired by \cite{Ringel_selfinj_cta} and Theorem~\ref{theorem.tilted2repfin} we also look at some quasi-tilted algebras.

\begin{proposition} \label{proposition.canonical}
Let $\Lambda$ be an iterated quasi-tilted algebra of canonical type $(2,2,2,2)$. If $\gld \Lambda \leqslant 2$ then $\Lambda$ is $2$-representation-finite.
\end{proposition}

\begin{proof}
The Auslander-Reiten quiver of a weighted projective line of type $(2,2,2,2)$, and hence also the Auslander-Reiten quiver of the corresponding cluster category, consists of tubes of rank $1$ and $2$. Thus any object $X$ in the cluster category satisfies
\[ X \iso \tau^2 X = \mathbb{S} \mathbb{S}_2 X \iso \mathbb{S} X. \]
The claim now follows from Corollary~\ref{corollary.selfinj2repfin} (3) \then\ (1).
\end{proof}

\begin{remark}
Proposition~\ref{proposition.canonical} shows in particular, that there is a one-parameter family of $2$-representation-finite algebras (namely the canonical algebras of type $(2,2,2,2)$). On the other hand, it is known by \cite{BGRS} that there are only countably many isomorphism classes of representation-finite algebras. Thus in this way (weak) $2$-representation-finiteness differs from (weak) $1$-representation-finiteness.
\end{remark}

\section{The stable category I: the \texorpdfstring{$n$}{n}-representation-finite case} \label{sect.selfinj}

In the previous section we have seen that, for an $n$-representation-finite algebra $\Lambda$, the $(n+1)$-preprojective algebra $\widetilde{\Lambda}$ is self-injective. Hence its stable module category is a triangulated category. In this section we will see that this category is precisely the $(n+1)$-Amiot cluster category of the stable Auslander algebra $\Gamma = \underline{\End}_{\Lambda}(\widetilde{\Lambda})$. We will actually see that $\stabmod \mathscr{U}\psL \approx \mathscr{D}_{\Gamma}$, and we have the following picture:
\[ \begin{tikzpicture}[xscale=3,yscale=-1.5]
 \node (A) at (0,0) {$\mathscr{D}_{\Gamma}$};
 \node (B) at (1,0) {$\stabmod \mathscr{U}\psL$};
 \node (C) at (0,1) {$\mathscr{C}_{\Gamma}^{n+1}$};
 \node (D) at (1,1) {$\stabmod \widetilde{\Lambda}$};
 \draw (A) -- node [above] {$\approx$} (B);
 \draw (C) -- node [above] {$\approx$} (D);
 \draw [->] (A) -- node [left] {$\pi$} (C);
 \draw [->] (B) -- node [right] {push-down} (D);
\end{tikzpicture} \]

\subsection{The \texorpdfstring{$(n+1)$}{(n+1)}-Calabi-Yau property}

First we prove that in the setup above $\stabmod \widetilde{\Lambda}$ is $(n+1)$-Calabi-Yau. We need the following notation, to be able to lift auto-equivalences from categories to their (stable) module categories.

\begin{definition}
Let $\mathscr{T}$ be a category, $\alpha \in \Aut(\mathscr{T})$ an auto-equivalence. Let $X \in \mod \mathscr{T}$. Then we denote by $\alpha X$ the module that is given by the composition
\[ \mathscr{T} \tol[30]{\alpha^{-1}} \mathscr{T} \tol[30]{X} (\mod k)^{\op}. \]
In particular we have $\alpha \Hom_{\mathscr{T}}(-, T) = \Hom_{\mathscr{T}}(-, \alpha T)$ (this is the reason for choosing to compose with $\alpha^{-1}$ above). Note that $\alpha$ induces auto-equivalences on $\mod \mathscr{T}$ and $\stabmod \mathscr{T}$.
\end{definition}

\begin{proposition} \label{proposition.serrestable1}
Let $\mathscr{T}$ be a triangulated category with a Serre functor $\leftsub{\mathscr{T}}{\mathbb{S}}$. Let $\mathscr{U}$ be a subcategory satisfying $\leftsub{\mathscr{T}}{\mathbb{S}} \mathscr{U} = \mathscr{U}$. Then $\stabmod \mathscr{U}$ is a triangulated category with Serre functor $\leftsub{\mathscr{T}}{\mathbb{S}} \circ [-1]_{\stabmod \mathscr{U}}$.
\end{proposition}

\begin{proof}
We have $D\Hom_{\mathscr{T}}(U, -) = \Hom_{\mathscr{T}}(-, \leftsub{\mathscr{T}}{\mathbb{S}} U)$. So projective and injective objects in $\mod \mathscr{U}$ coincide, and hence $\stabmod \mathscr{U}$ is a triangulated category.

We calculate the inverse Auslander-Reiten translation $\tau^-_{\mathscr{U}}$ of a $\mathscr{U}$-module $Y$ in the following diagram.
\[ \begin{tikzpicture}[xscale=4,yscale=-1]
 \node (A1) at (0,0) {$Y$};
 \node (B1) at (1,0) {$D\Hom_{\mathscr{T}}(U^0, -)$};
 \node (C1) at (2,0) {$D\Hom_{\mathscr{T}}(U^1, -)$};
 \node (B2) at (1,1) {$\Hom_{\mathscr{T}}(-, \leftsub{\mathscr{T}}{\mathbb{S}}U^0)$};
 \node (C2) at (2,1) {$\Hom_{\mathscr{T}}(-, \leftsub{\mathscr{T}}{\mathbb{S}}U^1)$};
 \node (A3) at (0,2) {$\tau_{\mathscr{U}}^- Y [-2]_{\stabmod \mathscr{U}}$};
 \node (A3+) at (-.5,1) {$\leftsub{\mathscr{T}}{\mathbb{S}}^{-1}Y$};
 \node (B3) at (1,2) {$\Hom_{\mathscr{T}}(-, U^0)$};
 \node (C3) at (2,2) {$\Hom_{\mathscr{T}}(-, U^1)$};
 \node (D3) at (3,2) {$\tau^-_{\mathscr{U}} Y$};
 \draw [>->] (A1) -- (B1);
 \draw [->] (B1) -- (C1);
 \draw [->] (B2) -- (C2);
 \draw [>->] (A3) -- (B3);
 \draw [->] (B3) -- (C3);
 \draw [->>] (C3) -- (D3);
 \draw [double distance=1.5pt] (B1) -- (B2);
 \draw [double distance=1.5pt] (C1) -- (C2);
 \draw [double distance=1.5pt] (A3+) -- (A3);
\end{tikzpicture} \]
By Auslander-Reiten duality we have
\begin{align*}
\underline{\Hom}_{\mathscr{U}}(X, Y) & = D \Ext^1_{\mathscr{U}}(\tau^-_{\mathscr{U}} Y, X) \\
& = D \Ext^1_{\mathscr{U}}(\leftsub{\mathscr{T}}{\mathbb{S}}^{-1} Y[2]_{\stabmod \mathscr{U}}, X) \\
& = D \underline{\Hom}_{\mathscr{U}}(\leftsub{\mathscr{T}}{\mathbb{S}}^{-1} Y[1]_{\stabmod \mathscr{U}}, X).
\end{align*}
This shows that $\leftsub{\mathscr{T}}{\mathbb{S}} \circ [-1]_{\stabmod \mathscr{U}}$ is the Serre functor on $\stabmod \mathscr{U}$.
\end{proof}

\begin{lemma} \label{lemma.n_ind_n+2}
Let $\mathscr{T}$ be a triangulated category with a full subcategory $\mathscr{U}$ satisfying
\begin{itemize}
\item $\Hom_{\mathscr{T}}(\mathscr{U}, \mathscr{U}[i]) = 0$ for $i \in \{1, \ldots, n-1\}$, and
\item $\mathscr{U}[n] = \mathscr{U}$.
\end{itemize}
Any triangles
\[ \begin{tikzpicture}[xscale=1.3,yscale=1.5]
 \node (X0) at (0,0) {$U_{n+1}$};
 \node (X1) at (2,0) {$X_{n-1}$};
 \node (X2) at (4,0) {$X_{n-2}$};
 \node (X8) at (8,0) {$X_1$};
 \node (X9) at (10,0) {$U_0$};
 \node (U0) at (1,1) {$U_n$};
 \node (U1) at (3,1) {$U_{n-1}$};
 \node (U2) at (5,1) {$U_{n-2}$};
 \node (U7) at (7,1) {$U_2$};
 \node (U8) at (9,1) {$U_1$};
 \node at (6,.5) {$\cdots$};
 \draw [->] (U0) -- (U1);
 \draw [->] (U1) -- (U2);
 \draw [->] (U7) -- (U8);
 \draw [->] (X0) -- (U0);
 \draw [->] (X1) -- (U1);
 \draw [->] (X2) -- (U2);
 \draw [->] (X8) -- (U8);
 \draw [->] (U0) -- (X1);
 \draw [->] (U1) -- (X2);
 \draw [->] (U7) -- (X8);
 \draw [->] (U8) -- (X9);
 \draw [->] (X1) -- node [pos=.15] {\tikz{\draw[-] (0,.15) -- (0,-.15);}} (X0);
 \draw [->] (X2) -- node [pos=.15] {\tikz{\draw[-] (0,.15) -- (0,-.15);}} (X1);
 \draw [->] (X9) -- node [pos=.15] {\tikz{\draw[-] (0,.15) -- (0,-.15);}} (X8);
\end{tikzpicture} \]
with $U_i \in \mathscr{U}$ give rise to a long exact sequence
\[ \begin{tikzpicture}[xscale=3,yscale=-1]
 \node (C0) at (2,0) {$\cdots$};
 \node (D0) at (3,0) {$\leftsub{\mathscr{T}}(-, U_0[-n])$};
 \node (A1) at (0,1) {$\leftsub{\mathscr{T}}(-, U_{n+1})$};
 \node (B1) at (1,1) {$\leftsub{\mathscr{T}}(-, U_n)$};
 \node (C1) at (2,1) {$\cdots$};
 \node (D1) at (3,1) {$\leftsub{\mathscr{T}}(-, U_0)$};
 \node (A2) at (0,2) {$\leftsub{\mathscr{T}}(-, U_{n+1}[n])$};
 \node (B2) at (1,2) {$\leftsub{\mathscr{T}}(-, U_n[n])$};
 \node (C2) at (2,2) {$\cdots$};
 \draw [->] (C0) -- (D0);
 \draw [->] (D0) -- node [above right,at end] {$\scriptstyle d_{-1}$} (3.5,0) arc (-90:90:.25) -- (-.5,.5) arc (270:90:.25) -- (A1);
 \draw [->] (A1) -- (B1);
 \draw [->] (B1) -- (C1);
 \draw [->] (C1) -- (D1);
 \draw [->] (D1) -- (3.5,1) arc (-90:90:.25) -- node [below right,at start] {$\scriptstyle d_0$} (-.5,1.5) arc (270:90:.25) -- (A2);
 \draw [->] (A2) -- (B2);
 \draw [->] (B2) -- (C2);
\end{tikzpicture} \]
in $\mod \mathscr{U}$, with $\Im d_{\ell} = \Hom_{\mathscr{T}}(-, X_1[\ell n + 1])$.
\end{lemma}

\begin{proof}
By the existence of the triangles in the lemma we have
\begin{align*}
& X_{n-1} \in \mathscr{U} * \mathscr{U}[1] && \text{and} && X_1 \in \mathscr{U}[-1] * \mathscr{U} \\
& X_{n-2} \in \mathscr{U} * \mathscr{U}[1] * \mathscr{U}[2] &&&& X_2 \in \mathscr{U}[-2] * \mathscr{U}[-1] * \mathscr{U} \\
& \qquad \vdots &&&& \qquad \vdots \\
& X_i \in \mathscr{U} * \mathscr{U}[1] * \cdots * \mathscr{U}[n-i] &&&& X_i \in \mathscr{U}[-i] * \mathscr{U}[-i+1] * \cdots * \mathscr{U}
\end{align*}
By the left inclusion above we have $\Hom_{\mathscr{T}}(\mathscr{U}, X_i[1]) = 0$ for $i \geqslant 2$. By the right inclusion, and since $\mathscr{U} = \mathscr{U}[n]$, we have $\Hom_{\mathscr{T}}(\mathscr{U}, X_i[-1]) = 0$ for $i \leqslant n-2$. Hence for $i \in \{2, \ldots, n-1\}$ we obtain a short exact sequence of $\mathscr{U}$-modules
\[ \underbrace{\leftsub{\mathscr{T}}{(-, X_{i-1}[-1])}}_{= 0} \to[30] \leftsub{\mathscr{T}}{(-, X_i)} \to[30] \leftsub{\mathscr{T}}{(-, U_i)} \to[30] \leftsub{\mathscr{T}}{(-, X_{i-1})} \to[30] \underbrace{\leftsub{\mathscr{T}}{(-, X_i[1])}}_{= 0}. \]
The leftmost and rightmost triangle give rise to exact sequences
\[ \leftsub{\mathscr{T}}{(-, X_{n-1}[-1])} \mono[30] \leftsub{\mathscr{T}}{(-, U_{n+1})} \to[30] \leftsub{\mathscr{T}}{(-, U_n)} \epi[30] \leftsub{\mathscr{T}}{(-, X_{n-1})} \text{, and} \]
\[ \leftsub{\mathscr{T}}{(-, X_1)} \mono[30] \leftsub{\mathscr{T}}{(-, U_1)} \to[30] \leftsub{\mathscr{T}}{(-, U_0)} \epi[30] \leftsub{\mathscr{T}}{(-, X_1[1])}. \]

Thus we have an exact sequence
\[ \leftsub{\mathscr{T}}(-,X_{n-1}[-1]) \mono \leftsub{\mathscr{T}}(-,U_{n+1}) \to \leftsub{\mathscr{T}}(-,U_n) \to \cdots \to \leftsub{\mathscr{T}}(-,U_1) \to \leftsub{\mathscr{T}}(-,U_0) \epi \leftsub{\mathscr{T}}(-,X_1[1]). \]
Similarly (since $\mathscr{U} = \mathscr{U}[\ell n]$) we get an exact sequence
\[ \leftsub{\mathscr{T}}(-,X_{n-1}[\ell n-1]) \mono \leftsub{\mathscr{T}}(-,U_{n+1}[\ell n]) \to \cdots \to \leftsub{\mathscr{T}}(-,U_0[\ell n]) \epi \leftsub{\mathscr{T}}(-,X_1[\ell n+1]) \]
for any $\ell \in \mathbb{Z}$. Now note that
\[ \leftsub{\mathscr{T}}{(-, X_{n-1}[n\ell - 1])} \iso \leftsub{\mathscr{T}}{(-, X_{n-2}[n\ell - 2])} \iso \cdots \iso \leftsub{\mathscr{T}}(-, X_1[\underbrace{n\ell - (n-1)}_{= n(\ell - 1) + 1}]) \]
on $\mathscr{U}$. Hence we obtain the long exact sequence of the lemma by concatenating the exact sequences above.
\end{proof}

\begin{proposition} \label{proposition.serrestable2}
Let $\mathscr{T}$ be a triangulated category, and let $\mathscr{U}$ be an $n$-cluster tilting subcategory of $\mathscr{T}$, satisfying $\mathscr{U}[n] = \mathscr{U}$. Then we have $[n+2]_{\stabmod \mathscr{U}} = [n]_{\mathscr{T}}$ on $\stabmod \mathscr{U}$.
\end{proposition}

\begin{proof}
By Lemma~\ref{lemma.factorismod} any object in $\mod \mathscr{U}$ is of the form $\Hom_{\mathscr{T}}(-, X)$ for some $X \in \mathscr{U} * \mathscr{U}[1]$. Choose a triangle $U_1 \to U_0 \to X \to U_1[1]$ with $U_i \in \mathscr{U}$. Also taking triangles as in Proposition~\ref{prop.ct-resolution} for $X[-1]$ we obtain a sequence of triangles as in Lemma~\ref{lemma.n_ind_n+2}. Now the claim follows from the long exact sequence in Lemma~\ref{lemma.n_ind_n+2}.
\end{proof}

\begin{theorem} \label{theorem.serrestable}
Let $\mathscr{T}$ be a triangulated category with an $n$-cluster tilting subcategory $\mathscr{U}$ satisfying $\mathscr{U}[n] = \mathscr{U}$. Then $\stabmod \mathscr{U}$ is a triangulated category, and $\leftsub{\stabmod \mathscr{U}}{\mathbb{S}}_{n+1} = \leftsub{\mathscr{T}}{\mathbb{S}}_n$ on $\stabmod \mathscr{U}$.
\end{theorem}

\begin{proof}
We have
\begin{align*}
\leftsub{\stabmod \mathscr{U}}{\mathbb{S}} \circ [-n-1]_{\stabmod \mathscr{U}} & = \leftsub{\mathscr{T}}{\mathbb{S}} \circ [-n-2]_{\stabmod \mathscr{U}} && \text{(by Proposition~\ref{proposition.serrestable1})} \\
& = \leftsub{\mathscr{T}}{\mathbb{S}} \circ [-n]_{\mathscr{T}} && \text{(by Proposition~\ref{proposition.serrestable2}).} \qedhere
\end{align*}
\end{proof}

We have the following analogous result to \cite{KelRei} for the case $n = 2$.

\begin{corollary} \label{corollary.stabmod_CY}
Let $\mathscr{T}$ be an $n$-Calabi-Yau triangulated category with an $n$-cluster tilting subcategory $\mathscr{U}$. Assume $\mathscr{U} = \mathscr{U}[n]$. Then $\stabmod \mathscr{U}$ is $(n+1)$-Calabi-Yau.
\end{corollary}

\begin{proof}
Since $\mathscr{T}$ is $n$-Calabi-Yau, we have $\leftsub{\mathscr{T}}{\mathbb{S}}_n = 1$ on $\mathscr{T}$. By Theorem~\ref{theorem.serrestable}, we have $\leftsub{\stabmod \mathscr{U}}{\mathbb{S}}_{n+1}= \leftsub{\mathscr{T}}{\mathbb{S}}_n = 1$ on $\stabmod \mathscr{U}$. Thus $\stabmod \mathscr{U}$ is $(n+1)$-Calabi-Yau.
\end{proof}

\subsection{A tilting object in \texorpdfstring{$\stabmod \mathscr{U}\psL$}{\_mod\_ U}} \label{subsect.stab_is_Db}

In the rest of this section, let $\Lambda$ be an $n$-representation-finite algebra, which is not semisimple. Recall that the $\Lambda$-module $\widetilde{\Lambda}$ is the unique basic $n$-cluster tilting object in $\mod \Lambda$. We denote by $\widetilde{\Lambda}_P$ (respectively, $\widetilde{\Lambda}_I$) the maximal direct summand of $\widetilde{\Lambda}$ (as a $\Lambda$-module) without non-zero projective (respectively, injective) direct summands.

For an algebra $\Gamma$, we denote by $\widehat{\Gamma}$ the \emph{repetitive category} of $\Gamma$ (see \cite{Ha}).

\begin{theorem} \label{theorem.stableisderived}
Let $\Gamma = \underline{\End}_{\Lambda}(\widetilde{\Lambda})$ be the stable $n$-Auslander algebra of $\Lambda$. Then we have $\mathscr{U}\psL \approx \widehat{\Gamma}$.
\end{theorem}

In the proof we will use the following piece of notation.

\begin{notation}
As in \cite{R}, for full subcategories $\mathscr{A}$ and $\mathscr{B}$ of an additive category $\mathscr{T}$ we denote by $\mathscr{A} \vee \mathscr{B}$ the full subcategory whose objects are direct sums of an object in $\mathscr{A}$ and an object in $\mathscr{B}$.
\end{notation}

We start by proving the following observation.

\begin{lemma} \label{lemma.S_decom_U}
\[ \mathscr{U}\psL = \bigvee_{i\in \mathbb{Z}} \add \mathbb{S}^i \widetilde{\Lambda}_P. \]
\end{lemma}

\begin{proof}
For simplicity we set $\mathscr{U}' := \bigvee_{i\in \mathbb{Z}} \add \mathbb{S}^i \widetilde{\Lambda}_P$. Since, by Theorem~\ref{theorem.tilde_is_ctmod}(2) we have $\widetilde{\Lambda}_P \in \mathscr{U}\psL$, and by Theorem~\ref{theorem.nrepfinderived} we have $\mathbb{S} \mathscr{U}\psL = \mathscr{U}\psL$ we clearly have $\mathscr{U}' \subseteq \mathscr{U}\psL$. Hence it only remains to show that $\mathscr{U}\psL \subseteq \mathscr{U}'$. Since $\mathscr{U}\psL = \bigvee_{i\in \mathbb{Z}} \add \mathbb{S}_n^i \Lambda$ by definition, we only have to show
\begin{enumerate}
\item $\Lambda \in \mathscr{U}'$,
\item $\mathbb{S}_n^{\pm 1} \mathscr{U}' \subseteq \mathscr{U}'$, or, equivalently, that $\mathbb{S}_n^{\pm 1} \widetilde{\Lambda}_P \in \mathscr{U}'$.
\end{enumerate}

To show (1) we first decompose $\Lambda = Q \oplus Q'$, where $Q$ is the sum of all indecomposable projective $\Lambda$-modules $P$ such that $\mathbb{S}^iP$ is a projective $\Lambda$-module for any $i \geqslant 0$. Then $\mathbb{S}$ induces an autoequivalence of $\add Q$, and therefore $\mathbb{S}^i Q \in \add Q \, \forall i \in \mathbb{Z}$. In particular $Q$ is a projective and injective $\Lambda$-module. Since $\Lambda$ is not self-injective, we have $Q' \neq 0$. We have $\Hom_{\mathscr{D}_{\Lambda}}(Q',Q) = \Hom_{\mathscr{D}_{\Lambda}}(\mathbb{S}^i Q',\mathbb{S}^i Q) = 0$ for $i \ll 0$ by Proposition~\ref{prop.tauderived}. Similarly we have $\Hom_{\mathscr{D}_{\Lambda}}(Q,Q')=0$. Since $\Lambda$ is connected this means $Q = 0$. Hence, for any indecomposable projective $\Lambda$-module $P$, there is $i > 0$ such that $\mathbb{S}^i P \not\in \add \Lambda$. For the minimal such $i$ we have $\mathbb{S}^i P$ is an injective non-projective $\Lambda$-module, and hence $\mathbb{S}^i P \in \add \widetilde{\Lambda}_P$. This proves (1).

It follows that also $\widetilde{\Lambda} = \Lambda \oplus \widetilde{\Lambda}_P \in \mathscr{U}'$, and hence that $\Lambda[n] \in \add \mathbb{S} \widetilde{\Lambda} \subseteq \mathscr{U}'$. Now claim (2) follows, since
\begin{align*}
& \mathbb{S}_n \widetilde{\Lambda}_P = \widetilde{\Lambda}_I \in \add \widetilde{\Lambda} \subseteq \mathscr{U}' \text{, and} \\
& \mathbb{S}_n^{-1} \widetilde{\Lambda}_P \in \add \mathbb{S}_n^{-1} \widetilde{\Lambda} = \add (\mathbb{S}_n^{-1} \widetilde{\Lambda}_I \oplus \mathbb{S}_n^{-1} D\Lambda) = \add (\widetilde{\Lambda}_P \oplus \Lambda[n]) \subseteq \mathscr{U}'. \qedhere
\end{align*}
\end{proof}

\begin{proof}[Proof of Theorem~\ref{theorem.stableisderived}]
Since $\Hom_{\Lambda}(\widetilde{\Lambda}_P, \Lambda) = 0$ by Proposition~\ref{prop.tauderived}(2), we have $\Gamma = \underline{\End}_{\Lambda}(\widetilde{\Lambda}) = \End_{\Lambda}(\widetilde{\Lambda}_P)$. 

By Lemma~\ref{lemma.S_decom_U}, we only have to show that
\[ \Hom_{\mathscr{D}}(\widetilde{\Lambda}_P, \mathbb{S}^i(\widetilde{\Lambda}_P)) = \left\{ \begin{array}{ll} 0 & i < 0 \\ \End_{\Lambda}(\widetilde{\Lambda}_P) & i = 0 \\ D \End_{\Lambda}(\widetilde{\Lambda}_P) & i = 1 \\ 0 & i > 1 \end{array} \right.. \]
The second line is clear, the third follows immediately from the definition of $\mathbb{S}$. 

To show the other lines, note that $\mathbb{S}^{-1} \mathscr{D}^{\geqslant n} \subseteq \mathscr{D}^{\geqslant n}$, and $\mathbb{S} \mathscr{D}^{\leqslant n} \subseteq \mathscr{D}^{\leqslant n}$. From
\[ \mathbb{S}^{-1} \widetilde{\Lambda}_P \in \add \mathbb{S}^{-1} \widetilde{\Lambda} = \add(\mathbb{S}^{-1} D\Lambda \oplus \mathbb{S}^{-1} \widetilde{\Lambda}_I) = \add (\Lambda \oplus \widetilde{\Lambda}_P[-n]),\]
we have inductively
\[\mathbb{S}^i \widetilde{\Lambda}_P \in (\add \Lambda) \vee \mathscr{D}^{\geqslant n}\]
for any $i<0$. Thus we have the first line.

From
\[\mathbb{S}^2 \widetilde{\Lambda}_P = \mathbb{S} \widetilde{\Lambda}_I [n] \in \add \mathbb{S}(\Lambda \oplus \widetilde{\Lambda}_P)[n] \subseteq \add (D\Lambda[n]) \vee \mathscr{D}^{\leqslant -2n},\]
we have inductively
\[\mathbb{S}^i \widetilde{\Lambda}_P \in \add D\Lambda[n] \vee \mathscr{D}^{\leqslant -2n} \]
for any $i>1$. Thus, since $\gld \Lambda \leqslant n$, we have the last line.
\end{proof}

Our main motivation for looking at the repetitive category is that, by \cite{Ha}, it is closely related to the derived category. Hence we obtain the following immediate consequence of Theorem~\ref{theorem.stableisderived}. Here we denote by $\mathtt{R} \colon \mod \Gamma = \mod (\add \widetilde{\Lambda}_P) \sub \mod \mathscr{U}\psL$ the inclusion induced by the projection $\mathscr{U}\psL \epi \add_{\Lambda} \widetilde{\Lambda}_P$. Since $\mathtt{R}$ is exact it induces a functor $\mathscr{D}_{\Gamma} \to \mathscr{D}_{\mathscr{U}\psL}$, which will also be denoted by $\mathtt{R}$.

\begin{corollary} \label{corollary.stabmodisderived}
\begin{enumerate}
\item In the situation of Theorem~\ref{theorem.stableisderived} we have
\[ \stabmod \mathscr{U}\psL \approx \mathscr{D}_{\Gamma}. \]
\item The equivalence in (1) can be chosen to be the composition of the functor $\mathtt{R} \colon \mathscr{D}_{\Gamma} \to \mathscr{D}_{\mathscr{U}\psL}$ and the projection from $\mathscr{D}_{\mathscr{U}\psL}$ onto $\stabmod \mathscr{U}\psL$.
\end{enumerate}
\end{corollary}

\begin{proof}
\begin{enumerate}
\item By Theorem~\ref{theorem.stableisderived} we have $\stabmod \mathscr{U}\psL \approx \stabmod \widehat{\Gamma}$. Since $\gld \Gamma \leqslant n+1 < \infty$ by Theorem~\ref{theorem.Auslalg_gld}, we have $\stabmod \widehat{\Gamma} \approx \mathscr{D}_{\Gamma}$ by \cite{Ha}.
\item Under the equivalence $\stabmod \widehat{\Gamma} \approx \stabmod \mathscr{U}\psL$ given by Theorem~\ref{theorem.stableisderived}, $\Gamma$ corresponds to $\mathtt{R}\Gamma$. It is a general fact for repetitive categories that $\Gamma$ is a tilting object in $\stabmod \widehat{\Gamma}$. Thus our composed functor $\mathscr{D}_{\Gamma} \to \mathscr{D}_{\mathscr{U}\psL} \to \stabmod \mathscr{U}\psL$ sends the tilting object $\Gamma\in \mathscr{D}_{\Gamma}$ to the tilting object $\mathtt{R}\Gamma \in \stabmod \mathscr{U}\psL$, hence is an equivalence. \qedhere
\end{enumerate}
\end{proof}

As an application for the case $n=1$, we have the following results, which show a surprising commutativity of $\stabmod -$ and $\widehat{-}$, and of $\stabmod -$ and $\mathscr{D}_{-}$.

\begin{corollary}
\begin{enumerate}
\item Let $\Lambda$ be a representation-finite hereditary algebra.
The we have equivalences
\[ \stabmod \widehat{\Lambda} \approx \widehat{\stabmod \Lambda}
\ \text{ and } \
\stabmod \mathscr{D}_{\Lambda} \approx \mathscr{D}_{\stabmod \Lambda}.\]
\item Let $\Lambda$ and $\Lambda'$ be representation-finite hereditary
algebras. If $\Lambda$ and $\Lambda'$ are derived equivalent, then their
stable Auslander algebras are derived equivalent.
\end{enumerate}
\end{corollary}

\begin{proof}
(1) By \cite{Ha} and Theorem~\ref{theorem.stableisderived}, we have
\[ \stabmod \widehat{\Lambda} \approx \mathscr{D}_{\Lambda} \approx \mathscr{U}\psL \approx \widehat{\stabmod \Lambda}.\]
The second equivalence now follows from Corollary~\ref{corollary.stabmodisderived}.

(2) By (1) we have
\[\mathscr{D}_{\stabmod \Lambda} \approx \stabmod \mathscr{D}_{\Lambda} \approx \stabmod \mathscr{D}_{\Lambda'} \approx \mathscr{D}_{\stabmod \Lambda'}. \qedhere\]
\end{proof}

One can show the corresponding statement to (2) above for $n$-representation-finite algebras, see \cite{Lad_der_eq_n-APR}. Notice that (2) above is not valid if we replace ``stable Auslander algebras'' by ``Auslander algebras''.

\subsection{The \texorpdfstring{$(n+1)$}{(n+1)}-Amiot cluster category and \texorpdfstring{$\stabmod \widetilde{\Lambda}$}{\_mod\_ Lambda\~{}}} \label{subsec.amiot_cluster_selfinj}

Since, by construction, $\mathscr{U}\psL$ is a covering of $\add \pi\Lambda$, we have an exact push-down functor $\mod \mathscr{U}\psL \to \mod (\add \pi \Lambda) = \mod \widetilde{\Lambda}$. It maps projectives to projectives, so it also induces a push-down functor $\stabmod \mathscr{U}\psL \to \stabmod \widetilde{\Lambda}$. The aim of this subsection is to show that the upper equivalence in the following diagram (which we found in Corollary~\ref{corollary.stabmodisderived}) induces the lower equivalence.
\[ \begin{tikzpicture}[xscale=3,yscale=-1.5]
 \node (A) at (0,0) {$\mathscr{D}_{\Gamma}$};
 \node (B) at (1,0) {$\stabmod \mathscr{U}\psL$};
 \node (C) at (0,1) {$\mathscr{C}_{\Gamma}^{n+1}$};
 \node (D) at (1,1) {$\stabmod \widetilde{\Lambda}$};
 \draw (A) -- node [above] {$\approx$} (B);
 \draw (C) -- node [above] {$\approx$} (D);
 \draw [->] (A) -- node [left] {$\pi$} (C);
 \draw [->] (B) -- node [right] {push-down} (D);
\end{tikzpicture} \]
In particular this will show that $\stabmod \widetilde{\Lambda}$ is equivalent to the $(n+1)$-Amiot cluster category of $\Gamma$, and hence has an $(n+1)$-cluster tilting object.

Our first task is to construct a functor $\mathscr{C}_{\Gamma}^{n+1} \to \stabmod \widetilde{\Lambda}$. We use the universal property of the $(n+1)$-Amiot cluster category (see Appendix~\ref{appendix}, in particular Theorem~\ref{thm.universal_stable}).

Clearly $\widetilde{\Lambda}_P$ is an ideal of $\widetilde{\Lambda}$, hence in particular a $\widetilde{\Lambda} \otimes_k \widetilde{\Lambda}^{\op}$-module. Since $\Lambda$ is a subalgebra of $\widetilde{\Lambda}$, the right action of $\widetilde{\Lambda}$ on $\widetilde{\Lambda}_P$ gives a $k$-algebra homomorphism
\[ \widetilde{\Lambda} \to[30] \End_{\Lambda}( \widetilde{\Lambda}_P ) = \Gamma, \]
and hence a functor
\begin{equation} \label{eq.funct_widetilde-gamma}
\proj \widetilde{\Lambda} \to[30] \proj \Gamma.
\end{equation}
We denote by $\mathtt{A} \colon \mathscr{D}_{\Gamma} \to \mathscr{D}_{\widetilde{\Lambda}}$ the induced restriction functor.

\begin{lemma} \label{lemma.comm_res_pd}
We have the following commutative diagram:
\[ \begin{tikzpicture}[xscale=3,yscale=1.5]
\node (A) at (0,1) {$\mathscr{D}_{\Gamma}$};
\node (B) at (1,1) {$\mathscr{D}_{\mathscr{U}\psL}$};
\node (C) at (1,0) {$\mathscr{D}_{\widetilde{\Lambda}}$};
\draw [->] (A) -- node [above] {$\mathtt{R}$} (B);
\draw [->] (B) -- node [right] {push-down} (C);
\draw [->] (A) -- node [below left] {$\mathtt{A}$} (C);
\end{tikzpicture} \]
\end{lemma}

\begin{proof}
The push-down functor $\mathscr{D}_{\mathscr{U}\psL} \to \mathscr{D}_{\widetilde{\Lambda}}$ is induced by the functor
\begin{align*}
\proj \widetilde{\Lambda} & \to[30] \Add \mathscr{U}\psL \\
\widetilde{\Lambda} & \mapsto[30] \bigoplus_{i\in \mathbb{Z}} \mathbb{S}_n^i \Lambda,
\end{align*}
where $\Add \mathscr{U}\psL$ is the smallest full subcategory of $D(\Mod \Lambda)$ containing $\mathscr{U}\psL$ and closed under arbitrary direct sums. Recall that the functor $\mathtt{R}$ is induced by the projection functor $\mathscr{U} \epi \add_{\Lambda} \widetilde{\Lambda}_P$. The composition of these functors
\[ \proj \widetilde{\Lambda} \to[30] \Add \mathscr{U} \epi[30] \Add_{\Lambda} \widetilde{\Lambda}_P = \Proj \Gamma \]
coincides with the functor in \eqref{eq.funct_widetilde-gamma}. Thus we have the commutativity.
\end{proof}

We have the following commutative diagram.
\begin{equation} \label{diagram.functors}
\begin{tikzpicture}[baseline=-.75cm,xscale=4,yscale=2]
 \node (A) at (0,0) {$\mathscr{D}_{\Gamma}$};
 \node (B) at (1,0) {$\mathscr{D}_{\mathscr{U}\psL}$};
 \node (C) at (2,0) {$\stabmod \mathscr{U}\psL$};
 \node (D) at (1,-1) {$\mathscr{D}_{\widetilde{\Lambda}}$};
 \node (E) at (2,-1) {$\stabmod \widetilde{\Lambda}$};
 \node (F) at (0,-1) {$\mathscr{C}_{\Gamma}^{n+1}$};
 \draw [->] (A) -- node [above,pos=.7] {$\mathtt{R}$} (B);
 \draw [->,out=45,in=135] (A) to node [pos=.7,right=15pt] {$\approx$} (C);
 \draw [->>] (B) -- (C);
 \draw [->>] (D) -- (E);
 \draw [->] (B) -- node [fill=white,inner sep=1pt] {push-down} (D);
 \draw [->] (C) -- node [fill=white,inner sep=1pt] {push-down} (E);
 \draw [->] (A) -- node [below left] {$\mathtt{A}$} (D);
 \draw [->] (A) -- node [left] {$\pi$} (F);
\end{tikzpicture}
\end{equation}
We will find a triangle functor $\mathtt{H} \colon \mathscr{C}_{\Gamma}^{n+1} \to \stabmod \widetilde{\Lambda}$ making \eqref{diagram.functors} commutative. We need the following observation.

\begin{proposition} \label{prop.amiot.factorization.conditions}
In the setup above there is a triangle
\[ X \to[30] \mathtt{A}(D \Gamma[-n-1]) \to[30] \mathtt{A}(\Gamma) \to[30] X[1] \]
in $\mathscr{D}_{\widetilde{\Lambda} \otimes_k \Gamma^{\op}}$, such that the image of $X$ under the forgetful functor $\mathscr{D}_{\widetilde{\Lambda} \otimes_k \Gamma^{\op}} \to \mathscr{D}_{\widetilde{\Lambda}}$ belongs to $\perf \widetilde{\Lambda}$. (Note that $\mathtt{A}(\Gamma)$ and $\mathtt{A}(D\Gamma)$ are naturally $\widetilde{\Lambda} \otimes_k \Gamma^{\op}$-modules, where the right $\Gamma$-module structure comes from the natural action on $\Gamma$ and $D \Gamma$, respectively.)
\end{proposition}

\begin{proof}
Note that
\[ \mathtt{A}(\Gamma) = \Hom_{\Lambda}(\widetilde{\Lambda}_P, \widetilde{\Lambda}_P) = \Hom_{\Lambda}(\tau_n^- \widetilde{\Lambda}, \widetilde{\Lambda}_P) = D \Ext_{\Lambda}^n(\widetilde{\Lambda}_P, \widetilde{\Lambda}) \]
as $\widetilde{\Lambda} \otimes_k \Gamma^{\op}$-modules, and similarly
\[ \mathtt{A}(D \Gamma [-n-1]) = D \Hom_{\Lambda}(\widetilde{\Lambda}_P, \widetilde{\Lambda}_P)[-n-1] = D \Hom_{\Lambda}(\widetilde{\Lambda}_P, \widetilde{\Lambda})[-n-1]. \]
Since $R\Hom_{\Lambda}(\widetilde{\Lambda}_P, \widetilde{\Lambda})$ is concentrated in degrees $0$ and $(-n)$ we have a triangle
\[ D \Ext_{\Lambda}^n(\widetilde{\Lambda}_P, \widetilde{\Lambda})[n] \to[30] D R\Hom_{\Lambda}(\widetilde{\Lambda}_P, \widetilde{\Lambda}) \to[30] \underbrace{D \Hom_{\Lambda}(\widetilde{\Lambda}_P, \widetilde{\Lambda})}_{= \mathtt{A}(D \Gamma)} \to[30] \underbrace{D \Ext_{\Lambda}^n(\widetilde{\Lambda}_P, \widetilde{\Lambda})}_{= \mathtt{A}(\Gamma)}[n+1] \]
in $\mathscr{D}_{\widetilde{\Lambda} \otimes_k \Gamma^{\op}}$. It remains to show that $DR\Hom_{\Lambda}(\widetilde{\Lambda}_P, \widetilde{\Lambda}) \in \perf \widetilde{\Lambda}$. We have $\widetilde{\Lambda}_P \in \perf \Lambda$ since $\gld \Lambda < \infty$, so in $\mathscr{D}_{\widetilde{\Lambda}}$ we have
\begin{align*}
DR\Hom_{\Lambda}(\widetilde{\Lambda}_P, \widetilde{\Lambda}) & \in \thick DR\Hom_{\Lambda}(\Lambda, \widetilde{\Lambda}) \\
& = \thick D \widetilde{\Lambda} = \perf \widetilde{\Lambda} \qedhere
\end{align*}
\end{proof}

Using the universal property of $(n+1)$-Amiot cluster categories we have the following consequence of Proposition~\ref{prop.amiot.factorization.conditions}.

\begin{proposition} \label{prop.H_exists}
There is a triangle functor $\mathtt{H} \colon \mathscr{C}_{\Gamma}^{n+1} \to \stabmod \widetilde{\Lambda}$ making Diagram~\eqref{diagram.functors} commutative.
\end{proposition}

\begin{proof}
This follows from the universal property of the $(n+1)$-Amiot cluster category and Proposition~\ref{prop.amiot.factorization.conditions} above. We will give details on the proof in Appendix~\ref{appendix}, since this requires a lot of background on DG categories, which is not closely related to the main subjects of this paper.
\end{proof}

Now we are ready to prove the following result.

\begin{theorem} \label{theorem.H_equiv_selfinj}
The functor $\mathtt{H} \colon \mathscr{C}_{\Gamma}^{n+1} \to \stabmod \widetilde{\Lambda}$ is an equivalence.
\end{theorem}

\begin{proof}
We first check that $\mathtt{H}$ is fully faithful on the image of $\mathscr{D}_{\Gamma} = \stabmod \mathscr{U}\psL$. Indeed, for $X, Y \in \mod \mathscr{U}\psL$ we have
\begin{align*}
\Hom_{\mathscr{C}_{\Gamma}^{n+1}}(X, Y) & = \Hom_{\mathscr{D}_{\Gamma} / (\leftsub{\mathscr{D}_{\Gamma}}{\mathbb{S}}_{n+1})}(X, Y) \\
& = \Hom_{\stabmod \mathscr{U}\psL / (\leftsub{\stabmod \mathscr{U}\psL}{\mathbb{S}}_{n+1})}(X, Y) \\
& = \coprod_{i \in \mathbb{Z}} \underline{\Hom}_{\mathscr{U}\psL}(X, \leftsub{\stabmod \mathscr{U}\psL}{\mathbb{S}}_{n+1}^i Y) \\
& = \coprod_{i \in \mathbb{Z}} \underline{\Hom}_{\mathscr{U}\psL}(X, \leftsub{\mathscr{D}_{\Lambda}}{\mathbb{S}}_{n}^i Y) \qquad \qquad \text{(by Theorem~\ref{theorem.serrestable})} \\
& = \underline{\Hom}_{\underbrace{\mathscr{U}\psL/(\leftsub{\mathscr{D}_{\Lambda}}{\mathbb{S}}_n)}_{= \widetilde{\Lambda}}}(X, Y)
\end{align*}
(here, by abuse of notation, we denote the images of $X$ and $Y$ in the various categories by $X$ and $Y$).

Now, since $\mathscr{C}_{\Gamma}^{n+1}$ is generated as a triangulated category by the image of $\mathscr{D}_{\Gamma}$, the functor $\mathtt{H}$ is fully faithful on the entire $\mathscr{C}_{\Gamma}^{n+1}$.

Finally note that all simple $\widetilde{\Lambda}$-modules are in the image of the push-down functor $\stabmod \mathscr{U}\psL \to \stabmod \widetilde{\Lambda}$, so they are also in the image of $\mathtt{H}$. Therefore $\mathtt{H}$ is dense.
\end{proof}

\begin{corollary} \label{cor.stab_is_Amiot}
Let $\Lambda$ be an $n$-representation-finite algebra.
\begin{enumerate}
\item The category $\stabmod \widetilde{\Lambda}$ is $(n+1)$-Calabi-Yau triangulated with an $(n+1)$-cluster tilting object.
\item The algebra $\widetilde{\Lambda}$ is weakly $(n+1)$-representation-finite. A cluster tilting object in $\mod \widetilde{\Lambda}$ is given by $\Hom_{\Lambda}(\widetilde{\Lambda}, \widetilde{\Lambda})$.
\end{enumerate}
\end{corollary}

\begin{proof}
(1) follows immediately from Theorems~\ref{theorem.H_equiv_selfinj} and \ref{theorem.clairemain2}.

(2) follows from (1), since for any $M \in \mod \widetilde{\Lambda}$, we have that the image of $M$ in $\stabmod \widetilde{\Lambda}$ is $(n+1)$-cluster tilting if and only if $M \oplus \widetilde{\Lambda}$ is $(n+1)$-cluster tilting in $\mod \widetilde{\Lambda}$. By Theorem~\ref{theorem.clairemain2} the image of $\Gamma$ is a cluster tilting object in $\mathscr{C}_{\Gamma}^{n+1}$. By Theorem~\ref{theorem.H_equiv_selfinj}
\[ \mathtt{H}(\Gamma) = \mathtt{A}(\Gamma) = \Hom_{\Lambda}(\widetilde{\Lambda}_P,\widetilde{\Lambda}_P) \]
is an $n$-cluster tilting object in $\stabmod \widetilde{\Lambda}$. Adding $\widetilde{\Lambda}$ we obtain the cluster tilting object $\Hom_{\Lambda}(\widetilde{\Lambda}_P, \widetilde{\Lambda}_P) \oplus \widetilde{\Lambda}$ in $\mod \widetilde{\Lambda}$. Finally note that
\[ \Hom_{\Lambda}(\widetilde{\Lambda}_P, \widetilde{\Lambda}_P) \oplus \widetilde{\Lambda} = \underbrace{\Hom_{\Lambda}(\widetilde{\Lambda}_I, \widetilde{\Lambda}_I)}_{= \Hom_{\Lambda}(\widetilde{\Lambda}, \widetilde{\Lambda}_I)} \oplus \Hom_{\Lambda}(\widetilde{\Lambda}, D\Lambda) = \Hom_{\Lambda}(\widetilde{\Lambda}, \widetilde{\Lambda}), \]
so $\Hom_{\Lambda}(\widetilde{\Lambda}, \widetilde{\Lambda})$ is $(n+1)$-cluster tilting in $\mod \widetilde{\Lambda}$ as claimed.
\end{proof}

\begin{remark}
In \cite[Section~6]{Iy_n-Auslander} (also see \cite[Section~5]{IO}), the first author described explicitly the quivers and relations of the relative $n$-Auslander algebra $n$-Aus$(Q)$ of Dynkin quivers $Q$. In the case of linear oriented $A_s$ the $(n-1)$-Auslander algebra $(n-1)$-Aus$(A_s)$ is $n$-representation-finite, and has the $n$-Auslander algebra $n$-Aus$(A_s)$ and the stable $n$-Auslander algebra $n$-Aus$(A_{s-1})$. Hence Theorem~\ref{theorem.H_equiv_selfinj} implies
\[ \stabmod \widetilde{(n-1)\text{-Aus}(A_s)} \approx \mathscr{C}^{n+1}_{n\text{-Aus}(A_{s-1})} \]
(see also \cite[Theorem~5.7]{IO}).
\end{remark}

\begin{example}
The $1$-Auslander algebra (= classical Auslander algebra) of $A_4$ is given by the left quiver with relations below. Its $2$-preprojective algebra is given by the right quiver below, with commutativity relations in all squares, and zero-relations where the squares are cut off on the border.
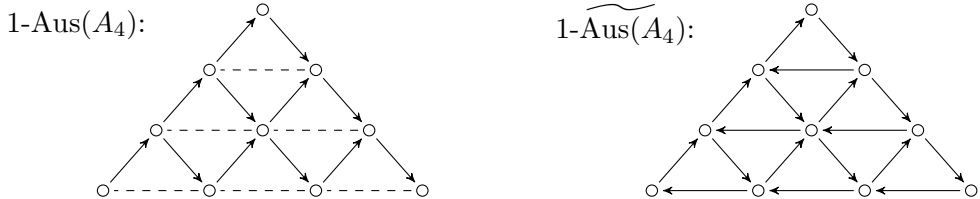
\[
\begin{tikzpicture}[xscale=.7,yscale=-.8]
 \node at (-3.5,.25) {$1$-Aus$(A_4)$:};
 \node (1) at (0,0) [vertex] {};
 \node (2) at (-1,1) [vertex] {};
 \node (3) at (1,1) [vertex] {};
 \node (4) at (-2,2) [vertex] {};
 \node (5) at (0,2) [vertex] {};
 \node (6) at (2,2) [vertex] {};
 \node (7) at (-3,3) [vertex] {};
 \node (8) at (-1,3) [vertex] {};
 \node (9) at (1,3) [vertex] {};
 \node (10) at (3,3) [vertex] {};
 \draw [->] (2) -- (1);
 \draw [->] (1) -- (3);
 \draw [dashed] (3) -- (2);
 \draw [->] (4) -- (2);
 \draw [->] (2) -- (5);
 \draw [->] (5) -- (3);
 \draw [->] (3) -- (6);
 \draw [dashed] (6) -- (5);
 \draw [dashed] (5) -- (4);
 \draw [->] (7) -- (4);
 \draw [->] (4) -- (8);
 \draw [->] (8) -- (5);
 \draw [->] (5) -- (9);
 \draw [->] (9) -- (6);
 \draw [->] (6) -- (10);
 \draw [dashed] (10) -- (9);
 \draw [dashed] (9) -- (8);
 \draw [dashed] (8) -- (7);
\end{tikzpicture}
\qquad \qquad
\begin{tikzpicture}[xscale=.7,yscale=-.8]
 \node at (-3.5,.25) {$\widetilde{1\text{-Aus}(A_4)}$:};
 \node (1) at (0,0) [vertex] {};
 \node (2) at (-1,1) [vertex] {};
 \node (3) at (1,1) [vertex] {};
 \node (4) at (-2,2) [vertex] {};
 \node (5) at (0,2) [vertex] {};
 \node (6) at (2,2) [vertex] {};
 \node (7) at (-3,3) [vertex] {};
 \node (8) at (-1,3) [vertex] {};
 \node (9) at (1,3) [vertex] {};
 \node (10) at (3,3) [vertex] {};
 \draw [->] (2) -- (1);
 \draw [->] (1) -- (3);
 \draw [->] (3) -- (2);
 \draw [->] (4) -- (2);
 \draw [->] (2) -- (5);
 \draw [->] (5) -- (3);
 \draw [->] (3) -- (6);
 \draw [->] (6) -- (5);
 \draw [->] (5) -- (4);
 \draw [->] (7) -- (4);
 \draw [->] (4) -- (8);
 \draw [->] (8) -- (5);
 \draw [->] (5) -- (9);
 \draw [->] (9) -- (6);
 \draw [->] (6) -- (10);
 \draw [->] (10) -- (9);
 \draw [->] (9) -- (8);
 \draw [->] (8) -- (7);
\end{tikzpicture} \]
The stable module category of the algebra $\widetilde{1\text{-Aus}(A_4)}$ is the same at the $3$-Amiot cluster category of the $2$-Auslander algebra of $A_3$, whose quiver is depicted below.
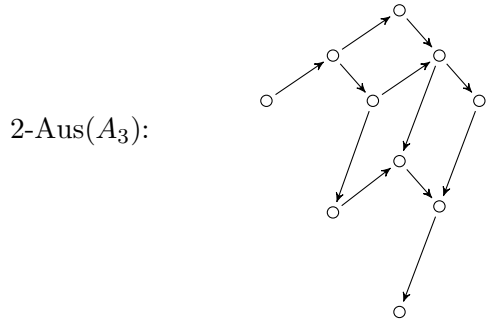
\[ \begin{tikzpicture}[xscale=.7,yscale=-.8]
 \node at (-6,2) {$2$-Aus$(A_3)$:};
 \node (1) at (0,0) [vertex] {};
 \node (2) at (-1.25,.75) [vertex] {};
 \node (3) at (.75,.75) [vertex] {};
 \node (4) at (-2.5,1.5) [vertex] {};
 \node (5) at (-.5,1.5) [vertex] {};
 \node (6) at (1.5,1.5) [vertex] {};
 \node (7) at (0,2.5) [vertex] {};
 \node (8) at (-1.25,3.35) [vertex] {};
 \node (9) at (.75,3.25) [vertex] {};
 \node (10) at (0,5) [vertex] {};
 \draw [->] (2) -- (1);
 \draw [->] (1) -- (3);
 \draw [->] (4) -- (2);
 \draw [->] (2) -- (5);
 \draw [->] (5) -- (3);
 \draw [->] (3) -- (6);
 \draw [->] (3) -- (7);
 \draw [->] (5) -- (8);
 \draw [->] (6) -- (9);
 \draw [->] (8) -- (7);
 \draw [->] (7) -- (9);
 \draw [->] (9) -- (10);
\end{tikzpicture} \]
\end{example}

\section{The stable category II: the vosnex case} \label{sect.non-selfinj}

In this section we generalize the results of the previous section to the following setup: $\Lambda$ is an algebra of global dimension $\gld \Lambda \leqslant n$, satisfying the vosnex property. Recall (see \ref{notation.vosnex}) that this means that $\Lambda$ is $\tau_n$-finite (i.e.\ $\tau_n^{-i} \Lambda = 0$ for $i \gg 0$ -- see \ref{def.tau_n}), and the cluster tilting subcategory $\mathscr{U}\psL \subseteq \mathscr{D}_{\Lambda}$ (see \ref{theorem.ct_in_Db}) satisfies $\Hom_{\mathscr{D}_{\Lambda}}(\mathscr{U}\psL[i], \mathscr{U}\psL) = 0$ for $i \in \{1, \ldots, n-2\}$. Recall that for $n=2$ the latter condition vanishes, so that in that case we only assume the algebra $\Lambda$ to be $\tau_2$-finite.

\subsection{One-sided \texorpdfstring{$n$}{n}-cluster tilting objects in \texorpdfstring{$\mod \Lambda$}{mod Lambda}}

The aim of this subsection is to show that $\widetilde{\Lambda}$ is the unique basic right $n$-cluster tilting object in $\mod \Lambda$ (Theorem~\ref{theorem.rightcto}), and dually $D \widetilde{\Lambda}$ is the unique basic left $n$-cluster tilting object in $\mod \Lambda$. Moreover we will show that $\gld \End_{\Lambda}(\widetilde{\Lambda}) \leqslant n+1$.

\begin{definition}
Let $\mathscr{A}$ be an abelian category. An object $M \in \mathscr{A}$ is called \emph{right $n$-cluster tilting object}, if
\[ \add M = \{X \in \mathscr{A} \mid \Ext^i_{\mathscr{A}}(X, M) = 0 \, \forall i \in \{1, \ldots, n-1\} \}. \]
The object $M$  is called \emph{left $n$-cluster tilting object}, if
\[ \add M = \{X \in \mathscr{A} \mid \Ext^i_{\mathscr{A}}(M, X) = 0 \, \forall i \in \{1, \ldots, n-1\} \}. \]
Clearly an object $M$ is an $n$-cluster tilting object if and only if it is a right $n$-cluster tilting object and a left $n$-cluster tilting object.
\end{definition}

The first aim of this subsection is to prove the following result.

\begin{theorem} \label{theorem.rightcto}
Let $\Lambda$ be an algebra of global dimension at most $n$ having the vosnex property.
\begin{enumerate}
\item The $\Lambda$-module $\widetilde{\Lambda}$ is the unique basic right $n$-cluster tilting object in $\mod \Lambda$.
\item The $\Lambda$-module $D \widetilde{\Lambda}$ is the unique basic left $n$-cluster tilting object in $\mod \Lambda$.
\end{enumerate}
\end{theorem}

For the proof we will need the following preparation.

\begin{proposition} \label{prop.prepare_right_cto}
Assume $\gld \Lambda \leqslant n$ and $\Lambda$ has the vosnex property. Then
\begin{enumerate}
\item $\Ext_{\Lambda}^i(D \Lambda, \widetilde{\Lambda}) = 0$ for any $i \in \{2, \ldots, n-1\}$.
\item $\Ext_{\Lambda}^i(\widetilde{\Lambda}, \widetilde{\Lambda}) = 0$ for any $i \in \{1, \ldots, n-1\}$.
\end{enumerate}
\end{proposition}

\begin{remark} \label{rem.t-trunc-triang}
We have the functorial triangle
\[ \underbrace{\trunc^{< i} X}_{\in \mathscr{D}_{\Lambda}^{< i}} \to[30] X \to[30] \underbrace{\trunc^{\geqslant i} X}_{\in \mathscr{D}_{\Lambda}^{\geqslant i}} \to[30] (\trunc^{<i} X)[1] \]
in $X \in \mathscr{D}_{\Lambda}$ induced by the standard t-structure on $\mathscr{D}_{\Lambda}$.
\end{remark}

\begin{proof}[Proof of Proposition~\ref{prop.prepare_right_cto}]
For (1) we consider the following triangles (see Remark~\ref{rem.t-trunc-triang} above)
\[ \underbrace{\trunc^{< 0} (\mathbb{S}_n^{-\ell} \Lambda)}_{\in \mathscr{D}^{< 0}} \to[30] \mathbb{S}_n^{-\ell} \Lambda \to[30] \underbrace{\Ho^0(\mathbb{S}_n^{-\ell} \Lambda)}_{= \tau_n^{-\ell} \Lambda} \to[30] (\trunc^{< 0} (\mathbb{S}_n^{-\ell} \Lambda))[1] \]
for $\ell \geqslant 0$. We show the following statements by induction on $\ell \geqslant 0$:
\begin{itemize}
\item $\trunc^{< 0} (\mathbb{S}_n^{-\ell} \Lambda) \in \mathscr{D}^{\leqslant 1-n}$, and
\item $\Ext_{\Lambda}^i(D \Lambda, \tau_n^{- \ell} \Lambda) = 0$ for any $i \in \{2, \ldots, n-1\}$.
\end{itemize}

For $\ell=0$ the first claim is clear. The second claim follows since
\[ \Ext_{\Lambda}^i(D\Lambda, \Lambda) = \Hom_{\mathscr{D}_{\Lambda}}(\underbrace{D \Lambda [-n]}_{\in \mathscr{U}\psL}, \underbrace{\Lambda}_{\in \mathscr{U}\psL} [i-n]) = 0 \qquad \text{for } i \in \{2, \ldots, n-1\} \]
by the vosnex property.

Now assume the claims above holds true for $\ell - 1$. We apply $\mathbb{S}_n^{-1}$ to the triangle for $\ell - 1$ and obtain the triangle
\begin{equation} \label{eq.loc_rest_1} \mathbb{S}_n^{-1} (\trunc^{< 0} (\mathbb{S}_n^{-(\ell-1)} \Lambda)) \to[30] \mathbb{S}_n^{-\ell} \Lambda \to[30] \mathbb{S}_n^{-1} (\tau_n^{-(\ell-1)} \Lambda) \to[30] \mathbb{S}_n^{-1} ((\trunc^{< 0} \mathbb{S}_n^{-(\ell-1)} \Lambda))[1]. \end{equation}
Note that $\mathbb{S}_n^{-1} (\trunc^{< 0} \mathbb{S}_n^{-(\ell-1)} \Lambda) \in \mathscr{D}_{\Lambda}^{\leqslant 1 - n}$, since inductively $\trunc^{< 0} \mathbb{S}_n^{-(\ell-1)} \Lambda \in \mathscr{D}_{\Lambda}^{\leqslant 1 - n}$.

By inductive assumption we also have $\Ext_{\Lambda}^i(D \Lambda, \tau_n^{- (\ell-1)} \Lambda) = 0$ for $i \in \{2, \ldots, n-1\}$. Therefore $\mathbb{S}_n^{-1} \tau_n^{- (\ell-1)} \Lambda$ can have non-zero homology only in degrees $0$, $(1-n)$, and $(-n)$. Hence in the triangle
\begin{equation} \label{eq.loc_rest_2} \tau_n^{< 0}(\mathbb{S}_n^{-1} \tau_n^{- (\ell-1)} \Lambda) \to[30] \mathbb{S}_n^{-1} \tau_n^{- (\ell-1)} \Lambda \to[30] \underbrace{\Ho^0(\mathbb{S}_n^{-1} \tau_n^{-(\ell-1)} \Lambda)}_{= \tau_n^{-\ell} \Lambda} \to[30] (\tau_n^{< 0}(\mathbb{S}_n^{-1} \tau_n^{- (\ell-1)} \Lambda))[1] \end{equation}
we have $\tau_n^{< 0}(\mathbb{S}_n^{-1} \tau_n^{- (\ell-1)} \Lambda) \in \mathscr{D}^{\leqslant 1-n}$.

We obtain the following octahedron, where the triangle we are aiming for is the one in the dotted box.
\[ \begin{tikzpicture}[xscale=5,yscale=-1.5]
 \node (A) at (1,0) {$\mathbb{S}_n^{-1} \trunc^{< 0}(\mathbb{S}_n^{-(\ell-1)} \Lambda)$};
 \node (B) at (2,0) {$\mathbb{S}_n^{-1} \trunc^{< 0}(\mathbb{S}_n^{-(\ell-1)} \Lambda)$};
 \node (C) at (0.25,1) {$\tau_n^{-\ell} \Lambda[-1]$};
 \node (D) at (1,1) {$\trunc^{< 0}(\mathbb{S}_n^{- \ell})$};
 \node (E) at (2,1) {$\mathbb{S}_n^{- \ell} \Lambda$};
 \node (F) at (2.75,1) {$\tau_n^{-\ell} \Lambda$};
 \node (G) at (0.25,2) {$\tau_n^{-\ell} \Lambda[-1]$};
 \node (H) at (1,2) {$\trunc^{< 0}(\mathbb{S}_n^{-1} \tau_n^{-(\ell - 1)} \Lambda)$};
 \node (I) at (2,2) {$\mathbb{S}_n^{-1} \tau_n^{-(\ell - 1)} \Lambda$};
 \node (J) at (2.75,2) {$\tau_n^{-\ell} \Lambda$};
 \node (K) at (1,3) {$\mathbb{S}_n^{-1} (\trunc^{< 0}(\mathbb{S}_n^{-(\ell-1)} \Lambda))[1]$};
 \node (L) at (2,3) {$\mathbb{S}_n^{-1} (\trunc^{< 0}(\mathbb{S}_n^{-(\ell-1)} \Lambda))[1]$};
 \draw [double distance=1.5pt] (A) -- (B);
 \draw [->] (C) -- (D);
 \draw [->] (D) -- (E);
 \draw [->] (E) -- (F);
 \draw [->] (G) -- (H);
 \draw [->] (H) -- (I);
 \draw [->] (I) -- (J);
 \draw [double distance=1.5pt] (K) -- (L);
 \draw [double distance=1.5pt] (C) -- (G);
 \draw [->] (A) -- (D);
 \draw [->] (D) -- (H);
 \draw [->] (H) -- (K);
 \draw [->] (B) -- (E);
 \draw [->] (E) -- (I);
 \draw [->] (I) -- (L);
 \draw [double distance=1.5pt] (F) -- (J);
 \node at (-.1,2) {\eqref{eq.loc_rest_1}:};
 \node at (2,-.6) {\eqref{eq.loc_rest_2}};
 \node at (2,-.4) [rotate=270] {:};
 \draw [dotted] (0,.75) -- (0,1.25) -- (3,1.25) -- (3,.75) -- cycle;
\end{tikzpicture} \]
We see that
\[ \trunc^{<0}(\mathbb{S}_n^{-\ell}) \in (\mathbb{S}_n^{-1} \trunc^{< 0}(\mathbb{S}_n^{-(\ell-1)} \Lambda)) * (\trunc^{< 0}(\mathbb{S}_n^{-1} \tau_n^{-(\ell - 1)} \Lambda)) \subseteq \mathscr{D}^{\leqslant 1-n} * \mathscr{D}^{\leqslant 1-n} = \mathscr{D}^{\leqslant 1-n}, \]
thus that the first claim holds for $\ell$.

To see that $\Ext_{\Lambda}^i(D \Lambda, \tau_n^{- \ell} \Lambda) = 0$ for $i \in \{2, \ldots, n-1\}$ we apply $\Hom_{\mathscr{D}}(D \Lambda, -)$ to the triangle $X_{\ell} \to \mathbb{S}_n^{- \ell} \Lambda \to \tau_n^{- \ell} \Lambda \to X_{\ell}[1]$, and note that $\Hom_{\mathscr{D}}(D\Lambda, \mathbb{S}_n^{-\ell} \Lambda[i]) = \Hom_{\mathscr{D}}(\Lambda, \mathbb{S}_n^{-\ell-1} \Lambda[i-n]) = 0$ by the vosnex property, and $\Hom_{\mathscr{D}}(D\Lambda, X_{\ell} [i+1]) = 0$ since $X_{\ell} \in \mathscr{D}_{\Lambda}^{\leqslant 1-n}$.

This completes the induction, and hence also the proof of (1).

(2) follows immediately from (1) by (the dual of) \cite[Proposition~2.5(a)]{Iy_n-Auslander}.
\end{proof}

\begin{lemma} \label{lemma.exactsequence}
For any $X \in \mod \Lambda$ there is an exact sequence $0 \to M_{n-1} \to \cdots \to M_0 \to X \to 0$ with $M_i \in \add \widetilde{\Lambda}$.
\end{lemma}

\begin{proof}
By Proposition~\ref{prop.ct-resolution} there are triangles
\[ \begin{tikzpicture}[xscale=1.3,yscale=1.5]
 \node (X0) at (0,0) {$U_{n-1}$};
 \node (X1) at (2,0) {$X_{n-2}$};
 \node (X7) at (6,0) {$X_2$};
 \node (X8) at (8,0) {$X_1$};
 \node (X9) at (10,0) {$X_0$};
 \node (U0) at (1,1) {$U_{n-2}$};
 \node (U1) at (3,1) {$U_{n-3}$};
 \node (U6) at (5,1) {$U_2$};
 \node (U7) at (7,1) {$U_1$};
 \node (U8) at (9,1) {$U_0$};
 \node at (4,.5) {$\cdots$};
 \draw [->] (U0) -- (U1);
 \draw [->] (U6) -- (U7);
 \draw [->] (U7) -- (U8);
 \draw [->] (X0) -- (U0);
 \draw [->] (X1) -- (U1);
 \draw [->] (X7) -- (U7);
 \draw [->] (X8) -- (U8);
 \draw [->] (U0) -- (X1);
 \draw [->] (U6) -- (X7);
 \draw [->] (U7) -- (X8);
 \draw [->] (U8) -- (X9);
 \draw [->] (X1) -- node [pos=.15] {\tikz{\draw[-] (0,.15) -- (0,-.15);}} (X0);
 \draw [->] (X8) -- node [pos=.15] {\tikz{\draw[-] (0,.15) -- (0,-.15);}} (X7);
 \draw [->] (X9) -- node [pos=.15] {\tikz{\draw[-] (0,.15) -- (0,-.15);}} (X8);
 \node (X9') at (10.8,0) {$X$};
 \draw [double distance=1.5pt] (X9) -- (X9');
\end{tikzpicture} \]
with $U_i \in \mathscr{U}\psL$. Note that, since $\Lambda \in \mathscr{U}\psL$, for any $U \in \mathscr{U}\psL$ we have
\[ \Ho^i(U) = \Hom_{\mathscr{D}_{\Lambda}}(\Lambda, U[i]) = 0 \begin{array}{ll} \text{for } i \in \{1, \ldots, n-1\} & \text{since } \mathscr{U}\psL \text{ is $n$-cluster tilting,} \\ \text{for } i \in \{-1, \ldots, -(n-2) \} & \text{since } \mathscr{U}\psL \text{ has the vosnex property.} \end{array} \]
Hence, on the one hand we have
\begin{align*}
& \underbrace{\Ho^i(U_{n-2})}_{=0} \to[30] \Ho^i(X_{n-2}) \to[30] \underbrace{\Ho^{i+1}(U_{n-1})}_{=0} && \text{for } i \in \{1, \ldots, n-2\} \\
& \underbrace{\Ho^i(U_{n-3})}_{=0} \to[30] \Ho^i(X_{n-3}) \to[30] \underbrace{\Ho^{i+1}(X_{n-2})}_{=0} && \text{for } i \in \{1, \ldots, n-3\} \\
& \qquad \vdots && \qquad \vdots \\
& \underbrace{\Ho^i(U_1)}_{=0} \to[30] \Ho^i(X_1) \to[30] \underbrace{\Ho^{i+1}(X_2)}_{=0} && \text{for } i = 1,
\end{align*}
%In particular we have $\Ho^1(X_1) = \cdots = \Ho^1(X_{n-2}) = \Ho^1(U_{n-1}) = 0$. On the other hand we have $\Ho^{-i}(U_j) = 0$ for $i \in \{1, \ldots, n-2\}$ by the vosnex property. Hence we also get
and on the other hand we have
\begin{align*}
& \underbrace{\Ho^{-(i+1)}(X)}_{=0} \to[30] \Ho^{-i}(X_1) \to[30] \underbrace{\Ho^{-i}(U_0)}_{=0} && \text{for } i \in \{1, \ldots, n-2\} \\
& \underbrace{\Ho^{-(i+1)}(X_1)}_{=0} \to[30] \Ho^{-i}(X_2) \to[30] \underbrace{\Ho^{-i}(X_2)}_{=0} && \text{for } i \in \{1, \ldots, n-3\} \\
& \qquad \vdots && \qquad \vdots \\
& \underbrace{\Ho^{-(i+1)}(X_{n-3})}_{=0} \to[30] \Ho^{-i}(X_{n-2}) \to[30] \underbrace{\Ho^{-i}(U_{n-1})}_{=0} && \text{for } i = 1.
\end{align*}
In particular $\Ho^1(X_i) = \Ho^{-1}(X_i) = 0 \, \forall i \in \{1, \ldots, n-2\}$. Thus we get exact sequences
\begin{align*}
& \underbrace{\Ho^{-1}(X_{n-2})}_{=0} \to \Ho^0(U_{n-1}) \to \Ho^0(U_{n-2}) \to \Ho^0(X_{n-2}) \to \underbrace{\Ho^1(U_{n-1})}_{=0} \\
& \underbrace{\Ho^{-1}(X_i)}_{=0} \to \Ho^0(X_{i+1}) \to \Ho^0(U_i) \to \Ho^0(X_i) \to \underbrace{\Ho^1(U_{i+1})}_{=0} && i \in \{1, \ldots, n-3\} \\
& \underbrace{\Ho^{-1}(X)}_{=0} \to \Ho^0(X_1) \to \Ho^0(U_0) \to \Ho^0(X) \to \underbrace{\Ho^1(X_1)}_{=0}.
\end{align*}
Gluing them, we have an exact sequence
\[ \Ho^0(U_{n-1}) \mono[30] \Ho^0(U_{n-2}) \to[30] \cdots \to[30] \Ho^0(U_0) \epi[30] \underbrace{\Ho^0(X)}_{= X}. \]
Now the claim follows, since $\Ho^0(U) \in \add \widetilde{\Lambda}$ for any $U \in \mathscr{U}$.
\end{proof}

\begin{lemma}[{\cite[Lemma~2.6]{Iy_n-Auslander}}] \label{lemma.ARformula_Ext}
Let $X, Y \in \mod \Lambda$, and $i \in \{1, \ldots, n-1\}$ such that
\[ \Ext^j_{\Lambda}(D \Lambda, X) = 0 \, \forall j \in \{n-i+1, \ldots, n-1\}. \] Then there is a surjection
\[ \Ext^{n-i}_{\Lambda}(Y, X) \epi[30] D \Ext^i_{\Lambda}(\tau_n^- X, Y). \]
\end{lemma}

Now we are ready to prove the main results of this subsection.

\begin{proof}[Proof of Theorem~\ref{theorem.rightcto}]
We only show (1), claim (2) is dual.

We first show that $\widetilde{\Lambda}$ is a right $n$-cluster tilting object. By Proposition~\ref{prop.prepare_right_cto}(2) we know that $\widetilde{\Lambda}$ is $n$-rigid. For $X \in \mod \Lambda$ consider the following diagram, where the upper sequence is exact (such a diagram exists by Lemma~\ref{lemma.exactsequence}), the $M_i$ are in $\add \widetilde{\Lambda}$, and the $X_i$ are the respective images of the maps.
\[ \begin{tikzpicture}[yscale=-1,xscale=2]
 \node (A9) at (.2,0) {$0$};
 \node (A8) at (1,0) {$M_{n-1}$};
 \node (A7) at (2,0) {$M_{n-2}$};
 \node (A5) at (3,0) {$\cdots$};
 \node (A3) at (4,0) {$M_1$};
 \node (A2) at (5,0) {$M_0$};
 \node (A1) at (6,0) {$X$};
 \node (A0) at (6.7,0) {$0$};
 \node (B8) at (1.5,1) {$X_{n-1}$};
 \node (B7) at (2.5,1) {$X_{n-2}$};
 \node (B5) at (3,1) {$\cdots$};
 \node (B4) at (3.5,1) {$X_2$};
 \node (B3) at (4.5,1) {$X_1$};
 \draw [->] (A9) -- (A8);
 \draw [->] (A8) -- (A7);
 \draw [->] (A7) -- (A5);
 \draw [->] (A5) -- (A3);
 \draw [->] (A3) -- (A2);
 \draw [->] (A2) -- (A1);
 \draw [->] (A1) -- (A0);
 \draw [double distance=1.5pt] (A8) -- (B8);
 \draw [>->] (B8) -- (A7);
 \draw [->>] (A7) -- (B7);
 \draw [>->] (B7) -- (A5);
 \draw [->>] (A5) -- (B4);
 \draw [>->] (B4) -- (A3);
 \draw [->>] (A3) -- (B3);
 \draw [>->] (B3) -- (A2);
\end{tikzpicture} \]
Now assume $\Ext^i_{\Lambda}(X, \widetilde{\Lambda}) = 0$ for any $i \in \{1, \ldots, n-1\}$. Then 
\[ \Ext^1_{\Lambda}(X, X_1) = \Ext^2_{\Lambda}(X, X_2) = \cdots = \Ext^{n-1}_{\Lambda}(X, X_{n-1}) = 0.\]
Hence the sequence $X_1 \mono M_0 \epi X$ splits, and we have $X \in \add \widetilde{\Lambda}$. This completes the proof that $\widetilde{\Lambda}$ is a right $n$-cluster tilting object.

Now let $T$ be any right $n$-cluster tilting object in $\mod \Lambda$. By definition we have $\Lambda \in \add T$. Now note that by Proposition~\ref{prop.prepare_right_cto}(1) and Lemma~\ref{lemma.ARformula_Ext} we have epimorphisms
\[ \Ext^{n-i}_{\Lambda}(T, \tau_n^{- \ell} \Lambda) \epi[30] D \Ext^i_{\Lambda}(\tau_n^{- (\ell + 1)} \Lambda, T) \]
for any $\ell \geqslant 0$. Hence we can see inductively that $\tau_n^{- \ell} \Lambda \in \add T$ for all $\ell \geqslant 0$. Hence $\widetilde{\Lambda} \in \add T$, and since $\widetilde{\Lambda}$ is also a right $n$-cluster tilting object $\add \widetilde{\Lambda} = \add T$.
\end{proof}

\begin{proposition} \label{prop.fingld}
Let $\Lambda$ be an algebra of global dimension at most $n$ satisfying the vosnex property. Then $\gld \End_{\Lambda}(\widetilde{\Lambda}) \leqslant n+1$ and $\gld \End_{\Lambda}(D \widetilde{\Lambda}) \leqslant n+1$.
\end{proposition}

\begin{proof}
We only prove the first claim, the second one is dual.

By Lemma~\ref{lemma.exactsequence}, for any $X \in \mod \Lambda$ there is an exact sequence
\[ M_{n-1} \mono[30] M_n \to[30] \cdots \to[30] M_0 \epi[30] X \]
with $M_i \in \add \widetilde{\Lambda}$. Since, by Proposition~\ref{prop.prepare_right_cto}(2) we have $\Ext_{\Lambda}^i(\widetilde{\Lambda}, \widetilde{\Lambda}) = 0 \, \forall i \in \{1, \ldots, n-1\}$ such a sequence is automatically a $\widetilde{\Lambda}$-resolution. Now the claim follows by a standard argument (e.g.\ \cite[Lemma~2.1]{EHIS}).
\end{proof}

\subsection{The Iwanaga-Gorenstein property}

\begin{definition}
An algebra (or category) is called \emph{Iwanaga-Gorenstein of dimension $d$} if all finitely generated projective left and right modules have injective dimension at most $d$.
\end{definition}

The next lemma shows that the categories we are interested in here are Iwanaga-Gorenstein of dimension at most $1$.

\begin{lemma} \label{lem.selinj_dim_1}
Let $\Lambda$ be an algebra of global dimension at most $n$, having the vosnex property. Then $\mathscr{U}\psL$ and $\widetilde{\Lambda}$ are Iwanaga-Gorenstein of dimension at most $1$.
\end{lemma}

\begin{proof}
By Lemma~\ref{lemma.equiv_vosnex}(1 \iff\ 2) we know that $\Lambda$ has the vosnex property if and only if $\mathbb{S} \mathscr{U}\psL \subseteq \mathscr{U}\psL * \mathscr{U}\psL[1]$. Hence for any $U \in \mathscr{U}\psL$ we have a triangle $U_1 \to U_0 \to \mathbb{S} U \to U_1[1]$. Now we apply $\Hom_{\mathscr{T}}(\mathscr{U}\psL, -)$ to this triangle, and note that $\Hom_{\mathscr{T}}(\mathscr{U}\psL, \mathbb{S}U[-1]) = D\Hom_{\mathscr{T}}(U, \mathscr{U}\psL[1]) = 0$. This gives a projective resolution of the injective $\mathscr{U}\psL$-module $\Hom_{\mathscr{T}}(-, \mathbb{S}U)$ of length $1$.

Dually, by Lemma~\ref{lemma.equiv_vosnex}(1 \iff\ 3) we obtain an injective resolution of length $1$ of projective $\mathscr{U}\psL$-modules.

The same proof works for $\widetilde{\Lambda}$.
\end{proof}

In the case of Iwanaga-Gorenstein algebras (or categories) we have the following natural replacement for (stable) module categories.

\begin{definition}
Let $\mathscr{U}$ be an Iwanaga-Gorenstein category. We denote the category of \emph{Cohen-Macaulay modules} over $\mathscr{U}$ by
\[ \CM(\mathscr{U}) = \{ X \in \mod \mathscr{U} \mid \Ext^i_{\mathscr{U}}(X, P) = 0 \, \forall i > 0 \text{ and any projective $\mathscr{U}$-module } P \}.\]
It is well-known that $\CM(\mathscr{U})$ forms a Frobenius category, and that it's stable category $\stabCM(\mathscr{U}) = \CM(\mathscr{U}) / (\proj \mathscr{U})$ is triangulated \cite{Ha}. The inverse translation and translation are given by $\Omega$ and $\Omega_{\CM}^-$, respectively. Here $\Omega$ denotes the ``usual'' syzygy functor, and $\Omega_{\CM}^- X$ is the cokernel of a left projective approximation of $X$.
\end{definition}

\subsection{The \texorpdfstring{$(n+1)$}{(n+1)}-Calabi-Yau property}

We have the following generalization of Theorem~\ref{theorem.serrestable} and the main result of \cite{KelRei}.

\begin{theorem} \label{theorem.serre_CM}
Let $\mathscr{T}$ be a triangulated category with an $n$-cluster tilting subcategory $\mathscr{U}$ satisfying the vosnex property. Then $\stabCM(\mathscr{U})$ is a triangulated category, and
\[ \leftsub{\stabCM(\mathscr{U})}{\mathbb{S}}_{n+1} = \leftsub{\mathscr{T}}{\mathbb{S}}_n \qquad \text{ on } \stabCM(\mathscr{U}). \]
\end{theorem}

\begin{corollary} \label{cor.stabCM_is_CY}
In the setup of Theorem~\ref{theorem.serre_CM} assume $\mathscr{T}$ is $n$-Calabi-Yau. Then $\stabCM(\mathscr{U})$ is $(n+1)$-Calabi-Yau.
\end{corollary}

\begin{remark}
Lemma~\ref{lem.selinj_dim_1} and Theorem~\ref{theorem.serre_CM} are obtained independently by Beligiannis \cite{Bel_stC}
\end{remark}

The rest of this subsection leads to the proof of Theorem~\ref{theorem.serre_CM}. We will calculate $\Omega^-_{\CM}$ by using the following observation.

\begin{lemma} \label{lemma.approx=approx}
Let $\mathscr{T}$ be a triangulated category, $\mathscr{U} \subseteq \mathscr{T}$. Assume $X \in \mathscr{T}$ is such that any map from $\Hom_{\mathscr{T}}(-, X)$ to a projective $\mathscr{U}$-module is representable. Then any left $\mathscr{U}$-approximation $X \to U$ gives rise to a left projective approximation $\Hom_{\mathscr{T}}(-, X) \to \Hom_{\mathscr{T}}(-, U)$.
\end{lemma}

\begin{proof}
This follows immediately from the definitions.
\end{proof}

To apply this, we need the following Yoneda-type result.

\begin{proposition} \label{prop.representable}
Let $\mathscr{T}$ be a triangulated category, and $\mathscr{U} \subseteq \mathscr{T}$ an arbitrary subcategory.
\begin{enumerate}
\item Assume $\Hom_{\mathscr{T}}(\mathscr{U}, \mathscr{U}[i]) = 0$ for any $i \in \{1, \ldots, m\}$ and $\Hom_{\mathscr{T}}(\mathscr{U}[i], \mathscr{U}) = 0$ for any $i \in \{1, \ldots, m\}$. Then for any 
\[ X \in \mathscr{U} * \mathscr{U}[1] * \cdots * \mathscr{U}[m] \]
and any $U \in \mathscr{U}$ the natural map
\[ \Hom_{\mathscr{T}}(X, U) \to[30] \Hom_{\mathscr{U}}(\Hom_{\mathscr{T}}(-, X), \Hom_{\mathscr{T}}(-, U)) \]
is an isomorphism.
\item Assume $\Hom_{\mathscr{T}}(\mathscr{U}, \mathscr{U}[i]) = 0$ for any $i \in \{1, \ldots, m\}$ and $\Hom_{\mathscr{T}}(\mathscr{U}[i], \mathscr{U}) = 0$ for any $i \in \{1, \ldots, m-1\}$. Then for any 
\[ X \in \mathscr{U} * \mathscr{U}[1] * \cdots * \mathscr{U}[m] \]
and any $U \in \mathscr{U}$ the natural map
\[ \Hom_{\mathscr{T}}(X, U) \to[30] \Hom_{\mathscr{U}}(\Hom_{\mathscr{T}}(-, X), \Hom_{\mathscr{T}}(-, U)) \]
is an epimorphism.
\end{enumerate}
\end{proposition}

\begin{proof}
We will use induction as follows: (1$_{m=0}$) and (2$_{m=0}$) are clear by the Yoneda Lemma. We show (1$_m$) \then\ $[$ (2$_{m+1}$) and (1$_{m+1}$) $]$.

Assume (1) it true for some $m$, and, in addition, the assumptions of (2$_{m+1}$) are satisfied. Let $X \in \mathscr{U} * \mathscr{U}[1] * \cdots * \mathscr{U}[m+1]$. Then there is a triangle
\[ Y \to[30] U_0 \to[30] X \to[30] Y[1] \]
with $U_0 \in \mathscr{U}$ and $Y \in \mathscr{U} * \mathscr{U}[1] * \cdots * \mathscr{U}[m]$. This gives rise to the exact sequence of $\mathscr{U}$-modules
\[ \Hom_{\mathscr{T}}(-, Y) \to[30] \Hom_{\mathscr{T}}(-, U_0) \to[30] \Hom_{\mathscr{T}}(-, X) \to[30] \Hom_{\mathscr{T}}(-, Y[1]). \]
The term $\Hom_{\mathscr{T}}(\mathscr{U}, Y[1])$ vanishes, since $\Hom_{\mathscr{T}}(\mathscr{U}, \mathscr{U}[i]) = 0$ for $i \in \{1, \ldots, m+1\}$. Applying $\Hom_{\mathscr{U}}(-, \Hom_{\mathscr{T}}(-,U))$ to this exact sequence, and $\Hom_{\mathscr{T}}(-, U)$ to the triangle above, we obtain the following commutative diagram.
\[ \begin{tikzpicture}[xscale=3.8,yscale=-2]
 \node (11) at (0,0) {$\Hom_{\mathscr{T}}(Y[1], U)$};
 \node (21) at (1,0) {$\Hom_{\mathscr{T}}(X, U)$};
 \node (31) at (2,0) {$\Hom_{\mathscr{T}}(U_0, U)$};
 \node (41) at (3,0) {$\Hom_{\mathscr{T}}(Y, U)$};
 \node (22) at (1,1) [rotate=15] {$\leftsub{\mathscr{U}}{(\leftsub{\mathscr{T}}{(-,X)}, \leftsub{\mathscr{T}}{(-, U)})}$};
 \node (32) at (2,1) [rotate=15] {$\leftsub{\mathscr{U}}{(\leftsub{\mathscr{T}}{(-,U_0)}, \leftsub{\mathscr{T}}{(-, U)})}$};
 \node (42) at (3,1) [rotate=15] {$\leftsub{\mathscr{U}}{(\leftsub{\mathscr{T}}{(-,Y)}, \leftsub{\mathscr{T}}{(-, U)})}$};
 \draw [->] (11) -- (21);
 \draw [->] (21) -- (31);
 \draw [->] (31) -- (41);
 \draw [>->] (22) -- (32);
 \draw [->] (32) -- (42);
 \draw [->] (21) -- (22);
 \draw [->] (31) -- node [left] {$\iso$} (32);
 \draw [->] (41) -- (42);
\end{tikzpicture} \]
By (1$_m$) the right vertical map is an isomorphism, hence the left vertical map is onto, as claimed in (2$_{m+1}$).

If moreover $\Hom_{\mathscr{T}}(\mathscr{U}[m+1], \mathscr{U}) = 0$, then $\Hom_{\mathscr{T}}(Y[1], U) = 0$. Hence the left vertical map is an isomorphism, as claimed in (1$_{m+1}$).
\end{proof}

\begin{corollary} \label{cor.representable}
Let $\mathscr{T}$ be a triangulated category with a Serre functor $\leftsub{\mathscr{T}}{\mathbb{S}}$. Let $\mathscr{U} \subseteq \mathscr{T}$ be an $n$-cluster tilting subcategory with $\leftsub{\mathscr{T}}{\mathbb{S}}\mathscr{U} \subseteq \mathscr{U} * \mathscr{U}[1] * \cdots * \mathscr{U}[\ell]$. Then
\begin{enumerate}
\item for any $X \in \mathscr{U} * \mathscr{U}[1] * \cdots * \mathscr{U}[n-1-\ell]$ any morphism from $\Hom_{\mathscr{T}}(-, X)$ to a projective $\mathscr{U}$-module $\Hom_{\mathscr{T}}(-, U)$ is uniquely representable by a morphism in $\Hom_{\mathscr{T}}(X,U)$.
\item If $\ell > 0$, then for any $X \in \mathscr{U} * \mathscr{U}[1] * \cdots * \mathscr{U}[n - \ell]$ any morphism from $\Hom_{\mathscr{T}}(-, X)$ to a projective $\mathscr{U}$-module $\Hom_{\mathscr{T}}(-, U)$ is representable by a morphism in $\Hom_{\mathscr{T}}(X,U)$.
\end{enumerate}
\end{corollary}

\begin{proof}
We have
\begin{align*}
D \Hom_{\mathscr{T}}(\mathscr{U}[i], \mathscr{U}) & = \Hom_{\mathscr{T}}(\mathscr{U}, \leftsub{\mathscr{T}}{\mathbb{S}} \mathscr{U}[i]) \\
& \subseteq \Hom_{\mathscr{T}}(\mathscr{U}, \mathscr{U}[i] * \cdots * \mathscr{U}[i + \ell]) \\
& = 0 \text{ for } 0 < i < n - \ell
\end{align*}
Then the claims follow from Proposition~\ref{prop.representable} by setting $m = n-1-\ell$ for (1), and $m = n-\ell$ for (2).
\end{proof}

We are now ready to prove the main result of this subsection.

\begin{proof}[Proof of Theorem~\ref{theorem.serre_CM}]
Fix $M \in \CM(\mathscr{U})$. Then $M$ is represented by some $X \in U*U[1]$.

\textsc{Step I:} We show $\leftsub{\mathscr{T}}{\mathbb{S}}_n^{-1} M \iso \tau_{\mod \mathscr{U}}^- \Hom_{\mathscr{T}}(-,X[n-1])$. Since $X \in \mathscr{U} * \mathscr{U}[1]$ there is a triangle
\[ U_1 \to[30] U_0 \to[30] X \to[30] U_1[1] \]
with $U_i \in \mathscr{U}$. This gives rise to the exact sequence
\[ \underbrace{\Hom_{\mathscr{T}}(-, U_0[n-1])}_{=0} \to[30] \Hom_{\mathscr{T}}(-, X[n-1]) \to[30] \Hom_{\mathscr{T}}(-, U_1[n]) \to[30] \Hom_{\mathscr{T}}(-, U_0[n]). \]
Note that the $\mathscr{U}$-modules $\Hom_{\mathscr{T}}(-, U_i[n]) = \Hom_{\mathscr{T}}(-, \leftsub{\mathscr{T}}{\mathbb{S}} (\leftsub{\mathscr{T}}{\mathbb{S}}_n^{-1} U_i))$ are injective, and
\[ \leftsub{D^{\rm b}(\mod \mathscr{U})}{\mathbb{S}}^{-1} (\Hom_{\mathscr{T}}(-, U_i[n])) = \Hom_{\mathscr{T}}(-, \leftsub{\mathscr{T}}{\mathbb{S}}^{-1}_n U_i) \]
are the corresponding projective $\mathscr{U}$-modules. Hence we obtain the exact sequence
\[\Hom_{\mathscr{T}}(-, \leftsub{\mathscr{T}}{\mathbb{S}}^{-1}_n U_1) \to[30] \Hom_{\mathscr{T}}(-, \leftsub{\mathscr{T}}{\mathbb{S}}^{-1}_n U_0) \epi[30] \tau^-_{\mod \mathscr{U}} \Hom_{\mathscr{T}}(-, X[n-1]).  \]
In other words, we have
\[ \leftsub{\mathscr{T}}{\mathbb{S}}^{-1}_n M \iso \tau^-_{\mod \mathscr{U}} \Hom_{\mathscr{T}}(-, X[n-1]). \]
\textsc{Step II:} We now show $M \iso \Omega^n \Hom_{\mathscr{T}}(-,X[n-1])$.

By Proposition~\ref{prop.ct-resolution}, we have triangles
\[ \begin{tikzpicture}[xscale=1.3,yscale=1.5]
 \node (X0) at (0,0) {$X^0$};
 \node (X1) at (2,0) {$X^1$};
 \node (X2) at (4,0) {$X^2$};
 \node (X8) at (8,0) {$X^{n-2}$};
 \node (X9) at (10,0) {$U^{n-1}$};
 \node (U0) at (1,1) {$U^0$};
 \node (U1) at (3,1) {$U^1$};
 \node (U2) at (5,1) {$U^2$};
 \node (U7) at (7,1) {$U^{n-3}$};
 \node (U8) at (9,1) {$U^{n-2}$};
 \node at (6,.5) {$\cdots$};
 \draw [->] (U0) -- (U1);
 \draw [->] (U1) -- (U2);
 \draw [->] (U7) -- (U8);
 \draw [->] (X0) -- (U0);
 \draw [->] (X1) -- (U1);
 \draw [->] (X2) -- (U2);
 \draw [->] (X8) -- (U8);
 \draw [->] (U0) -- (X1);
 \draw [->] (U1) -- (X2);
 \draw [->] (U7) -- (X8);
 \draw [->] (U8) -- (X9);
 \draw [->] (X1) -- node [pos=.15] {\tikz{\draw[-] (0,.15) -- (0,-.15);}} (X0);
 \draw [->] (X2) -- node [pos=.15] {\tikz{\draw[-] (0,.15) -- (0,-.15);}} (X1);
 \draw [->] (X9) -- node [pos=.15] {\tikz{\draw[-] (0,.15) -- (0,-.15);}} (X8);
 \node (X0') at (-.8,0) {$X$};
 \draw [double distance=1.5pt] (X0') -- (X0);
\end{tikzpicture} \]
with $U_i \in \mathscr{U}$. Note that, since $\Hom_{\mathscr{T}}(\mathscr{U}, \mathscr{U}[i]) = 0 \, \forall i \in \{1, \ldots, n-1\}$, the maps $X^i \to U^i$ above are left $\mathscr{U}$-approximations.

Clearly we have $X^i \in \mathscr{U} * \cdots * \mathscr{U}[i+1]$. Hence, by Corollary~\ref{cor.representable}, any map from $\Hom_{\mathscr{T}}(-, X^i)$ to projective $\mathscr{U}$-modules is representable. By Lemma~\ref{lemma.approx=approx}, the maps $X^i \to U^i$ above induce left approximations by projectives
\[ \Hom_{\mathscr{T}}(-, X^i) \to[30] \Hom_{\mathscr{T}}(-, U^i). \]
Since $\Hom_{\mathscr{T}}(-, X^i[1]) = 0$ for $i < n-2$ the maps
\[ \Hom_{\mathscr{T}}(-, U^i) \to[30] \Hom_{\mathscr{T}}(-, X^{i+1}) \]
are onto for $i < n-2$. Thus we have the following diagram of $U$-modules:
\[ \begin{tikzpicture}[xscale=1.3,yscale=1.5]
 \node (X0) at (0,0) {$_{\mathscr{T}}(-,X^0)$};
 \node (X1) at (2,0) {$_{\mathscr{T}}(-,X^1)$};
 \node (X2) at (4,0) {$_{\mathscr{T}}(-,X^2)$};
 \node (X8) at (8,0) {$_{\mathscr{T}}(-,X^{n-2})$};
 \node (X9) at (10,0) {$_{\mathscr{T}}(-,U^{n-1})$};
 \node (U0) at (1,1) {$_{\mathscr{T}}(-,U^0)$};
 \node (U1) at (3,1) {$_{\mathscr{T}}(-,U^1)$};
 \node (U2) at (5,1) {$_{\mathscr{T}}(-,U^2)$};
 \node (U7) at (7,1) {$_{\mathscr{T}}(-,U^{n-3})$};
 \node (U8) at (9,1) {$_{\mathscr{T}}(-,U^{n-2})$};
 \node at (6,.5) {$\cdots$};
 \draw [->] (U0) -- (U1);
 \draw [->] (U1) -- (U2);
 \draw [->] (U7) -- (U8);
 \draw [->] (X0) -- node [left] {ap} (U0);
 \draw [->] (X1) -- node [left] {ap} (U1);
 \draw [->] (X2) -- node [left] {ap} (U2);
 \draw [->] (X8) -- node [left] {ap} (U8);
 \draw [->>] (U0) -- (X1);
 \draw [->>] (U1) -- (X2);
 \draw [->>] (U7) -- (X8);
 \draw [->] (U8) -- (X9);
 \node (E) at (7.7,-.6) {$_{\mathscr{T}}(-,X^{n-2}[1])$};
 \draw [->>] (X9) -- (E);
\end{tikzpicture} \]
Therefore we can see inductively that for any $i \in \{0, \ldots, n-2\}$
\[ \Hom_{\mathscr{T}}(-,X^i) \iso \Omega_{\CM}^{-i} M \]
and that the maps
\[ \Hom_{\mathscr{T}}(-,X^i) \to[30] \Hom_{\mathscr{T}}(-, U^i) \]
are mono. In particular we have
\[ \Omega^n \Hom_{\mathscr{T}} (-,X^{n-2}[1]) \iso \Hom_{\mathscr{T}}(-,X_0) = M.\]
Now the assertion follows from
\[ \Hom_{\mathscr{T}}(-,X^{n-2}[1]) \iso \Hom_{\mathscr{T}}(-,X^{n-3}[2]) \iso \cdots \iso \Hom_{\mathscr{T}}(-,X^0[n-1]). \]
(It should also be noted that the above isomorphic $\mathscr{U}$-modules do not have to lie in $\CM(\mathscr{U})$.)

\textsc{Step III:} We now prove the claim of the theorem.

Let $N \in \CM(\mathscr{U})$. Then we have
\begin{align*}
& D\underline{\Hom}_{\mathscr{U}}( \leftsub{\mathscr{T}}{\mathbb{S}}_n^{-1}M, N) \\
& = \Ext^1_{\mathscr{U}}(N, \tau_{\mod \mathscr{U}} \leftsub{\mathscr{T}}{\mathbb{S}}_n^{-1} M) && \text{(by AR duality)} \\
& = \underline{\Hom}_{\mathscr{U}}(\Omega^{n+1} N, \Omega^n \tau_{\mod \mathscr{U}} \leftsub{\mathscr{T}}{\mathbb{S}}_n^{-1} M) && \text{(since $N \in \CM(\mathscr{U})$)} \\
& = \underline{\Hom}_{\mathscr{U}}( \Omega^{n+1} N, M) && \text{(by Steps I and II)}
\end{align*}
Hence $\Omega^{-(n+1)}_{\stabCM(\mathscr{U})} \circ \leftsub{\mathscr{T}}{\mathbb{S}}_n$ is the Serre functor $\leftsub{\stabCM(\mathscr{U})}{\mathbb{S}}$ on $\stabCM(\mathscr{U})$. In other words
\[ \leftsub{\stabCM}{\mathbb{S}}_{n+1} = \Omega^{n+1} \circ \leftsub{\stabCM(\mathscr{U})}{\mathbb{S}} = \leftsub{\mathscr{T}}{\mathbb{S}}_n. \qedhere\]
\end{proof}

\subsection{Identifying \texorpdfstring{$\stabCM(\mathscr{U}\psL)$}{\_CM\_ U} with the derived category of \texorpdfstring{$\Gamma$}{Gamma}}

In this subsection we restrict our attention back to the basic setup of this section: We have an algebra $\Lambda$ of global dimension at most $n$, satisfying the vosnex property (see \ref{notation.vosnex}). By Lemma~\ref{lem.selinj_dim_1} we know that $\mathscr{U}\psL$ is Iwanaga-Gorenstein of dimension at most $1$. Our aim in this subsection is to identify $\stabCM(\mathscr{U})$ as the derived category of some algebra $\Gamma$. More precisely we will show the following.

\begin{theorem} \label{thm.CM_is_derived}
In the setup above
\[ \mathscr{D}_{\Gamma} \approx \stabCM(\mathscr{U}\psL), \]
with $\Gamma = \underline{\End}_{\Lambda}(\widetilde{\Lambda})$.
\end{theorem}
We will give two equivalences. The first uses a tilting object $T$ (see Construction~\ref{const.tiltingobj} and Theorem~\ref{thm.equiv_via_tilting}), while the latter uses restriction along some functor $\mathscr{U}\psL \to \proj \Gamma$ (see Construction~\ref{const.for_res} and Theorem~\ref{thm.equiv_via_res}).

Note that Theorem~\ref{thm.CM_is_derived} generalizes Section~\ref{subsect.stab_is_Db}, and specifically Corollary~\ref{corollary.stabmodisderived}(1).

\subsubsection{A tilting object}

We set $\mathscr{U}\psL^i = \add \leftsub{\mathscr{U}\psL}{\mathbb{S}}_n^i \Lambda$ for $i \in \mathbb{Z}$. Then $\mathscr{U} = \bigvee_{i \in \mathbb{Z}} \mathscr{U}\psL^i$. We write $\mathscr{U}\psL^{< \ell} = \bigvee_{i < \ell} \mathscr{U}\psL^i, \mathscr{U}\psL^{> \ell} = \bigvee_{i > \ell} \mathscr{U}\psL^i$ and similar variations. For $X \in \mod \mathscr{U}\psL$ we set
\[ \Supp X = \add \{ U \in \mathscr{U}\psL \text{ indecomposable} \mid X(U) \neq 0 \}. \]
We observe that $\Hom_{\mathscr{U}\psL}(X, Y) = 0$ whenever $X, Y \in \mod \mathscr{U}\psL$ with $\Supp X \cap \Supp Y = \emptyset$.

We have an equivalence $\mathscr{U} / \mathscr{U}^{\geqslant \ell} \approx \mathscr{U}^{<\ell}$, which induces natural functors $\mathscr{U} \to \mathscr{U}^{<\ell}$ and $\mod \mathscr{U}^{<\ell} \to \mod \mathscr{U}$.  Thus we may identify
\[ \mod \mathscr{U}^{<\ell} = \{X \in \mod \mathscr{U} \mid \Supp X \subseteq \mathscr{U}^{< \ell} \} \overset{\mathclap{\text{full}}}{\subseteq} \mod \mathscr{U}. \]
Similarly we identify
\[ \mod \mathscr{U}^{>\ell} = \{X \in \mod \mathscr{U} \mid \Supp X \subseteq \mathscr{U}^{>\ell}\} \overset{\mathclap{\text{full}}}{\subseteq} \mod \mathscr{U}. \]

\begin{remark} \label{remark.pic_U}
The following picture describes the distribution of the $\mathscr{U}^i$ in $\mathscr{D}$ in comparison to the standard t-structure.
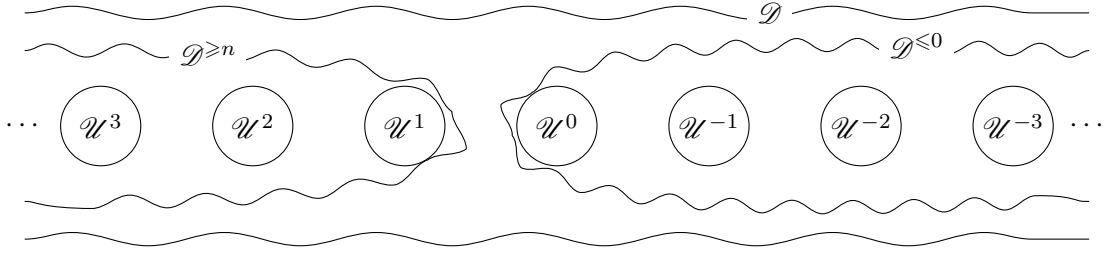
\[ \begin{tikzpicture}[yscale=-1]
 \draw [decorate,decoration=snake,segment length=2cm] (0,0) -- node [pos=.7,fill=white] {$\mathscr{D}$} (14,0);
 \draw [decorate,decoration=snake,segment length=2cm] (0,3) -- (14,3);
 \draw [decorate,decoration=snake,segment length=.8cm,looseness=.5] (14,.5) [out=180,in=270] to node [pos=.3,fill=white] {$\mathscr{D}^{\leqslant 0}$} (6.5,1.5) [out=90,in=180] to (14,2.5);
 \draw [decorate,decoration=snake,segment length=1cm,looseness=.5] (0,.5) [out=0,in=270] to node [pos=.4,fill=white] {$\mathscr{D}^{\geqslant n}$} (5.5,1.5) [out=90,in=0] to (0,2.5);
 \foreach \i in {3,2,...,-3}
  {
   \draw (7 - 2 * \i, 1.5) circle (15pt);
   \node at (7 - 2* \i, 1.5) {$\mathscr{U}^{\i}$};
  }
 \node at (0,1.5) {$\cdots$};
 \node at (14,1.5) {$\cdots$};
\end{tikzpicture} \]
That is, we have
\[ \mathscr{U}^{\leqslant 0} \subseteq \mathscr{D}^{\leqslant 0} \qquad \text{and} \qquad \mathscr{U}^{> 0}  \subseteq \mathscr{D}^{\geqslant n}. \]
In particular note that $\Hom_{\mathscr{D}}(\mathscr{U}^{\leqslant 0}, \mathscr{U}^{> 0}) = 0$, and hence by shifting $\Hom_{\mathscr{D}}(\mathscr{U}^{\leqslant \ell}, \mathscr{U}^{> \ell}) = 0$ for any $\ell \in \mathbb{Z}$.
\end{remark}

\begin{observation} \label{obs.geq_<}
For $X \in \mod \mathscr{U}\psL$ there is a functorial short exact sequence
\[ X^{\geqslant 0} \mono[30] X \epi[30] X^{< 0} \]
with $X^{\geqslant 0} \in \mod \mathscr{U}\psL^{\geqslant 0}$ and $X^{< 0} \in \mod \mathscr{U}\psL^{< 0}$. (This functorial exact sequence should not be confused with the functorial triangle in \ref{rem.t-trunc-triang})
\end{observation}

The aim of this subsection is to establish the following proposition. Together with the next subsection it will provide a proof of Theorem~\ref{thm.CM_is_derived} above.

For compacter notation we write $P_U = \Hom_{\mathscr{D}_{\Lambda}}(-, U)$ for the projective $\mathscr{U}\psL$-module corresponding to $U$.

\begin{construction} \label{const.tiltingobj}
We set
\[ V := \bigoplus_{\substack{U \in \mathscr{U}\psL^{< 0} \\ \text{indec.} \\ \Supp P_U \not\subseteq \mathscr{U}\psL^{< 0}}} \in \mathscr{U}\psL, \qquad \text{and} \qquad T := P_V^{\geqslant 0}. \]
Note that the direct sum is finite since there are only finitely many indecomposable objects $U \in \mathscr{U}\psL^{<0}$ satisfying $\Supp P_U \not\subseteq \mathscr{U}\psL^{< 0}$ (this is equivalent to $\Hom_{\mathscr{D}_{\Lambda}}(\mathscr{U}\psL^{\geqslant 0}, U) \neq 0$).

Note that since $T = P_V^{\geqslant 0} = \Omega P_V^{< 0}$ (see the defining sequence in Observation~\ref{obs.geq_<}) we have $T \in \CM(\mathscr{U}\psL)$.
\end{construction}

\begin{theorem} \label{thm.equiv_via_tilting}
The object $T$ as in Construction~\ref{const.tiltingobj} above has endomorphism ring $\Gamma$, and the composition
\[ \mathscr{D}_{\Gamma} \tol[40]{\scriptstyle T \otimes_{\Gamma}^L -} \mathscr{D}_{\mathscr{U}\psL} \to[30] \stabCM(\mathscr{U}\psL) \]
is an equivalence.
\end{theorem}

The first step towards the proof of this theorem is the following result.

\begin{proposition} \label{prop.is_tilting_object}
The object $T$ as in Construction~\ref{const.tiltingobj} is a tilting object in $\stabCM(\mathscr{U}\psL)$. (That is, we have $\underline{\Hom}_{\mathscr{U}}(T, \Omega_{\CM}^i T) = 0 \, \forall i \neq 0$, and $\forall X \in \stabCM(\mathscr{U}) \setminus \{0\}$ there is $i \in \mathbb{Z}$ such that $\underline{\Hom}_{\mathscr{U}}(T, \Omega_{\CM}^i X) \neq 0$.)
\end{proposition}

We prove this theorem in several shorter lemmas. First, in the Lemmas~\ref{lemma.no_neg_selfext} and \ref{lemma.no_pos_selfext}, we check that $T$ has neither negative nor positive self-extensions. Then, in Proposition~\ref{prop.T_generates}, we show that $T$ generates $\stabCM(\mathscr{U}\psL)$.

The following observation will be used throughout what follows.

\begin{observation} \label{obs.geq_leq}
Let $\ell \in \mathbb{Z}$. For any $X \in \mod \mathscr{U}^{\geqslant \ell}$ also $\Omega X \in \mod \mathscr{U}^{\geqslant \ell}$.
\end{observation}

\begin{lemma} \label{lemma.no_neg_selfext}
In the setup above we have
\[ \underline{\Hom}_{\mathscr{U}\psL}(P_U^{\geqslant 0}, \Omega^{-i}_{\CM} P_{U'}^{\geqslant 0}) = \Ext^i_{\mathscr{U}\psL}(P_U^{\geqslant 0}, P_{U'}^{\geqslant 0}) = 0 \quad \forall i > 0 \]
for any $U$ and $U'$ in $\mathscr{U}\psL$.
\end{lemma}

\begin{proof}
The first equality holds by general theory of stable categories of Frobenius categories.

For the second one, note that, since $\Supp \Omega^i P_U^{< 0} = \Supp \Omega^{i-1} P_U^{\geqslant 0} \subseteq \mathscr{U}\psL^{\geqslant 0}$ and $\Supp P_{U'}^{< 0} \subseteq \mathscr{U}\psL^{< 0}$ we have $\Hom_{\mathscr{U}\psL}(\Omega^i P_U^{< 0}, P_{U'}^{<0}) = 0$. Therefore also $\Ext^i_{\mathscr{U}\psL}(P_U^{< 0}, P_{U'}^{<0}) = 0$ for any $i > 0$. Now the short exact sequence $P_{U'}^{\geqslant 0} \mono P_{U'} \epi P_{U'}^{< 0}$ gives rise to a sequence
\[ \Ext^{i-1}_{\mathscr{U}\psL}(P_U^{\geqslant 0}, P_{U'}^{< 0}) \to[30] \Ext^i_{\mathscr{U}\psL}(P_U^{\geqslant 0}, P_{U'}^{\geqslant 0}) \to[30] \underbrace{\Ext^i_{\mathscr{U}\psL}(P_U^{\geqslant 0}, P_{U'})}_{ = 0} . \]
For $i > 1$ we have $\Ext^{i-1}_{\mathscr{U}\psL}(P_U^{\geqslant 0}, P_{U'}^{< 0}) = \Ext^i_{\mathscr{U}\psL}(P_U^{< 0}, P_{U'}^{< 0}) = 0$, and for $i = 1$ we have $\Hom_{\mathscr{U}\psL}(P_U^{\geqslant 0}, P_{U'}^{< 0}) = 0$ since the supports are disjoint. Therefore, for any $i > 0$ we have $\Ext^i_{\mathscr{U}\psL}(P_U^{\geqslant 0}, P_{U'}^{\geqslant 0}) = 0$.
\end{proof}

\begin{lemma} \label{lem.shift_of_support}
In the setup above we have $\Omega^i(\mod \mathscr{U}^{\geqslant 0}) \subseteq \mod \mathscr{U}\psL^{> 0}$ for any $i > n$.
\end{lemma}

\begin{proof}
Let $X \in \mod \mathscr{U}^{\geqslant 0}$. We consider the short exact sequence $X^{> 0} \mono X \epi X^{\leqslant 0}$ (similar to Observation~\ref{obs.geq_leq}). 

Let $Q_n \mono Q_{n-1} \to \cdots \to Q_0 \epi X(\Lambda)$ be a projective resolution of $X(\Lambda)$ as a $\Lambda$-module. Since $X \in \mod \mathscr{U}^{\geqslant 0}$ we have $\Supp X^{\leqslant 0} \subseteq \mathscr{U}^0 = \add \Lambda$. Thus we have a complex
\[ 0 \to[30] P_{Q_n} \to[30] \cdots \to[30] P_{Q_0} \to[30] X^{\leqslant 0} \]
of $\mathscr{U}$-modules. Now the morphism $P_{Q_0} \to X^{\leqslant 0}$ factors through $X \epi X^{\leqslant 0}$. The composition $P_{Q_1} \to P_{Q_0} \to X$ vanishes, since it vanishes on $\add \Lambda \subseteq \mathscr{U}$. Hence we have a complex
\begin{equation} \label{eq.complex_supportshift} 0 \to[30] P_{Q_n} \to[30] \cdots \to[30] P_{Q_0} \to[30] X \to[30] 0. \end{equation}
Evaluating at $\Lambda$ the complex \eqref{eq.complex_supportshift} becomes an exact sequence, and evaluating at $\mathscr{U}^{<0}$ all terms vanish. Thus all homologies of \eqref{eq.complex_supportshift} belong to $\mod \mathscr{U}^{>0}$. Now the claim of the lemma follows from $\Omega(\mod \mathscr{U}^{>0}) \subseteq \mod \mathscr{U}^{>0}$.
\end{proof}

\begin{lemma} \label{lemma.no_pos_selfext}
In the setup above we have
\[ \underline{\Hom}_{\mathscr{U}\psL}(P_U^{\geqslant 0}, \Omega^i P_{U'}^{\geqslant 0}) = 0 \quad \forall i > 0 \]
for any $U$ and $U'$ in $\mathscr{U}\psL$.
\end{lemma}

\begin{proof}
We have
\begin{align*}
\underline{\Hom}_{\mathscr{U}\psL}(P_U^{\geqslant 0}, \Omega^i P_{U'}^{\geqslant 0}) & = D \underline{\Hom}_{\mathscr{U}\psL}(\Omega^i P_{U'}^{\geqslant 0}, \leftsub{\stabCM(\mathscr{U}\psL)}{\mathbb{S}} P_U^{\geqslant 0}) \\
& = D \underline{\Hom}_{\mathscr{U}\psL}(\Omega^i P_{U'}^{\geqslant 0}, \Omega_{\CM}^{-(n+1)} \leftsub{\stabCM(\mathscr{U}\psL)}{\mathbb{S}}_{n+1} P_U^{\geqslant 0}) \\
& = D \underline{\Hom}_{\mathscr{U}\psL}(\Omega^{i+n+1} P_{U'}^{\geqslant 0}, \leftsub{\mathscr{U}\psL}{\mathbb{S}}_n P_U^{\geqslant 0}) && \text{(by Theorem~\ref{theorem.serre_CM})} \\
& = D \underline{\Hom}_{\mathscr{U}\psL}(\Omega^{i+n+1} P_{U'}^{\geqslant 0}, P_{(\leftsub{\mathscr{U}\psL}{\mathbb{S}}_n U)}^{\geqslant 1}) && \text{(by $\leftsub{\mathscr{U}}{\mathbb{S}}_n P_U^{\geqslant 0} = P_{\leftsub{\mathscr{U}}{\mathbb{S}}_n U}^{\geqslant 1}$)}\\
& = D \Ext^{i+n+1}_{\mathscr{U}\psL}(P_{U'}^{\geqslant 0}, P_{(\leftsub{\mathscr{U}\psL}{\mathbb{S}}_n U)}^{\geqslant 1}).
\end{align*}
The short exact sequence $P_{(\leftsub{\mathscr{U}\psL}{\mathbb{S}}_n U)}^{\geqslant 1} \mono P_{(\leftsub{\mathscr{U}\psL}{\mathbb{S}}_n U)} \epi P_{(\leftsub{\mathscr{U}\psL}{\mathbb{S}}_n U)}^{< 1}$ gives rise to an exact sequence
\[ \Ext^{i+n}_{\mathscr{U}\psL}(P_{U'}^{\geqslant 0}, P_{(\leftsub{\mathscr{U}\psL}{\mathbb{S}}_n U)}^{< 1}) \to[30] \Ext^{i+n+1}_{\mathscr{U}\psL}(P_{U'}^{\geqslant 0}, P_{(\leftsub{\mathscr{U}\psL}{\mathbb{S}}_n U)}^{\geqslant 1}) \to[30] \underbrace{\Ext^{i+n+1}_{\mathscr{U}\psL}(P_{U'}^{\geqslant 0}, P_{(\leftsub{\mathscr{U}\psL}{\mathbb{S}}_n U)})}_{= 0}. \]
Thus the proof is complete if we can show that $\Ext^{i+n}_{\mathscr{U}\psL}(P_{U'}^{\geqslant 0}, P_{(\leftsub{\mathscr{U}\psL}{\mathbb{S}}_n U)}^{< 1}) = 0$. Indeed, by Lemma~\ref{lem.shift_of_support} we know $\Omega^{n+i} P_{U'}^{\geqslant 0} \in \mod \mathscr{U}^{> 0}$ for any $i > 0$, and hence $\Hom_{\mathscr{U}\psL}(\Omega^{n+i} P_{U'}^{\geqslant 0}, P_{(\leftsub{\mathscr{U}\psL}{\mathbb{S}}_n U)}^{< 1}) = 0$ by looking at supports. Consequently also $\Ext^{i+n}_{\mathscr{U}\psL}(P_{U'}^{\geqslant 0}, P_{(\leftsub{\mathscr{U}\psL}{\mathbb{S}}_n U)}^{< 1}) = 0$.
\end{proof}

To complete the proof of Proposition~\ref{prop.is_tilting_object}, it remains to show that $T$ generates the category $\stabCM(\mathscr{U}\psL)$.

\begin{proposition} \label{prop.T_generates}
Let $X \in \CM(\mathscr{U}\psL)$ be indecomposable and non-projective. Then there is $U \in \mathscr{U}\psL$ and $n \in \mathbb{Z}$ such that $\underline{\Hom}_{\mathscr{U}\psL}(P_U^{\geqslant 0}, \Omega^n_{\CM} X) \neq 0$.
\end{proposition}

For the proof we need the following preparations.

\begin{observation} \label{obs.syzygy_supp_bound}
Let $m > 0$ such that $\tau_n^{-m} \Lambda = 0$. (Note that such an $m$ always exists since $\Lambda$ is assumed to be $\tau_n$-finite.) Clearly we then have the following:
\begin{enumerate}
\item For $U \in \mathscr{U}^{< r}$ we have $\Supp P_U \subseteq \mathscr{U}^{< r+m}$.
\item For $X \in \mod \mathscr{U}^{< r}$ we have $\Supp \Omega X \subseteq \mathscr{U}^{< r+m}$.
\end{enumerate}
\end{observation}

\begin{lemma} \label{lem.susp_moves}
Let $X \in \CM(\mathscr{U})$. Then $\Omega^i X \in \mod \mathscr{U}^{> 0}$ and $\Omega_{\CM}^{-i} X \in \mod \mathscr{U}^{<0}$ for $i \gg 0$.
\end{lemma}

\begin{proof}
The first claim follows by repeated application of Lemma~\ref{lem.shift_of_support}.

For the proof of the latter assertion let $m$ be as in Observation~\ref{obs.syzygy_supp_bound}. Let $r > 0$ such that $X \in \mod \mathscr{U}^{< r}$. Then
\[ \underline{\Hom}_{\mathscr{U}}(\mod \mathscr{U}^{\geqslant -2m} \cap \CM(\mathscr{U}), \Omega_{\CM}^{-i} X) = \underline{\Hom}_{\mathscr{U}}(\Omega^i(\mod \mathscr{U}^{\geqslant -2m} \cap \CM(\mathscr{U})), X). \]
Note that, by Lemma~\ref{lem.shift_of_support}, for $i \gg 0$ we have
\[ \Omega^i(\mod \mathscr{U}^{\geqslant -2m} \cap \CM(\mathscr{U})) \subseteq \mod \mathscr{U}^{\geqslant r} \cap \CM(\mathscr{U}). \]
Thus, since $X \in \mod \mathscr{U}^{< r}$, the above $\Hom$-space vanishes. This means that the inclusion $(\Omega_{\CM}^{-i} X)^{\geqslant -2m} \sub \Omega_{\CM}^{-i} X$ factors through a projective $\mathscr{U}$-module. Let $P_U \epi \Omega_{\CM}^{-i} X$ be a projective cover. The left commutative triangle below gives rise to the right one.
\[ \begin{tikzpicture}[xscale=3,yscale=1.5]
 \node (A) at (0,0) {$(\Omega_{\CM}^{-i} X)^{\geqslant -2m}$};
 \node (B) at (1,0) {$\Omega_{\CM}^{-i} X$};
 \node (C) at (1,-1) {$P_U$};
 \node (D) at (2,0) {$(\Omega_{\CM}^{-i} X)^{< -2m}$};
 \draw [right hook->] (A) -- (B);
 \draw [->>] (B) -- (D);
 \draw [->>] (C) -- (B);
 \draw [dashed,->] (A) -- (C);
 \node (A) at (3,0) {$(\Omega_{\CM}^{-i} X)^{\geqslant -2m}$};
 \node (B) at (4,0) {$(\Omega_{\CM}^{-i} X)^{\geqslant -2m}$};
 \node (C) at (4,-1) {$P_U^{\geqslant -2m}$};
 \draw [->] (A) -- node [above] {id} (B);
 \draw [->>] (C) -- (B);
 \draw [dashed,->] (A) -- (C);
\end{tikzpicture} \]
Thus $(\Omega_{\CM}^{-i} X)^{\geqslant -2m}$ is a direct summand of $P_U^{\geqslant -2m}$, and we can write $(\Omega_{\CM}^{-i} X)^{\geqslant -2m} = P_{U_1} \oplus P_{U_2}^{\geqslant -2m}$ with $U_1 \in \mathscr{U}^{\geqslant -2m}$ and $U_2 \in \mathscr{U}^{< -2m}$. Since $\Omega_{\CM}^{-i} X$ does not contain any projective direct summands we have $\Ext_{\mathscr{U}}^1((\Omega_{\CM}^{-i} X)^{< -2m}, P_{U_1'}) \neq 0$ for any non-zero direct summand $U_1'$ of $U_1$. Hence $\Hom_{\mathscr{U}}(\Omega((\Omega_{\CM}^{-i} X)^{< -2m}), P_{U_1'}) \neq 0$ for any such $U_1'$. Since $\Supp \Omega ((\Omega_{\CM}^{-i} X)^{< -2m}) \subseteq \mathscr{U}^{< -m}$ (see Observation~\ref{obs.syzygy_supp_bound}(2)) we have $U_1 \in \mathscr{U}^{< -m}$. So we have $U_1 \oplus U_2 \in \mathscr{U}^{< -m}$, and hence $\Supp P_{U_1} \oplus P_{U_2} \subseteq \mathscr{U}^{< 0}$ by Observation~\ref{obs.syzygy_supp_bound}(1). Thus
\[ \Supp \Omega_{\CM}^{-i} X \subseteq \Supp (P_{U_1} \oplus P_{U_2}) \cup \mathscr{U}^{< -2m} \subseteq \mathscr{U}^{< 0}. \qedhere \]
\end{proof}

\begin{proof}[Proof of Proposition~\ref{prop.T_generates}]
Assume conversely, that there is an indecomposable non-projective $X \in \CM(\mathscr{U}\psL)$ such that $\underline{\Hom}_{\mathscr{U}\psL}(P_U^{\geqslant 0}, \Omega^n_{\CM} X) = 0$ for any $U \in \mathscr{U}\psL$ and $n \in \mathbb{Z}$. Possibly replacing $X$ by a suspension we may assume $X \in \mod \mathscr{U}\psL^{< 0}$ (by Lemma~\ref{lem.susp_moves}).

We choose a minimal projective resolution
\[ \cdots \to[30] P_{U_2} \to[30] P_{U_1} \to[30] P_{U_0} \to[30] X \to[30] 0. \]
We will show that the induced exact sequence
\begin{equation} \label{eq.lower_part_sequence}
\cdots \to[30] P_{U_2}^{\geqslant 0} \tol[30]{\scriptstyle d_1} P_{U_1}^{\geqslant 0} \tol[30]{\scriptstyle d_0} P_{U_0}^{\geqslant 0} \tol[30]{\scriptstyle d_{-1}} 0 \to[30] 0
\end{equation}
is a direct sum of complexes of the form $\cdots \to 0 \to H \to H \to 0 \to \cdots$ with $H \in \mod \mathscr{U}\psL$. For $i$ sufficiently large we have $P_{U_i} = P_{U_i}^{\geqslant 0}$ by Lemma~\ref{lem.susp_moves}. This is a contradiction, since in the upper complex all maps are in the radical.

Since Sequence~\eqref{eq.lower_part_sequence} is exact, we only have to show that the map from $P_{U_i}^{\geqslant 0}$ to the image $\Im d_{i-1} \subseteq P_{U_{i-1}}^{\geqslant 0}$ is a split epimorphism for any $i \geqslant 0$. We show this by induction on $i$. The case $i=0$ is clear. Now consider the following diagram.
\[ \begin{tikzpicture}[xscale=2,yscale=1.5]
 \node (A) at (0,0) {$P_{U_{i+1}}$};
 \node (B) at (1,.5) {$\Omega^i X$};
 \node (C) at (2,0) {$P_{U_i}$};
 \node (D) at (0,-1) {$P_{U_{i+1}}^{\geqslant 0}$};
 \node (E) at (1,-.5) {$\Im d_i$};
 \node (F) at (2,-1) {$P_{U_i}^{\geqslant 0}$};
 \draw [->>] (A) -- (B);
 \draw [right hook->] (B) -- (C);
 \draw [->>] (D) -- (E);
 \draw [right hook->] (E) -- (F);
 \draw [right hook->] (D) -- (A);
 \draw [right hook->] (E) -- (B);
 \draw [right hook->] (F) -- (C);
 \draw [dashed,->] (E) -- (A);
 \draw [dashed,->,bend right=30] (E) to (D);
 \node at (1,0) [fill=white,inner sep=2pt] {};
 \draw [->] (A) -- (C);
 \draw [->] (D) -- (F);
\end{tikzpicture} \]
By inductive assumption the map $P_{U_i}^{\geqslant 0} \epi \Im d_{i-1}$ splits, and hence $\Im d_i$ is a direct summand of $P_{U_i}^{\geqslant 0}$. Hence, by the assumption $\underline{\Hom}_{\mathscr{U}\psL}(P_U^{\geqslant 0}, \Omega^i X) = 0 \, \forall U, i$ we have $\underline{\Hom}_{\mathscr{U}\psL}(\Im d_i, \Omega^i X) = 0$. Therefore the middle vertical inclusion above factors through a projective module. Hence the dashed diagonal map exists, making the triangle above it commutative. The lower dashed map in the diagram, making the triangle including both dashed maps commutative, exists, since clearly $\Im d_i$ (being a submodule of $P_{U_i}^{\geqslant 0}$) can only map to the part of $P_{U_{i+1}}$ which is supported in $\mathscr{U}\psL^{\geqslant 0}$. Hence the map $P_{U_{i+1}}^{\geqslant 0} \to \Im d_i$ is a split epimorphism.
\end{proof}

\subsubsection{The endomorphism algebra of the tilting object}

\begin{theorem} \label{thm.identify_endoring}
Let $T = P_V^{\geqslant 0}$ as in Construction~\ref{const.tiltingobj}. Then
\[ \underline{\End}_{\mathscr{U}\psL}(T) = \End_{\mathscr{U}\psL}(T) = \End_{\Lambda}(\widetilde{\Lambda}_P) = \underline{\End}_{\Lambda}(\widetilde{\Lambda}) \overset{\mathclap{\text{def}}}{=} \Gamma, \]
where $\widetilde{\Lambda}_P$ denotes the maximal summand of the $\Lambda$-module $\widetilde{\Lambda}$ without non-zero projective direct summands.
\end{theorem}

The proof will rely on the following proposition.

\begin{proposition} \label{prop.adjoints}
\begin{enumerate}
\item The exact restriction functor
\[ \mathtt{res} \colon \mod \mathscr{U}\psL \to[30] \mod \Lambda \colon X \mapsto[30] X(\Lambda) \]
has a right adjoint, which is given by
\[ \mathtt{G}_* \colon \mod \Lambda \to[30] \mod \mathscr{U}\psL \colon X \mapsto \Hom_{\mathscr{D}_{\Lambda}}(\mathtt{G}-, X), \]
where $\mathtt{G}$ is the projection functor $\mathscr{U} \to \mathscr{U}^{\geqslant 0}$. (Note that $\add \Lambda \subseteq \mathscr{U}$, so the restriction functor above makes sense.)
\item The functors in (1) induce mutually inverse equivalences between $\mod \Lambda$ and
\[ \mathscr{R} := \{ X \in \mod \mathscr{U}\psL \mid \exists \text{ an exact sequence } P_{U_1} \to[30] P_{U_0} \epi[30] X \text{ with } U_i \in \add \Lambda \}. \]
\item For $i \leqslant 0$ we have $P_{\mathbb{S}_n^i \Lambda}^{\geqslant 0} \in \mathscr{R}$, and under the equivalence above $P_{\mathbb{S}_n^i \Lambda}^{\geqslant 0}$ corresponds to $\tau_n^i \Lambda$.
\end{enumerate}
\end{proposition}

\begin{proof}
(1) We first check that $\mathtt{G}_*$ is right exact: Let $M_1 \mono M_2 \epi M_3$ be a short exact sequence in $\mod \Lambda$. Then we have an exact sequence
\[ \Hom_{\mathscr{D}_{\Lambda}}(\mathtt{G} - , M_1) \to[30] \Hom_{\mathscr{D}_{\Lambda}}(\mathtt{G} - , M_2) \to[30] \Hom_{\mathscr{D}_{\Lambda}}(\mathtt{G} - , M_3) \to[30] \Hom_{\mathscr{D}_{\Lambda}}(\mathtt{G} - , M_1[1]) \]
in $\mod \mathscr{U}\psL$. By Remark~\ref{remark.pic_U} we have $\mathscr{U}\psL^{\geqslant 0} \subseteq \add \Lambda \vee \mathscr{D}_{\Lambda}^{\geqslant n}$. Hence, since $\gld \Lambda \leqslant n$, we have $\Hom_{\mathscr{D}_{\Lambda}}(\mathtt{G} - , M_1[1]) = 0$. The right exactness of the functor $\mathtt{G}_*$ follows.

Now we prove claim (1). We have to show that for $X \in \mod \Lambda$ and $Y \in \mod \mathscr{U}\psL$ we have an isomorphism
\[ \Hom_{\mathscr{U}\psL}( \Hom_{\mathscr{D}_{\Lambda}}(\mathtt{G} - , X), Y) \iso \Hom_{\Lambda}(X, Y(\Lambda)) \]
which is natural in $X$ and $Y$. By the right exactness of $\mathtt{G}_*$ is suffices to consider the case $X = \Lambda$. We have
\begin{align*}
\Hom_{\mathscr{U}\psL}( \Hom_{\mathscr{D}_{\Lambda}}(\mathtt{G} - , \Lambda), Y) & = \Hom_{\mathscr{U}\psL}( \underbrace{\Hom_{\mathscr{D}_{\Lambda}}(\mathscr{U}\psL, \Lambda)}_{= P_{\Lambda}}, Y) && \text{(since $\Lambda \in \mathscr{U}^{\leqslant 0}$)} \\
& = Y(\Lambda) && \text{(by the Yoneda Lemma)} \\
& = \Hom_{\Lambda}(\Lambda, Y(\Lambda))
\end{align*}
This completes the proof of (1).

(2) Clearly the functors induce mutually inverse equivalences
\[ \add \Lambda \arrow[30]{<->} \add P_{\Lambda}. \]
Now claim (2) follows form the fact that both functors are right exact.

(3) We start by observing that $P_{\mathbb{S}_n^i \Lambda}^{\geqslant 0} = \Hom_{\mathscr{D}_{\Lambda}}(\mathtt{G} - , \mathbb{S}_n^i \Lambda)$. By Remark~\ref{remark.pic_U} we have $\mathbb{S}_n^i \Lambda \in \mathscr{U}\psL^i \subseteq \mathscr{D}_{\Lambda}^{\leqslant 0}$. Thus we have a triangle
\[ \trunc^{< 0} (\mathbb{S}_n^i \Lambda) \to[30] \mathbb{S}_n^i \Lambda \to[30] \Ho^0(\mathbb{S}_n^i \Lambda) \to[30] (\trunc^{< 0}(\mathbb{S}_n^i \Lambda))[1] \]
in $\mathscr{D}_{\Lambda}$ induced by the standard t-structure on $\mathscr{D}$. Since $\mathscr{U}\psL^{\geqslant 0} \subseteq \add \Lambda \vee \mathscr{D}_{\Lambda}^{\geqslant n}$, and $\trunc^{< 0} (\mathbb{S}_n^i \Lambda), (\trunc^{< 0} (\mathbb{S}_n^i \Lambda))[1] \in \mathscr{D}_{\Lambda}^{< 0}$, we have
\[ \Hom_{\mathscr{D}_{\Lambda}}(\mathtt{G} - , \trunc^{< 0} (\mathbb{S}_n^i \Lambda)) = 0 = \Hom_{\mathscr{D}_{\Lambda}}(\mathtt{G} - , (\trunc^{< 0} (\mathbb{S}_n^i \Lambda))[1]), \]
and thus
\[ P_{\mathbb{S}_n^i \Lambda}^{\geqslant 0} = \Hom_{\mathscr{D}_{\Lambda}}(\mathtt{G} - , \mathbb{S}_n^i \Lambda) \iso \Hom_{\mathscr{D}_{\Lambda}}(\mathtt{G} - , \underbrace{\Ho^0(\mathbb{S}_n^i \Lambda)}_{\in \mod \Lambda}). \]
We see that $P_{\mathbb{S}_n^i \Lambda}^{\geqslant 0}$ is in the image of the functor $\mathtt{G}_*$, and thus in $\mathscr{R}$. For the final claim note that
\[ \Ho^0(\mathbb{S}_n^i \Lambda) = \tau_n^i \Lambda \]
by \ref{def.tau_n}.
\end{proof}

Now we are ready to prove the main result of this subsection.

\begin{proof}[Proof of Theorem~\ref{thm.identify_endoring}]
By Proposition~\ref{prop.adjoints} we have $\End_{\mathscr{U}}(T) = \End_{\Lambda}(T(\Lambda))$, and by Proposition~\ref{prop.adjoints}(3) we have $T(\Lambda) = \widetilde{\Lambda}$. Thus the middle equality follows.

We next show the rightmost equality. Clearly it suffices to show that $\Hom_{\Lambda}( \tau_n^i \Lambda, \Lambda) = 0$ for $i < 0$. As in the proof of Proposition~\ref{prop.adjoints}(3) we consider the triangle
\[ \trunc^{< 0} (\mathbb{S}_n^i \Lambda) \to[30] \mathbb{S}_n^i \Lambda \to[30] \underbrace{\Ho^0(\mathbb{S}_n^i \Lambda)}_{= \tau_n^i \Lambda} \to[30] (\trunc^{< 0}(\mathbb{S}_n^i \Lambda))[1]. \]
Since $\Hom_{\mathscr{D}_{\Lambda}}( \mathbb{S}_n^i \Lambda, \Lambda) = 0$ for $i < 0$, and $\Hom_{\mathscr{D}_{\Lambda}}((\trunc^{< 0}(\mathbb{S}_n^i \Lambda))[1], \Lambda) \subseteq \Hom_{\mathscr{D}_{\Lambda}}(\mathscr{D}_{\Lambda}^{\leqslant -2}, \Lambda) = 0$, the claim follows.

It remains to show the first equality of the theorem. More precisely, we have to show that no endomorphism of $T$ factors through a projective $\mathscr{U}\psL$-module. For $i < 0$ we have
\begin{align*}
\Hom_{\mathscr{U}\psL}(P_{\mathbb{S}_n^i \Lambda}^{\geqslant 0}, P_U) & = \Hom_{\Lambda}(\tau_n^i \Lambda, \underbrace{P_U(\Lambda)}_{\mathclap{\substack{= \Hom_{\mathscr{D}_{\Lambda}}(\Lambda, U) \\ = 0 \text{ for } U \in \mathscr{U}\psL^{> 0}}}}) = 0 && \text{for } U \in \mathscr{U}\psL^{> 0} \\
\Hom_{\mathscr{U}\psL}(P_{\mathbb{S}_n^i \Lambda}^{\geqslant 0}, P_U) & = \Hom_{\Lambda}(\tau_n^i \Lambda, U) = 0 && \text{for } U \in \mathscr{U}\psL^0 = \add \Lambda
\intertext{by the proof of the rightmost equality of the theorem above, and}
\Hom_{\mathscr{U}\psL}(P_U, P_{\mathbb{S}_n^i \Lambda}^{\geqslant 0}) & = P_{\mathbb{S}_n^i \Lambda}^{\geqslant 0}(U) = 0 && \text{for } U \in \mathscr{U}\psL^{< 0}.
\end{align*}
Summing up these three vanishing properties we obtain the claim.
\end{proof}

\subsubsection{Proof of Theorem~\ref{thm.equiv_via_tilting}} \label{subsub.derived_via_tilt}

For the proof of Theorem~\ref{thm.equiv_via_tilting} we will need the following two observations.

The first is Wakamatsu's Lemma for Krull-Schmidt
triangulated categories.

\begin{lemma}[{\cite[Proposition 2.3(1)]{IyYo}}] \label{lemma.wakamatsu_triang}
Let $\mathscr{T}$ be a Krull-Schmidt triangulated category and let $\mathscr{W}$ be a thick subcategory of $\mathscr{T}$. If $\mathscr{W}$ is contravariantly finite in $\mathscr{T}$, then we have a t-structure $(\mathscr{W}, \mathscr{W}^{\perp})$ in $\mathscr{T}$ for $\mathscr{W}^{\perp} := \{ T \in \mathscr{T} \mid \Hom_{\mathscr{T}}(\mathscr{W}, T) = 0 \}$.
\end{lemma}

\begin{lemma} \label{lemma.contr_fin_preserved}
Let $\mathscr{T}$ be a triangulated category and $\mathscr{X}$ and $\mathscr{Y}$
 be full subcategories of $\mathscr{T}$. If $\mathscr{X}$ and $\mathscr{Y}$ are contravariantly finite in $\mathscr{T}$, then so is $\mathscr{X} * \mathscr{Y}$.
\end{lemma}

\begin{proof}
Let $T \in \mathscr{T}$, and let $X \to T$ be right $\mathscr{X}$-approximation. Denote by $T'$ the cone of this approximation. Let $Y \to T'$ be a right $\mathscr{Y}$-approximation, with cone $T''$. We obtain the following octahedron:
\[ \begin{tikzpicture}[xscale=2,yscale=-1.2]
 \node (A) at (1,0) {$T''[-1]$};
 \node (B) at (2,0) {$T''[-1]$};
 \node (C) at (0,1) {$X$};
 \node (D) at (1,1) {$H$};
 \node (E) at (2,1) {$Y$};
 \node (F) at (3,1) {$X[1]$};
 \node (G) at (0,2) {$X$};
 \node (H) at (1,2) {$T$};
 \node (I) at (2,2) {$T'$};
 \node (J) at (3,2) {$X[1]$};
 \node (K) at (1,3) {$T''$};
 \node (L) at (2,3) {$T''$};
 \draw [double distance=1.5pt] (A) -- (B);
 \draw [->] (C) -- (D);
 \draw [->] (D) -- (E);
 \draw [->] (E) -- (F);
 \draw [->] (G) -- (H);
 \draw [->] (H) -- (I);
 \draw [->] (I) -- (J);
 \draw [double distance=1.5pt] (K) -- (L);
 \draw [double distance=1.5pt] (C) -- (G);
 \draw [->] (A) -- (D);
 \draw [->] (D) -- (H);
 \draw [->] (H) -- (K);
 \draw [->] (B) -- (E);
 \draw [->] (E) -- (I);
 \draw [->] (I) -- (L);
 \draw [double distance=1.5pt] (F) -- (J);
\end{tikzpicture} \]
Since the map $T \to T'$ is a contravariant $\mathscr{X}$-ghost, and the map $T' \to T''$ is a contravariant $\mathscr{Y}$-ghost it follows from the ghost lemma that the composition $T \to T''$ is a contravariant $\mathscr{X} * \mathscr{Y}$-ghost. Since $H \in \mathscr{X} * \mathscr{Y}$ it follows that $H \to T$ is a right $\mathscr{X} * \mathscr{Y}$-approximation.
\end{proof}

We are now ready to complete the proof of Theorem~\ref{thm.equiv_via_tilting}.

\begin{proof}[Proof of Theorem~\ref{thm.equiv_via_tilting}]
The first claim is part of the statement of Theorem~\ref{thm.identify_endoring}.

For the second statement we denote by $\mathtt{F}$ the composition
\[ \mathtt{F} \colon \mathscr{D}_{\Gamma} \tol[40]{\scriptstyle T \otimes_{\Gamma}^L -} \mathscr{D}_{\mathscr{U}\psL} \tol[40]{\scriptstyle \text{proj}} \stabCM(\mathscr{U}\psL). \]
Then $\mathtt{F}$ induces an equivalence $\perf \Gamma \to \mathscr{W} := \thick(T)$ by Proposition~\ref{prop.is_tilting_object}. Since by Proposition~\ref{prop.fingld} we have $\gld \Gamma \leqslant n+1$ it follows that $\mathscr{D}_{\Gamma} = \perf \Gamma$. It remains to prove $\stabCM(\mathscr{U}\psL) = \mathscr{W}$.

Let $\mathscr{X} := \add \{ \Gamma[i] | i \in \mathbb{Z} \} \subseteq \mathscr{D}_{\Gamma}$, and $\mathscr{Y} := \add \{ T[i] | i \in \mathbb{Z} \}\subseteq \stabCM(\mathscr{U}\psL)$. Since $\gld \Gamma \leqslant n+1$ we have $\mathscr{D}_{\Gamma} = \mathscr{X} * \cdots * \mathscr{X}$ ($n+2$ factors) by \cite[Theorem~5.5]{ABIM}. Hence, by the equivalence $\mathscr{D}_{\Gamma} \to \mathscr{W}$ we also have $\mathscr{W} = \mathscr{Y} * \cdots * \mathscr{Y}$ ($n+2$ factors).

Since $\stabCM(\mathscr{U}\psL)$ is $\Hom$-finite, and since for any $M \in \stabCM(\mathscr{U}\psL)$ the spaces $\underline{\Hom}_{\mathscr{U}\psL}(T, M[i])$ are non-zero for only finitely many $i$, the subcategory $\mathscr{Y} \subseteq \stabCM(\mathscr{U}\psL)$ is contravariantly finite in $\stabCM(\mathscr{U}\psL)$. By Lemma~\ref{lemma.contr_fin_preserved} we have that $\mathscr{W}$ is also contravariantly finite in $\stabCM(\mathscr{U}\psL)$. By Lemma~\ref{lemma.wakamatsu_triang} we have a t-structure $(\mathscr{W}, \mathscr{W}^{\perp})$ in $\stabCM(\mathscr{U}\psL)$, with $\mathscr{W}^{\perp} = \{M \in \stabCM(\mathscr{U}\psL) \mid \underline{\Hom}_{\mathscr{U}\psL}(\mathscr{W}, M) = 0\}$. By Proposition~\ref{prop.is_tilting_object} we have $\mathscr{W}^{\perp} = 0$. Thus we have $\stabCM(\mathscr{U}\psL) = \mathscr{W}$.
\end{proof}

\subsubsection{The equivalence via restriction}

\begin{construction} \label{const.for_res}
We denote by $\mathtt{r}$ the functor
\begin{align*}
\mathtt{r} \colon \mathscr{U}\psL & \to[30] \add_{\mathscr{U}\psL} T \approx \proj \Gamma \\
\underset{\mathclap{\text{indec.}}}{U} & \mapsto[30] \left\{ \begin{array}{ll} P_U^{\geqslant 0} & \text{ if } U \in \mathscr{U}^{< 0} \\ 0 & \text{ if } U \in \mathscr{U}^{\geqslant 0} \end{array} \right.
\end{align*}
We denote by $\mathtt{R} \colon \mod \Gamma \to \mod \mathscr{U}\psL$ the restriction functor along $\mathtt{r}$. Clearly $\mathtt{R}$ is exact, so it induces a functor $\mathscr{D}_{\Gamma} \to \mathscr{D}_{\mathscr{U}\psL}$, which will also be called $\mathtt{R}$.
\end{construction}

The aim of this subsection is to prove the following.

\begin{theorem} \label{thm.equiv_via_res}
The composition of $\mathtt{R} \colon \mathscr{D}_{\Gamma} \to \mathscr{D}_{\mathscr{U}\psL}$ with projection $\mathscr{D}_{\mathscr{U}\psL} \to \stabCM(\mathscr{U}\psL)$ is a triangle equivalence.
\end{theorem}

We will actually show the following more precise statement, relating this equivalence to the one induced by $T$.

\begin{proposition} \label{prop.equivalences_same}
In the setup above both compositions in the following diagram are naturally isomorphic, where both right functors are natural projection.
\[ \begin{tikzpicture}[xscale=3,yscale=.5]
 \node (A) at (0,0) {$\mathscr{D}_{\Gamma}$};
 \node (B1) at (1,1) {$\mathscr{D}_{\mathscr{U}\psL}$};
 \node (B2) at (1,-1) {$\mathscr{D}_{\mathscr{U}\psL}$};
 \node (C) at (2,0) {$\stabCM(\mathscr{U}\psL)$};
 \draw [->] (A) -- node [above left] {$\scriptstyle T[1] \otimes_{\Gamma}^L -$} (B1);
 \draw [->] (A) -- node [below left] {$\mathtt{R}$} (B2);
 \draw [->>] (B1) -- (C);
 \draw [->>] (B2) -- (C);
\end{tikzpicture} \]
\end{proposition}

For the proof we will need the following construction.

\begin{construction}
We denote by $\sigma = P_{-}^{\geqslant 0}$ the functor
\[ \mathscr{U}\psL \to[30] \mod \mathscr{U}\psL \colon U \mapsto[30] P_U^{\geqslant 0}. \]
For $X \in \mod \mathscr{U}$ we denote by $\sigma^* M = \Hom_{\mathscr{U}\psL}(P_-^{\geqslant 0}, M)$ the module given by
\[ \mathscr{U}\psL^{\op} \to[30] \mod k \colon U \mapsto[30] \Hom_{\mathscr{U}\psL}(P_U^{\geqslant 0}, M). \]
We observe that $\sigma^* M$ is a $\mathscr{U}\psL \otimes_k \End_{\mathscr{U}\psL}(M)^{\op}$-module.
\end{construction}

We have the following observations relating $\sigma^*$ to the claim of Proposition~\ref{prop.equivalences_same}.

\begin{lemma} \label{lemma.same_equiv_t}
We have $(\sigma^* T)^{\geqslant 0} \iso T$ as $\mathscr{U}\psL \otimes_k \Gamma^{\op}$-modules.
\end{lemma}

\begin{proof}
We have
\begin{align*}
(\sigma^* T)^{\geqslant 0}(\underset{\mathclap{\text{indec.}}}{U}) & = \left\{ \begin{array}{ll} (\sigma^* T)(U) & \text{if } U \in \mathscr{U}\psL^{\geqslant 0} \\ 0 & \text{if } U \in \mathscr{U}\psL^{< 0} \end{array} \right. && (\text{by definition of } ?^{\geqslant 0}) \\
& = \left\{ \begin{array}{ll} \Hom_{\mathscr{U}}(\underbrace{P_U^{\geqslant 0}}_{= P_U}, T) & \text{if } U \in \mathscr{U}\psL^{\geqslant 0} \\ 0 & \text{if } U \in \mathscr{U}\psL^{< 0} \end{array} \right. && (\text{by definition of } \sigma^*) \\
& = T(U) && (\text{since } T \in \mod \mathscr{U}^{\geqslant 0}) \qedhere
\end{align*}
\end{proof}

\begin{lemma} \label{lemma.same_equiv_res}
We have $(\sigma^* T)^{< 0} \iso \mathtt{R}(\Gamma)$ as $\mathscr{U}\psL \otimes_k \Gamma^{\op}$-modules.
\end{lemma}

\begin{proof}
Similar to the proof of the previous lemma we have
\begin{align*}
(\sigma^* T)^{< 0}(\underset{\mathclap{\text{indec.}}}{U}) & = \left\{ \begin{array}{ll} (\sigma^* T)(U) & \text{if } U \in \mathscr{U}\psL^{< 0} \\ 0 & \text{if } U \in \mathscr{U}\psL^{\geqslant 0} \end{array} \right. && (\text{by definition of } ?^{< 0}) \\
& = \left\{ \begin{array}{ll} \Hom_{\mathscr{U}\psL}(P_U^{\geqslant 0}, T) & \text{if } U \in \mathscr{U}\psL^{< 0} \\ 0 & \text{if } U \in \mathscr{U}\psL^{\geqslant 0} \end{array} \right. && (\text{by definition of } \sigma^*)
\end{align*}
Now note that
\[ \mathtt{R}(\Gamma)(\underset{\mathclap{\text{indec.}}}{U}) = \Hom_{\mathscr{U}\psL}(\mathtt{r}(U), T) = \left\{ \begin{array}{ll} \Hom_{\mathscr{U}\psL}(P_U^{\geqslant 0}, T) & \text{if } U \in \mathscr{U}\psL^{< 0} \\ 0 & \text{if } U \in \mathscr{U}\psL^{\geqslant 0} \end{array} \right. \qedhere \]
\end{proof}

\begin{proof}[Proof of Proposition~\ref{prop.equivalences_same}]
It suffices to show that there is a triangle $T \to X \to \mathtt{R}(\Gamma) \to T[1]$ in $\mathscr{D}_{\mathscr{U}\psL \otimes_k \Gamma^{\op}}$ such that $X$ is mapped to $\perf \mathscr{U}\psL$ by the forgetful functor $\mathscr{D}_{\mathscr{U}\psL \otimes_k \Gamma^{\op}} \to \mathscr{D}_{\mathscr{U}\psL}$.

Clearly the short exact sequence
\[ \underbrace{(\sigma^* T)^{\geqslant 0}}_{= T} \mono[30] \sigma^* T \epi[30] \underbrace{(\sigma^* T)^{< 0}}_{= \mathtt{R}(\Gamma)} \]
of $\mathscr{U}\psL \otimes_k \Gamma^{\op}$-modules gives rise to a triangle in $\mathscr{D}_{\mathscr{U}\psL \otimes_k \Gamma^{\op}}$ (here the equalities in the subscript hold by Lemmas~\ref{lemma.same_equiv_t} and \ref{lemma.same_equiv_res}, respectively). It remains to see that $\sigma^* T \in \perf \mathscr{U}\psL$. Indeed we have
\begin{align*}
\sigma^* T & = \Hom_{\mathscr{U}\psL}(P_-^{\geqslant 0}, T) \\
& = \Hom_{\mathscr{U}\psL}(P_-^{\geqslant 0}, P_V) && (\text{since } T = P_V^{\geqslant 0}) \\
& = R \Hom_{\mathscr{U}\psL}(P_-^{\geqslant 0}, P_V) && (\text{since } P_U^{\geqslant 0} \in \CM(\mathscr{U}\psL) \, \forall U) \\
& = D R \Hom_{\mathscr{U}\psL}(\leftsub{\mathscr{U}\psL}{\mathbb{S}}^- P_V, P_-^{\geqslant 0})
\intertext{Now note that since, by Lemma~\ref{lem.selinj_dim_1}, we have $P_V \in \thick(\inj \mathscr{U}\psL)$, and hence $\leftsub{\mathscr{U}\psL}{\mathbb{S}}^- P_V \in \perf \mathscr{U}\psL$, and}
\sigma^* T & \in \thick \{ D R \Hom_{\mathscr{U}\psL}(P_U, P_-^{\geqslant 0}) \mid U \in \mathscr{U}\psL \}.
\end{align*}
Note also that for $U \in \mathscr{U}\psL^{< 0}$ we have $R \Hom_{\mathscr{U}\psL}(P_U, P_-^{\geqslant 0}) = 0$. For $U \in \mathscr{U}\psL^{\geqslant 0}$ we have
\begin{align*}
D R \Hom_{\mathscr{U}\psL}(P_U, P_-^{\geqslant 0}) & = D \Hom_{\mathscr{U}\psL}(P_U, P_-^{\geqslant 0}) \\
& = D \Hom_{\mathscr{U}\psL}(P_U, P_-) \\
& = D \Hom_{\mathscr{D}_{\Lambda}}(U, -) \\
& \in \inj \mathscr{U}\psL \\
& \subseteq \perf \mathscr{U}\psL && (\text{by Lemma~\ref{lem.selinj_dim_1}}) \qedhere
\end{align*}
\end{proof}

\subsection{The \texorpdfstring{$(n+1)$}{(n+1)}-Amiot cluster category and \texorpdfstring{$\stabCM(\widetilde{\Lambda})$}{\_CM\_ Lambda\~{}} }

The construction in this subsection mostly works as in the self-injective case in Section~\ref{subsec.amiot_cluster_selfinj}.

Here we show that the stable category of Cohen-Macaulay $\widetilde{\Lambda}$ modules may be identified with the $(n+1)$-Amiot cluster category of $\Gamma$ as in the following diagram.
\[ \begin{tikzpicture}[xscale=3,yscale=-1.5]
 \node (A) at (0,0) {$\mathscr{D}_{\Gamma}$};
 \node (B) at (1,0) {$\stabCM(\mathscr{U}\psL)$};
 \node (C) at (0,1) {$\mathscr{C}_{\Gamma}^{n+1}$};
 \node (D) at (1,1) {$\stabCM(\widetilde{\Lambda})$};
 \draw (A) -- node [above] {$\approx$} (B);
 \draw (C) -- node [above] {$\approx$} (D);
 \draw [->] (A) -- node [left] {$\pi$} (C);
 \draw [->] (B) -- node [right] {push-down} (D);
\end{tikzpicture} \]

As in Section~\ref{subsec.amiot_cluster_selfinj}, $\widetilde{\Lambda}_P$ is an ideal of $\widetilde{\Lambda}$, hence in particular a $\widetilde{\Lambda} \otimes_k \widetilde{\Lambda}^{\op}$-module. Since $\Lambda$ is a subalgebra of $\widetilde{\Lambda}$, the right action of $\widetilde{\Lambda}$ on $\widetilde{\Lambda}_P$ gives a $k$-algebra homomorphism
\[ \widetilde{\Lambda} \to[30] \End_{\Lambda}( \widetilde{\Lambda}_P ) = \Gamma, \]
and hence a functor
\[ \proj \widetilde{\Lambda} \to[30] \proj \Gamma.\]
We denote by $\mathtt{A} \colon \mathscr{D}_{\Gamma} \to \mathscr{D}_{\widetilde{\Lambda}}$ the induced restriction functor.

\begin{lemma}
We have the following commutative diagram:
\[ \begin{tikzpicture}[xscale=3,yscale=1.5]
\node (A) at (0,1) {$\mathscr{D}_{\Gamma}$};
\node (B) at (1,1) {$\mathscr{D}_{\mathscr{U}\psL}$};
\node (C) at (1,0) {$\mathscr{D}_{\widetilde{\Lambda}}$};
\draw [->] (A) -- node [above] {$\mathtt{R}$} (B);
\draw [->] (B) -- node [right] {push-down} (C);
\draw [->] (A) -- node [below left] {$\mathtt{A}$} (C);
\end{tikzpicture} \]
\end{lemma}

\begin{proof}
The proof of Lemma~\ref{lemma.comm_res_pd} carries over.
\end{proof}

We now have the following commutative diagram.
\begin{equation} \label{diagram.functors_non-selfinj}
\begin{tikzpicture}[baseline=-.75cm,xscale=4,yscale=2]
 \node (A) at (0,0) {$\mathscr{D}_{\Gamma}$};
 \node (B) at (1,0) {$\mathscr{D}_{\mathscr{U}\psL}$};
 \node (C) at (2,0) {$\stabCM(\mathscr{U}\psL)$};
 \node (D) at (1,-1) {$\mathscr{D}_{\widetilde{\Lambda}}$};
 \node (E) at (2,-1) {$\stabCM(\widetilde{\Lambda})$};
 \node (F) at (0,-1) {$\mathscr{C}_{\Gamma}^{n+1}$};
 \draw [->] (A) -- node [above,pos=.7] {$\mathtt{R}$} (B);
 \draw [->,out=45,in=135] (A) to node [pos=.7,right=15pt] {$\approx$} (C);
 \draw [->>] (B) -- (C);
 \draw [->>] (D) -- (E);
 \draw [->] (B) -- node [fill=white,inner sep=1pt] {push-down} (D);
 \draw [->] (C) -- node [fill=white,inner sep=1pt] {push-down} (E);
 \draw [->] (A) -- node [below left] {$\mathtt{A}$} (D);
 \draw [->] (A) -- node [left] {$\pi$} (F);
\end{tikzpicture}
\end{equation}

We will find a triangle functor $\mathtt{H} \colon \mathscr{C}_{\Gamma}^{n+1} \to \stabCM \widetilde{\Lambda}$ making diagram \eqref{diagram.functors_non-selfinj} commutative. We need the following observation, which is shown by the proof of Proposition~\ref{prop.amiot.factorization.conditions} without any changes.

\begin{proposition} \label{prop.amiot_factor_cond_non-selfinj}
In the setup above there is a triangle
\[ X \to[30] \mathtt{A}(D \Gamma[-n-1]) \to[30] \mathtt{A}(\Gamma) \to[30] X[1] \]
in $\mathscr{D}_{\widetilde{\Lambda} \otimes_k \Gamma^{\op}}$, such that $X$ is mapped to $\perf \widetilde{\Lambda}$ by the forgetful functor $\mathscr{D}_{\widetilde{\Lambda} \otimes_k \Gamma^{\op}} \to \mathscr{D}_{\widetilde{\Lambda}}$.
\end{proposition}

Using the universality of the $(n+1)$-Amiot cluster category, we have the following consequences. See Appendix~\ref{appendix} for details on the proof.

\begin{proposition}
There is a triangle functor $\mathtt{H} \colon \mathscr{C}_{\Gamma}^{n+1} \to \stabCM(\widetilde{\Lambda})$ making Diagram~\eqref{diagram.functors_non-selfinj} commutative.
\end{proposition}

Now we have the following result.

\begin{theorem} \label{theorem.H_equiv_CM}
The functor $\mathtt{H} \colon \mathscr{C}_{\Gamma}^{n+1} \to \stabCM(\widetilde{\Lambda})$ is an equivalence.
\end{theorem}

\begin{proof}
The proof of Theorem~\ref{theorem.H_equiv_selfinj} carries over word for word (with the reference to Theorem~\ref{theorem.serrestable} replaced by a reference to Theorem~\ref{theorem.serre_CM}).
\end{proof}

\begin{corollary} \label{cor.CMunderline_cto}
Let $\Lambda$ be an algebra of global dimension at most $n$, such that the $n$-Amiot cluster category $\mathscr{C}_{\Lambda}^n$ is $\Hom$-finite. Assume further that the category $\mathscr{U}\psL$ (see Theorem~\ref{theorem.ct_in_Db}) has the vosnex property. Then the category $\stabCM(\widetilde{\Lambda})$ is $(n+1)$-Calabi-Yau triangulated with an $(n+1)$-cluster tilting object.
\end{corollary}

\begin{proof}
This follows immediately from Theorems~\ref{theorem.H_equiv_CM} and \ref{theorem.clairemain2}.
\end{proof}

\begin{example}
Let $\Lambda$ be the Auslander algebra of non-linearly oriented $A_3$ as in the left picture below, and $n=2$.
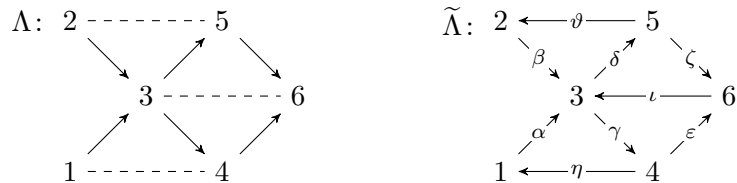
\[ \begin{tikzpicture}
 \node at (-.5,2) {$\Lambda \colon$};
 \node (1) at (0,0) {$1$};
 \node (2) at (0,2) {$2$};
 \node (3) at (1,1) {$3$};
 \node (4) at (2,0) {$4$};
 \node (5) at (2,2) {$5$};
 \node (6) at (3,1) {$6$};
 \draw [->] (1) -- (3);
 \draw [->] (2) -- (3);
 \draw [->] (3) -- (4);
 \draw [->] (3) -- (5);
 \draw [->] (4) -- (6);
 \draw [->] (5) -- (6);
 \draw [dashed] (1) -- (4);
 \draw [dashed] (2) -- (5);
 \draw [dashed] (3) -- (6);
\end{tikzpicture}
\qquad \qquad
\begin{tikzpicture}
 \node at (-.5,2) {$\widetilde{\Lambda} \colon$};
 \node (1) at (0,0) {$1$};
 \node (2) at (0,2) {$2$};
 \node (3) at (1,1) {$3$};
 \node (4) at (2,0) {$4$};
 \node (5) at (2,2) {$5$};
 \node (6) at (3,1) {$6$};
 \draw [->] (1) -- node [inner sep=1pt,fill=white] {$\scriptstyle \alpha$} (3);
 \draw [->] (2) -- node [inner sep=1pt,fill=white] {$\scriptstyle \beta$} (3);
 \draw [->] (3) -- node [inner sep=1pt,fill=white] {$\scriptstyle \gamma$} (4);
 \draw [->] (3) -- node [inner sep=1pt,fill=white] {$\scriptstyle \delta$} (5);
 \draw [->] (4) -- node [inner sep=1pt,fill=white] {$\scriptstyle \varepsilon$} (6);
 \draw [->] (5) -- node [inner sep=1pt,fill=white] {$\scriptstyle \zeta$} (6);
 \draw [<-] (1) -- node [inner sep=1pt,fill=white] {$\scriptstyle \eta$} (4);
 \draw [<-] (2) -- node [inner sep=1pt,fill=white] {$\scriptstyle \vartheta$} (5);
 \draw [<-] (3) -- node [inner sep=1pt,fill=white] {$\scriptstyle \iota$} (6);
\end{tikzpicture} \]
Then $\widetilde{\Lambda}$ is given by the right quiver above, subject to the relations
\[ (\gamma \eta, \delta \vartheta, \eta \alpha - \varepsilon \iota, \vartheta \beta - \zeta \iota, \iota \gamma, \iota \delta, \alpha \gamma, \beta \delta, \gamma \varepsilon - \delta \zeta ). \]
Calculating the $\tau_2^-$ powers of the projective $\Lambda$ modules we obtain
\begin{align*}
\Lambda \colon \quad & 1 && 2 && \begin{matrix} 3 \\ 1 \quad 2 \end{matrix} && \begin{matrix} 4 \\ 3 \\ 2 \end{matrix} && \begin{matrix} 5 \\ 3 \\ 1 \end{matrix} && \begin{matrix} 6 \\ 4 \quad 5 \\ 3 \end{matrix} \\
\tau_2^- \Lambda \colon \quad & 4 && 5 && \begin{matrix} 6 \\ 4 \quad 5 \end{matrix} && - && - && -
\end{align*}
Thus $\widetilde{\Lambda}_P = \bigoplus_{i > 0} \tau_2^{-i} \Lambda = 4 \oplus 5 \oplus \begin{smallmatrix} 6 \\ 4 \quad 5 \end{smallmatrix}$, and $\End_{\Lambda}(\widetilde{\Lambda}_P) = k[\circ \to \circ \arrow{<-} \circ]$. Thus, by Theorem~\ref{theorem.H_equiv_CM} the category $\stabCM(\widetilde{\Lambda})$ is the classical $3$-cluster category of $A_3$.
\end{example}

\renewcommand{\thesection}{A}

\section{Appendix: The universal property of the \texorpdfstring{$n$}{n}-Amiot cluster category} \label{appendix}

In this appendix we work out in detail how Keller's general result \cite[Theorem 4]{K_orbit} applies to our setup. Everything stated here is due to Keller, and can also be found in \cite{K_orbit} (also see \cite{DGcat, K_cyclic, K_onDG} for background). To avoid set theoretical issues we assume all categories to be small throughout this appendix.

\subsection{Pretriangulated DG categories}

See \cite{K_cyclic}, in particular Section~2 (note that in that paper pretriangulated DG categories are called exact DG categories).

\begin{definition}
A \emph{DG category} is a $\mathbb{Z}$-graded category (i.e.\ morphism spaces are $\mathbb{Z}$-graded, and composition of morphisms respects this grading) with a differential $d$ of degree $1$ satisfying the Leibniz rule.

For a DG category $\mathscr{X}$ we denote  by $\Ho^0(\mathscr{X})$ the category with the same objects as $\mathscr{X}$, and with $\Hom_{\Ho^0(\mathscr{X})}(X_1, X_2) = \Ho^0(\Hom_{\mathscr{X}}(X_1, X_2))$.
\end{definition}

\begin{example} \label{ex.complexes1}
Let $\mathscr{A}$ be an additive category. Then $C_{\rm dg}(\mathscr{A})$, the set of complexes over $\mathscr{A}$, becomes a DG category by setting
\begin{align*}
& \HOM_{\mathscr{A}}^n (A, B) = \prod_{i \in \mathbb{Z}} \Hom_{\mathscr{A}}(A^i, B^{i+n}) \text{, and} \\
& d((f_i)_{i \in \mathbb{Z}}) = ( f_i d_B - (-1)^n d_A f_{i+1} ) \qquad \text{for } (f_i)_{i \in \mathbb{Z}} \in \HOM_{\mathscr{A}}^n (A, B).
\end{align*}
To calculate morphisms in $\Ho^0(C_{\rm dg}(\mathscr{A}))$ note that for $A, B \in C_{\rm dg}(\mathscr{A})$ we have
\begin{align*}
{\rm Z}^0(\HOM_{\mathscr{A}}(A, B)) & = \{ (f_i)_{i \in \mathbb{Z}} \in \prod_{i \in \mathbb{Z}} \Hom_{\mathscr{A}}(A^i, B^i) \mid d((f_i)_{i \in \mathbb{Z}}) = 0 \} \\
& = \{ (f_i)_{i \in \mathbb{Z}} \mid d_A f_{i+1} = f_i d_B \} \\
& = \text{usual maps of complexes} \\
{\rm B}^0(\HOM_{\mathscr{A}}(A, B)) & = \{ d((h_i)_{i \in \mathbb{Z}}) \mid h_i \in \Hom_{\mathscr{A}}(A^i, B^{i-1}) \} \\
& = (h_i d_B + d_A h_{i+1})_{i \in \mathbb{Z}} \mid h_i \in \Hom_{\mathscr{A}}(A^i, B^{i-1}) \} \\
& = \text{homotopies of complexes}
\intertext{and hence}
\Ho^0(\HOM_{\mathscr{A}}(A, B)) & = \text{maps of complexes up to homotopy.}
\end{align*}
\end{example}

\begin{definition}[see \cite{K_onDG}, Sections~2.3 and 3]
Let $\mathscr{X}$ be a DG category. A \emph{DG $\mathscr{X}$-module} is a DG functor $\mathscr{X}^{\op} \to C_{\rm dg}(\Mod k)$. By abuse of notation we also denote the DG category of DG $\mathscr{X}$-modules by $C_{\rm dg}(\mathscr{X})$.

We denote the (non DG) category of DG $\mathscr{X}$-modules by $C(\mathscr{X}) = Z^0(C_{\rm dg}(\mathscr{X}))$, and we denote by $D(\mathscr{X})$ the derived category of DG $\mathscr{X}$-modules, that is the category obtained from $C(\mathscr{X})$ by inverting quasi-isomorphisms.

A DG $\mathscr{X}$-module is \emph{representable}, if it is isomorphic to a DG $\mathscr{X}$-module of the form $\HOM_{\mathscr{X}}(-, X)$ for some object $X \in \mathscr{X}$.
\end{definition}

\begin{definition}
A DG category $\mathscr{X}$ is called \emph{pretriangulated} (or \emph{exact}), if the class of representable DG $\mathscr{X}$-modules is closed under translations and mapping cones.
\end{definition}

Note that by the Yoneda Lemma we have the subcategory of representable DG $\mathscr{X}$-modules is equivalent to $\mathscr{X}$, for any DG category $\mathscr{X}$.

\begin{construction}[{Keller \cite[2.2(d)]{K_cyclic}}]
Let $\mathscr{X}$ be a DG category. We denote by $\pretr \mathscr{X}$ the \emph{pretriangulated hull of $\mathscr{X}$}, that is the smallest subcategory of $C_{\rm dg}(\mathscr{X})$ which contains the representable DG $\mathscr{X}$-modules, and is closed under translations and mapping cones.
\end{construction}

The universal property of the $n$-Amiot cluster categories builds on the following result of Keller.

\begin{proposition}[{Keller \cite[2.2(d)]{K_cyclic}}]
Let $\mathscr{X}$ be a DG category. Then
\begin{enumerate}
\item $\pretr \mathscr{X}$ is a pretriangulated DG category, and
\item the natural functor $\mathscr{X} \to \pretr \mathscr{X}$ is universal among DG functors from $\mathscr{X}$ to pretriangulated DG categories.
\end{enumerate}
\end{proposition}

The reason for calling these DG algebras pretriangulated, and the connection to our setup, is the following. 

\begin{proposition}[{Keller \cite{K_cyclic}}]
Let $\mathscr{X}$ be a pretriangulated DG category. Then $\Ho^0(\mathscr{X})$ is an algebraic triangulated category. Moreover any algebraic triangulated category comes up in this construction.
\end{proposition}

\begin{example}[{\cite[2.2(a)]{K_cyclic}}] \label{ex.subcat_complex}
Let $\mathscr{A}$ be an additive category, $\mathscr{X}$ a full DG subcategory of $C_{\rm dg}(\mathscr{A})$ (see Example~\ref{ex.complexes1} above) which is closed under translations and mapping cones. Then $\mathscr{X}$ is a pretriangulated DG category, and (by the calculation in Example~\ref{ex.complexes1}) the triangulated category $\Ho^0(\mathscr{X})$ is the corresponding subcategory of the homotopy category of complexes in $\mathscr{A}$.
\end{example}

\begin{notation}
We will mostly need the following instances of Example~\ref{ex.subcat_complex} above, for which we therefore introduce special names.
\begin{enumerate}
\item Let $\Delta$ be a finite dimensional algebra. We denote by
\[ \mathscr{P}_{\Delta} := C_{\rm dg}^{-, \rm b}(\proj \Delta) \]
the DG category of right bounded complexes of finitely generated projective $\Delta$-modules with bounded homology.

Then $\Ho^0(\mathscr{P}_{\Delta}) = \mathscr{D}_{\Delta}$, the bounded derived category of $\Delta$-modules.
\item Let $\Delta$ be a self-injective finite dimensional algebra (resp.\ an Iwanaga-Gorenstein finite dimensional algebra). We denote by
\[ \mathscr{A}_{\Delta} := C_{\rm dg}^{\infty, \emptyset}(\proj \Delta) \]
the DG category of acyclic complexes of finitely generated projective $\Delta$-modules.

Then $\Ho^0(\mathscr{A}_{\Delta}) = \stabmod \Delta$ the stable module category of $\Delta$ (resp.\ $\Ho^0(\mathscr{A}_{\Delta}) = \stabCM(\Delta)$ the stable category of Cohen-Macaulay $\Delta$-modules).
\end{enumerate}
\end{notation}

\subsection{Functors}

We now look at how certain functors between triangulated categories are reflected in the setup of pretriangulated DG categories. It turns out (see \cite[Section~9]{K_orbit}, in particular Section~9.6) that we should consider the category $\rep$ (see Definition~\ref{def.rep} below).

For the rest of this appendix all undecorated tensor products are understood to be tensor products over the base field $k$. Note that for two DG categories $\mathscr{X}$ and $\mathscr{Y}$, the tensor product $\mathscr{X} \otimes \mathscr{Y}$ is a DG category in a natural way (see for instance \cite[2.3]{K_onDG}).

\begin{definition}[{\cite[Section~9.2]{K_orbit}}] \label{def.rep}
Let $\mathscr{X}$ and $\mathscr{Y}$ be DG categories. We denote by $\rep(\mathscr{X}, \mathscr{Y})$. the full subcategory of $D(\mathscr{X}^{\op} \otimes \mathscr{Y})$ formed by the objects $\mathtt{R}$, such for all $X \in \mathscr{X}$ the object $\mathtt{R}(X \otimes -)$ is isomorphic to a representable DG $\mathscr{Y}$-module in $D(\mathscr{Y})$.
\end{definition}

The motivating example for this construction is the following:

\begin{example}[{\cite[Section~9.2]{K_orbit}}]
Let $\mathtt{F} \colon \mathscr{X} \to \mathscr{Y}$ be a functor of DG categories. Then $\mathtt{F}$ induces an object $\mathtt{R}_{\mathtt{F}}$ given by $\mathtt{R}_{\mathtt{F}}(X \otimes Y) = \HOM_{\mathscr{Y}}(Y, \mathtt{F} X)$ in $\mathscr{D}_{\mathscr{X}^{\op} \otimes \mathscr{Y}}$. Since $\mathtt{R}_{\mathtt{F}}(X \otimes -)$ is represented by $\mathtt{F} X$ it lies in $\rep(\mathscr{X}, \mathscr{Y})$.
\end{example}

The composition of functors in $\rep$ is given by the derived tensor product.

\begin{example}
Let $\mathtt{F} \colon \mathscr{X} \to \mathscr{Y}$ and $\mathtt{R} \in \rep(\mathscr{Y}, \mathscr{Z})$ for three DG categories $\mathscr{X}, \mathscr{Y}, \mathscr{Z}$. Then the composition $\mathtt{R}_{\mathtt{F}} \mathtt{R} \in \rep(\mathscr{X}, \mathscr{Z})$ is given by
\[ (\mathtt{R}_{\mathtt{F}} \mathtt{R}) (X \otimes Z) = \mathtt{R}(\mathtt{F} X \otimes Z). \]
In particular, for $\mathtt{R} = \mathtt{R}_{\mathtt{G}}$, we have
\[ \mathtt{R}_{\mathtt{F}} \mathtt{R}_{\mathtt{G}} = \mathtt{R}_{\mathtt{F} \mathtt{G}}. \]
\end{example}

\subsection{Orbit categories}

We give the following definition for the sake of completeness. However we do not explain it here. The reader is advised to look up Section~5.1 in \cite{K_orbit}.

\begin{definition}[{\cite[Section~5.1]{K_orbit}}] \label{def.orbitDG}
Let $\mathscr{X}$ be a DG category, $\mathtt{S} \colon \mathscr{X} \to \mathscr{X}$ a DG functor inducing an equivalence on $\Ho^0(\mathscr{X})$. Then the \emph{DG orbit category} $\mathscr{X} / \mathtt{S}$ has the same objects as $\mathscr{X}$, and
\begin{align*} & \Hom_{\mathscr{X} / \mathtt{S}}(X_1, X_2) \\
= \, & {\rm colim} \Big[ \bigoplus_{j \geqslant 0} \leftsub{\mathscr{X}}(\mathtt{S}^j X_1, X_2) \tol[30]{\scriptstyle \mathtt{S}} \bigoplus_{j \geqslant 0} \leftsub{\mathscr{X}}(\mathtt{S}^j X_1, \mathtt{S}X_2)  \tol[30]{\scriptstyle \mathtt{S}} \bigoplus_{j \geqslant 0} \leftsub{\mathscr{X}}(\mathtt{S}^j X_1, \mathtt{S}^2 X_2)  \tol[30]{\scriptstyle \mathtt{S}} \cdots \Big] .
\end{align*}
\end{definition}

\begin{definition}[{\cite[Section~5.1]{K_orbit}}]
Let $\mathscr{T} = \Ho^0(\mathscr{X})$ be an algebraic triangulated category, $\mathtt{S} \colon \mathscr{X} \to \mathscr{X}$ a DG functor inducing an equivalence on $\mathscr{T}$. Then the \emph{triangulated orbit category} of $\mathscr{T}$ modulo $\mathtt{S}$ is defined to be
\[ \Ho^0(\pretr (\mathscr{X} / \mathtt{S})). \]
\end{definition}

The instance of this definition we need here is the following. Here and in the sequel, for a modules $X$ over a finite dimensional algebra $\Delta$ we denote by $\mathtt{p}_{\Delta}(X)$ a projective resolution of $X$.

\begin{proposition}[{Keller, a special case of \cite[Theorem~7.1]{K_orbit}}] \label{prop.dg_for_amiot}
Let $\Delta$ be a finite dimensional algebra of finite global dimension, and $n \in \mathbb{N}$. Then the functor
\[ \mathtt{S} := \mathtt{p}_{\Delta \otimes \Delta^{\op}}(D \Delta)[-n] \otimes_{\Delta} - \colon \mathscr{P}_{\Delta} \to[30] \mathscr{P}_{\Delta} \]
induces the autoequivalence $\mathbb{S}_n$ of $\mathscr{D}_{\Delta}$. The $n$-Amiot cluster category (defined in Construction~\ref{const.n-Amiot}) is equivalent to the triangulated orbit category
\[ \mathscr{C}_{\Delta}^n \approx \Ho^0(\pretr (\mathscr{P}_{\Delta} / \mathtt{S} )). \]
\end{proposition}

Before we can formulate the universal property of orbit categories in the full generality of \cite{K_orbit} we need one more piece of notation.

\begin{definition}[{\cite[Section~9.3]{K_orbit}}]
Let $\mathscr{X}$ be a DG category and $\mathtt{S} \colon \mathscr{X} \to \mathscr{X}$ a cofibrant DG module in $\rep(\mathscr{X}, \mathscr{X})$, and let $\mathscr{Y}$ be a DG category. Then we denote by $\eff(\mathscr{X}, \mathtt{S}, \mathscr{Y})$ (from effa\c{c}able -- removing the effect of $\mathtt{S}$) the category with objects
\[ (\mathtt{R}, \phi) \text{ with } \mathtt{R} \in \rep(\mathscr{X}, \mathscr{Y}) \text{ and } \phi \colon \mathtt{R} \to \mathtt{R}_{\mathtt{S}} \mathtt{R} \text{ an quasi-isomorphism of DG} \mathscr{X}^{\op} \otimes \mathscr{Y} \text{-modules} \]
and where the morphisms from $(\mathtt{R}, \phi)$ to $(\mathtt{R}', \phi')$ are obtained from morphisms $f \colon \mathtt{R} \to \mathtt{R}'$ of DG $\mathscr{X}^{\op} \otimes \mathscr{Y}$-modules, such that $f \phi' = \phi (\mathtt{S} f)$, by making quasi-isomorphisms invertible.

Similarly we can define the category $\eff^{\op}(\mathscr{X}, \mathtt{S}, \mathscr{Y})$ with objects $(\mathtt{R}, \phi)$, where $\phi$ is a quasi-isomorphism $\mathtt{R}_{\mathtt{S}} \mathtt{R} \to \mathtt{R}$.
\end{definition}

We are now ready to state Keller's universal property.

\begin{theorem}[{Keller \cite[Theorem~4 in Section~9.6]{K_orbit}}] \label{theo.keller_universal}
Let $\mathscr{X}$ be a pretriangulated DG category, and $\mathtt{S} \colon \mathscr{X} \to \mathscr{X}$ as in Definition~\ref{def.orbitDG}. Then for any pretriangulated DG category $\mathscr{Y}$ there are equivalences
\[ \eff(\mathscr{X}, \mathtt{S}, \mathscr{Y}) \arrowl[30]{<-}{\approx} \rep(\pretr(\mathscr{X}/\mathtt{S}), \mathscr{Y}) \tol[30]{\approx} \eff^{\op}(\mathscr{X}, \mathtt{S}, \mathscr{Y}). \]
\end{theorem}

\begin{remark}
Keller only states the first equivalence in of Theorem~\ref{theo.keller_universal}. The other equivalence follows from the first one for $\mathscr{X}^{\op}$ and $\mathscr{Y}^{\op}$. (Note that it follows from the definition that $\mathscr{X}^{\op} / \mathtt{S} = (\mathscr{X} / \mathtt{S})^{\op}$, and from the universal property of the pretriangulated hull that $\pretr (\mathscr{X}^{\op}) = (\pretr \mathscr{X})^{\op}$.)
\end{remark}

Applying this theorem to the setup of Proposition~\ref{prop.dg_for_amiot} we obtain the following.

\begin{corollary} \label{cor.keller_universal_cc}
Let $\Delta$ be a finite dimensional algebra of finite global dimension, and
\[ \mathtt{S} := \mathtt{p}_{\Delta \otimes \Delta^{\op}}(D \Delta)[-n] \otimes_{\Delta} - \colon \mathscr{P}_{\Delta} \to \mathscr{P}_{\Delta} \]
as in Proposition~\ref{prop.dg_for_amiot}. Then  for any pretriangulated DG category $\mathscr{Y}$ there are equivalences
\[ \eff(\mathscr{P}_{\Delta}, \mathtt{S}, \mathscr{Y}) \arrowl[30]{<-}{\approx} \rep(\pretr(\mathscr{P}_{\Delta} / \mathtt{S}), \mathscr{Y}) \tol[30]{\approx} \eff^{\op}(\mathscr{P}_{\Delta}, \mathtt{S}, \mathscr{Y}). \]
\end{corollary}

\subsection{Application to our setup}

We now show the following instance of Keller's general result. This is what we apply in this paper.

\begin{theorem} \label{thm.universal_stable}
Let $\Delta$ be a finite dimensional algebra of finite global dimension, and let $\Pi$ be an Iwanaga-Gorenstein finite dimensional algebra. Let $M \in \mathscr{D}_{\Pi \otimes \Delta^{\op}}$, and
\[ \mathtt{A} = M \otimes_{\Delta}^L - \colon \mathscr{D}_{\Delta} \to \mathscr{D}_{\Pi} \]
the triangle functor given by the derived tensor product. Assume there is a triangle
\[ X \to[30] \mathtt{A}(D \Delta [-n]) \to[30] \mathtt{A}(\Delta) \to[30] X[1] \]
in $\mathscr{D}_{\Pi \otimes \Delta^{\op}}$, such that $X$ is mapped to $\perf \Pi$ by the forgetful functor $\mathscr{D}_{\Pi \otimes \Delta^{\op}} \to \mathscr{D}_{\Pi}$.

Then there is a functor $\mathtt{H} \colon \mathscr{C}_{\Delta}^n \to \stabCM(\Pi)$ making the following diagram commutative.
\[ \begin{tikzpicture}[xscale=3,yscale=1.5]
 \node (A) at (0,1) {$\mathscr{D}_{\Delta}$};
 \node (B) at (1,1) {$\mathscr{D}_{\Pi}$};
 \node (C) at (0,0) {$\mathscr{C}_{\Delta}^n$};
 \node (D) at (1,0) {$\stabCM(\Pi)$};
 \draw [->] (A) -- node [above] {$\mathtt{A}$} (B);
 \draw [->] (C) -- node [above] {$\mathtt{H}$} (D);
 \draw [->] (A) -- (C);
 \draw [->] (B) -- (D);
\end{tikzpicture} \]
\end{theorem}

\begin{remark}
The same result also holds if we have a triangle ``in the other direction'', that is a triangle
\[ X \to[30] \mathtt{A}(\Delta) \to[30] \mathtt{A}(D \Delta [-n]) \to[30] X[1] \]
(and all other assumptions as in Theorem~\ref{thm.universal_stable} above).
\end{remark}

For the proof we need the following observation.

\begin{proposition}[Keller] \label{prop.rep_for_proj}
Let $\Pi$ be an Iwanaga-Gorenstein finite dimensional algebra. Setting
\[ \mathtt{G}(X \otimes Y) = \HOM_{\Pi}(Y, X) \text{ for } X \in \mathscr{P}_{\Pi} \text{ and } Y \in \mathscr{A}_{\Pi} \]
we have $\mathtt{G} \in \rep(\mathscr{P}_{\Pi}, \mathscr{A}_{\Pi})$, and $\mathtt{G}$ induces the natural functor $\mathscr{D}_{\Pi} \to \stabCM(\Pi)$.
\end{proposition}

\begin{proof}
Take $\mathtt{a} X \in \mathscr{A}_{\Pi}$ such that $\trunc^{< i} \mathtt{a} X = \trunc^{< i} X$ for $i \ll 0$. Then we have a chain homomorphism $\mathtt{a} X \to X$, and a triangle
\[ \mathtt{a} X \to[30] X \to[30] C \to[30] \mathtt{a} X[1] \]
in the homotopy category of $\proj \Pi$, with $C$ a left bounded complex of finitely generated projective $\Pi$-modules.

Since $Y$ is acyclic, so is $\HOM_{\Pi}(Y, C)$. Thus the morphism
\[ \HOM_{\Pi}(Y, \mathtt{a} X) \to[30] \HOM_{\Pi}(Y, X) = \mathtt{G}(X \otimes Y) \]
is a quasi-isomorphism. Thus $\mathtt{G}(X \otimes -)$ is representable (since it is represented by $\mathtt{a}X$), and we have $\mathtt{G} \in \rep(\mathscr{P}_{\Pi}, \mathscr{A}_{\Pi})$.
\end{proof}

\begin{proof}[Proof of Theorem~\ref{thm.universal_stable}]
Our first task is to translate the setup of the theorem to functors between the DG categories $\mathscr{P}_{\Delta}$, $\mathscr{P}_{\Pi}$, and $\mathscr{A}_{\Pi}$. We have the following functors and objects in the respective $\rep$-categories:
\begin{enumerate}
\item  The functor $\mathtt{p}_{\Pi \otimes \Delta^{\op}} (M) \otimes_{\Delta} - \colon \mathscr{P}_{\Delta} \to \mathscr{P}_{\Pi}$ induces $\mathtt{A} \colon \mathscr{D}_{\Delta} \to \mathscr{D}_{\Pi}$. We denote by $\mathtt{R}_{\mathtt{A}}$ the corresponding object in $\rep(\mathscr{P}_{\Delta}, \mathscr{P}_{\Pi})$.
\item The functor $\mathtt{S} := \mathtt{p}_{\Delta \otimes \Delta^{\op}} (D\Delta)[-n] \otimes_{\Delta} - \colon \mathscr{P}_{\Delta} \to \mathscr{P}_{\Delta}$ induces $D\Delta[-n] \otimes^L_{\Delta} - \colon \mathscr{D}_{\Delta} \to \mathscr{D}_{\Delta}$. We have the corresponding object $\mathtt{R}_{\mathtt{S}}$ in $\rep(\mathscr{P}_{\Delta}, \mathscr{P}_{\Delta})$.
\item We denote by $\mathtt{R}_X$ the object in $\rep(\mathscr{P}_{\Delta}, \mathscr{P}_{\Pi})$ corresponding to the functors $\mathtt{p}_{\Pi \otimes \Delta^{\op}}(X) \otimes_{\Delta} - \colon \mathscr{P}_{\Delta} \to \mathscr{P}_{\Pi}$.
\end{enumerate}

Now note that the triangle of the theorem induces a triangle
\[ \mathtt{R}_X \to[30] \mathtt{R}_{\mathtt{S}}\mathtt{R}_{\mathtt{A}} \tol[30]{\phi} \mathtt{R}_{\mathtt{A}} \to[30] \mathtt{R}_X[1] \]
in $\rep(\mathscr{P}_{\Delta}, \mathscr{P}_{\Pi})$.

We denote by $\mathtt{G} \in \rep(\mathscr{P}_{\Pi}, \mathscr{A}_{\Pi})$ the object constructed in Proposition~\ref{prop.rep_for_proj} above. Then the object $\mathtt{R}_{\mathtt{A}} \mathtt{G} \in \rep(\mathscr{P}_{\Delta}, \mathscr{A}_{\Pi})$ is a lift of the composition $\mathscr{D}_{\Delta} \tol{\scriptstyle \mathtt{A}} \mathscr{D}_{\Pi} \to \stabCM(\Pi)$. Multiplying the above triangle with $\mathtt{G}$ we obtain a triangle
\[ \mathtt{R}_X \mathtt{G} \to[30] \mathtt{R}_{\mathtt{S}}\mathtt{R}_{\mathtt{A}} \mathtt{G} \tol[30]{\phi \mathtt{G}} \mathtt{R}_{\mathtt{A}} \mathtt{G} \to[30] \mathtt{R}_X \mathtt{G} [1] \]
in $\rep(\mathscr{P}_{\Delta}, \mathscr{A}_{\Pi})$. We claim that $\phi \mathtt{G}$ is a quasi-isomorphism. To see this we only have to show that $\mathtt{R}_X \mathtt{G}$ is acyclic. Indeed, for any $P \in \mathscr{P}_{\Delta}$ and $A \in \mathscr{A}_{\Pi}$ we have
\begin{align*}
(\mathtt{R}_X \mathtt{G})(P \otimes A) & = \mathtt{G}(X \otimes_{\Delta} P \otimes A) \\
& = \HOM_{\Pi}(A, X \otimes_{\Delta} P),
\end{align*}
and this is acyclic, since $A$ is acyclic, and $P \in \perf \Delta$ and $X \in \perf \Pi$ imply $X \otimes_{\Delta} P \in \perf \Pi$.

Consequently $(\mathtt{R}_{\mathtt{A}} \mathtt{G}, \phi \mathtt{G})$ lies in $\eff^{\op}(\mathscr{P}_{\Delta}, \mathtt{S}, \mathscr{A}_{\Pi})$. By Corollary~\ref{cor.keller_universal_cc} $\mathtt{R}_{\mathtt{A}} \mathtt{G}$ is given by an object in $\rep(\pretr(\mathscr{P}_{\Delta} / \mathtt{S}), \mathscr{A}_{\Pi})$, which induces the desired triangle functor $\mathtt{H} \colon \mathscr{C}_{\Delta}^n \to \stabCM(\Pi)$.
\end{proof}

\newcommand{\etalchar}[1]{$^{#1}$}

\end{document}